\def\reserveinserts#1{}
\documentclass[HARVARD,LATO1COL]{WileyNJDv5}

\usepackage[utf8]{inputenc} 

\articletype{Preprint}%

\usepackage{pgfplots}
\usepackage{comment}

\usepackage{multirow}
\usepackage{multicol}

\historydates{}
\doiheadtext{}

\makeatletter
\def\oddfoot@titlepage@info{}
\def\evenfoot@titlepage@info{}
\makeatother

\begin{document}

\title{Revisiting recursive methods for Dyson and Keldysh in NEGF: Part I}

\author[1]{Edoardo di Napoli}

\author[2]{Alessandro Pecchia}

\author[1]{Gustavo Ramirez-Hidalgo$^{*}$}

\authormark{RAMIREZ-HIDALGO \textsc{et al.}}
\titlemark{Revisiting recursive methods for Dyson and Keldysh in NEGF: Part I}


\address[1]{\orgdiv{J{\"u}lich Supercomputing Centre}, \orgname{Forschungszentrum J{\"u}lich GmbH}, \orgaddress{Wilhelm-Johnen-Stra\ss{}e 52428 J\"ulich, \country{Germany}}}

\address[2]{\orgdiv{CNR-ISMN}, \orgname{Consiglio Nazionale delle Ricerche}, \orgaddress{\state{Monterotondo}, \country{Italy}}}

\corres{$^{*}$Corresponding author: Gustavo Ramirez-Hidalgo. \email{g.ramirez.hidalgo@fz-juelich.de}}


\abstract[Abstract]{The simulation of quantum transport in nanodevices requires the solution of the Dyson and Keldysh equations, a task dominated by the inversion of massive, block-tridiagonal matrices. While the Recursive Green's Function (RGF) method has long been the standard $O(N)$ solver for quasi-1D systems, its formulation has typically been restricted to sequential execution and nearest-neighbor interactions. In this work, we carefully reformulate RGF through the lens of Domain Decomposition and Schur Complement theory. This allows us to extend the recursive formalism to block $n$-diagonal systems (handling higher-order stencils) and to derive a parallel algorithm, Domain-Decomposition based RGF (DDRGF), which stitches macroscopic domains via reduced interface systems. We explore data dependencies in DDRGF in detail, by means of block-sparse structures and tracing back to the desired output as a block tridiagonal approximation, giving a clear, reproducible and extensible formulation.
We validate these algorithms using \texttt{LibNEGF.jl}, a Julia-based implementation, demonstrating that the structural insights of domain decomposition provide a robust pathway for high-performance quantum transport simulations on modern multi-core clusters. The theory presented here lays down the base for tackling the Keldysh problem, to be similarly handled in future stages of our work. Although the target here is the acceleration of kernels in the non-equilibrium Green's function method, the algorithms and the implementations presented can be immediately used in any application involving block $n$-diagonal systems.}


\jnlcitation{}


\maketitle

\renewcommand\thefootnote{}
\footnotetext{\textbf{Abbreviations:} NEGF, non-equilibrium Green's function; RGF, recursive Green's function.}

\renewcommand\thefootnote{\fnsymbol{footnote}}
\setcounter{footnote}{1}

\section{Introduction}

The Non-Equilibrium Green's Function (NEGF) formalism stands as the \textit{de facto} standard for modeling quantum transport in nanodevices, ranging from post-silicon transistors to molecular junctions \cite{Luisier2006, Datta2005}. As device scaling reaches the atomic limit, the computational burden of these simulations grows severely. The central bottleneck lies in solving the Dyson equation for the retarded Green's function, $G^R(E) = (E I - H - \Sigma^R)^{-1}$, and the Keldysh equation for the correlation function, $G^< = G^R \Sigma^< G^A$ \cite{Vetsch2025}. For realistic systems comprising tens of thousands of atoms, direct dense inversion with $O(N^3)$ scaling is prohibitive.\\
For decades, the specific sparsity structure of quasi-1D nanowires---typically block-tridiagonal---has been exploited by the recursive Green's function (RGF) algorithm to achieve linear $O(N)$ scaling \cite{d1990conductance,mackinnon1985calculation,svizhenko2002two,svizhenko2002role}. However, the classical RGF formulation faces two main and distinct challenges in the exascale era. First, it is inherently sequential, creating a scalability wall on modern supercomputing nodes \cite{Drouvelis2006, Petersen2009}. Second, it is rigidly tied to nearest-neighbor interactions; extending it to higher-order discretizations (block $n$-diagonal) often requires artificial block enlargement, inflating arithmetic and memory costs \cite{Kilic2013, Jain2007}.\\
In the realm of numerical linear algebra, significant progress has been made to generalize transport solvers. Techniques such as parallel Selected Inversion (P-SelInv) \cite{Jacquelin2016} and block-tridiagonal solvers on GPUs (Serinv) \cite{Maillou2025} have been recently developed to take advantage of large-scale and heterogeneous machines. Similarly, nested dissection approaches \cite{Li2008, Zhao2018, Hetmaniuk2013} and hierarchical matrix compressions \cite{Chavez2018} have pushed the boundaries of large-scale electrostatics.
Despite these advances, a disconnection remains between the high-performance numerical linear algebra community and the physics-driven need for non-equilibrium observables. In this paper, we revisit the recursive algorithms for NEGF with the specific aim of bridging this gap. Our contribution is threefold:
\begin{enumerate}
    \item We reformulate the RGF method explicitly within the framework of domain decomposition. By identifying the recursive steps as operations on Schur complements of interface separators, we generalize the algorithm to handle block $n$-diagonal systems naturally. This allows for efficient treatment of next-nearest-neighbor tight-binding models without unnecessary block inflation.
    \item We introduce DDRGF (Domain-Decomposition RGF), a parallel algorithm that partitions the device into macroscopic domains. Although it has similarities with previous parallelization strategies, e.g., \cite{Petersen2009,Maillou2025}, it is novel in its structural presentation, allowing for easy reproducibility while being extensible to stencils beyond nearest-neighbors. This structural description enables a transparent formalization of the data dependencies within the method. Furthermore, by having a domain decomposition which is a \textit{refinement} of the process decomposition, we can study strong scaling at fixed total complexity.
    \item We write down a very detailed cost model for DDRGF based on its clear structural description. By means of this model, the concurrency-to-complexity trade-off is immediately exposed, allowing a deeper understanding of the parallelization strategy and tuning it for optimal performance in an automated way.
\end{enumerate}
The algorithms presented here are implemented in \texttt{LibNEGF.jl} \cite{LibNEGF_jl_2026_20043546}, a high-performance Julia library. We validate the accuracy and scaling of our approach, demonstrating that reformulating transport problems through the lens of domain decomposition not only clarifies the underlying linear algebra but also unlocks new pathways for parallelism and algorithmic extension. Our unified approach also serves as the theoretical base for realizing similar extensions on the Keldysh problem, which will be part of subsequent stages of our work. Moreover, the high-performance implementations in \texttt{LibNEGF.jl} are the first steps towards a publicly available library that can be easily integrated into workflows dealing with block banded systems, to be equipped with various levels of parallelism and supporting several GPU backends.\\
The structure of this work is as follows. We formalize and extend the RGF method in sect.\ \ref{sec:LDU_and_RGF}, first by making clear the connection between domain decomposition and the $LDU$ factorization, and then by applying this relation recursively, enabling the solution of block $n$-diagonal systems with $n \geq 3$ via RGF, sect.\ \ref{sec:RGF_block_ndiagonal}. Then, clusterization of the principal layers via domain decomposition leads to the introduction of the DDRGF method, sect.\ \ref{sec:DDRGF}, making evident the unavoidable trade-off between concurrency and complexity when parallelizing the RGF method. Cost models are written for all these methods for understanding their scaling, but also to put their pros and cons into perspective. Finally, extensive numerical methods are presented in sect.\ \ref{sec:NumericalExperiments}, with a discussion on why these parallelization strategies are necessary, but also on what are their limitations.

\section{The LDU factorization and the recursive Green's function method}\label{sec:LDU_and_RGF}
Let us, in general, define the argument of the inversion in the Dyson equation to be a matrix $M := EI - H - \Sigma^{r}$, with $H,\Sigma^{r} \in \mathbb{C}^{n \times n}$ called \textit{Hamiltonian} and \textit{self-energy}, respectively, and $E \in \mathbb{R}$ an energy. We restrict ourselves to a physical system composed of a set of concatenated regions placed in a one-dimensional fashion, where each region has a certain number of interacting degrees of freedom.  In NEGF, the approximations are such that $M$, which models this one-dimensional arrange of regions, is block $n$-diagonal. The Dyson problem in NEGF consists of computing $\textrm{b}n{diag}(M^{-1})$, i.e., extracting the block $n$-diagonal part of the inverse of the block $n$-diagonal matrix $M$.\\
Let us assume that $M$ admits a block $LDU$ decomposition, where $L$ and $U$ are block lower- and upper-triangular, respectively, and $D$ is block diagonal, with all of $L$, $D$ and $U$ being square. To admit this decomposition, it is required that the pivots appearing in the block Gaussian elimination process are nonzero, which is fulfilled if and only if all of its leading principal minors up to order $n-1$ are nonzero. We assume from here on that the matrices appearing in NEGF computations comply with these requirements, due to their physical nature\footnote{We use synthetic matrices in our numerical experiments, sect.\ \ref{sec:NumericalExperiments}, but we ensure the compliance with these conditions for having invertible systems.}.
We start this section by establishing the relationship between the $LDU$ factorization and domain decomposition. This is followed by the use of this connection to re-derive the RGF method to solve the Dyson equation on block tridiagonal system matrices, in sect.\ \ref{sec:RGF_block_tridiagonal}. Then, we extend the RGF method in sect.\ \ref{sec:RGF_block_ndiagonal} and show how to use it on block $n$-diagonal matrices with $n \geq 3$. The approach followed in this section serves as the base for transforming the RGF method to expose concurrency, see sect.\ \ref{sec:DDRGF}, leading to the DDRGF method.

\subsection{LDU and domain decomposition}\label{sec:RGF_ldu_and_domain_decomposition}
We start by connecting domain decomposition to the LDU factorization, which is a simple but useful association that helps in understanding and extending the RGF method.
In general, the underlying degrees of freedom of the physical system associated to any given matrix $M$ can be separated into two non-overlapping but interacting domains. Let us assume we are given a matrix $\widetilde{M}$ and we possibly reorder its degrees of freedom via a permutation matrix $R$, i.e., we work with the possibly transformed matrix $M = R \widetilde{M} R^{T}$. We then group the degrees of freedom of the system in two non-overlapping but interacting regions, 1 and 2, leading to the following block structure for $M$:
$$ M = \left(
\begin{array}{c c}
    M_{11} & M_{12}   \\
    M_{21} & M_{22}
\end{array}
\right). $$
Note that the operator $R$ is the identity in sect.s \ref{sec:RGF_block_tridiagonal} and \ref{sec:RGF_block_ndiagonal}, but different from the identity in sect.\ \ref{sec:DDRGF}. We can then perform an UDL decomposition on $M$, with $U$ and $L$ block upper- and lower-triangular and $D$ block diagonal, and furthermore take its inverse:
\begin{equation} \label{eq:2x2_UDL_factorization}
M = 
\begin{pmatrix}
I & M_{12}M_{22}^{-1}\\
0 & I
\end{pmatrix}
\begin{pmatrix}
M_{S} & 0\\
0     & M_{22}
\end{pmatrix}
\begin{pmatrix}
I                 & 0\\
M_{22}^{-1}M_{21} & I
\end{pmatrix}, \ \ \ M_{S} = M_{11} - M_{12}M_{22}^{-1}M_{21},
\end{equation}
\begin{equation} \label{eq:2x2_explicit_inverse}
M^{-1} = 
\begin{pmatrix}
I                  & 0\\
-M_{22}^{-1}M_{21} & I
\end{pmatrix}
\begin{pmatrix}
M_{S}^{-1} & 0\\
0     & M_{22}^{-1}
\end{pmatrix}
\begin{pmatrix}
I & -M_{12}M_{22}^{-1}\\
0 & I
\end{pmatrix}.
\end{equation}
In the next section we make use of this factorization to rederive the RGF method.

\subsection{Recursive Green's function method for block tridiagonal systems}\label{sec:RGF_block_tridiagonal}
Let us work in this section under the assumption that any two of the regions concatenated in a one-dimensional way interact only if they are next to each other. This translates to a block tridiagonal matrix in the NEGF formalism. Let us then replace the general label $M$ by a more specific $T$, to indicate the block tridiagonal matrix modeling the system. We can make the block-sparse structure of $T$ more clear by writing it as
\begin{equation}\label{eq:parametrization_M_tridiagonal}
    T_{ab} = T_{a}^{[-1]}\delta_{a-1,b} + T_{a}^{[0]}\delta_{a,b} + T_{a}^{[+1]}\delta_{a+1,b}, \ a,b = 1, 2, ..., \ell,
\end{equation}
where $T_{a}^{[-1]}$, $T_{a}^{[0]}$ and $T_{a}^{[+1]}$ are the left, center and right dense matrices populating the block row $a$ of $T$, $\delta$ is the delta Kronecker and we have $\ell$ concatenated regions, called \textit{principal layers} in the NEGF formalism. This formulation of $T$ will be useful later in this section and the next one, and even more later in sect.\ \ref{sec:DDRGF}.\\
We make use now of the interplay of domain decomposition and the LDU factorization in the derivation of RGF, by means of a $4{\times}4$ example.
This matrix corresponds to a physical system divided in four principal layers $\{1,2,3,4\}$, with nearest neighbor interactions among the layers:
\begin{equation} \label{eq:block_tridiagonal_matrix}
T = \left(
\begin{array}{c c c c}
    T_{11} & T_{12} & 0      & 0       \\
    T_{21} & T_{22} & T_{23} & 0       \\
    0      & T_{32} & T_{33} & T_{34}  \\
    0      & 0      & T_{43} & T_{44}
\end{array}
\right).
\end{equation}
Next, we split the physical system in two non-overlapping but interacting domains: $\mathcal{D}_{1} = \{1,2,3\}$ and $\mathcal{D}_{2} = \{4\}$. Then, we take the UDL (i.e., block upper-triangular, diagonal and lower-triangular) factorization based on this domain decomposition:
\begin{equation*}
T_{S}^{(4)} := T = 
\left(
\begin{array}{c c c c}
    I      & 0      & 0      & 0              \\
    0      & I      & 0      & 0              \\
    0      & 0      & I      & T_{34}t_{44}  \\
    0      & 0      & 0      & I
\end{array}
\right)
\left(
\begin{array}{c c c c}
    T_{11} & T_{12} & 0              & 0       \\
    T_{21} & T_{22} & T_{23}         & 0       \\
    0      & T_{32} & t_{S}^{(3)}    & 0       \\
    0      & 0      & 0              & t_{S}^{(4)}
\end{array}
\right)
\left(
\begin{array}{c c c c}
    I      & 0      & 0                  & 0       \\
    0      & I      & 0                  & 0       \\
    0      & 0      & I                  & 0       \\
    0      & 0      & t_{44}T_{43}      & I
\end{array}
\right),
\end{equation*}
with $t_{S}^{(4)} = T_{44}$, $t_{44} = \left( t_{S}^{(4)} \right)^{-1}$ and $t_{S}^{(3)} = T_{33} - T_{34} t_{44} T_{43}$. Note that, in constructing $T_{S}^{(4)}$, we have simply applied the factorization in eq.\ \ref{eq:2x2_UDL_factorization}. With this factorization at hand we can take the inverse of $T$ just as in eq.\ \ref{eq:2x2_explicit_inverse}:
\begin{equation} \label{eq:one_step_LDU_for_T_inverse}
T^{-1} = \left(T_{S}^{(4)}\right)^{-1} = 
\left(
\begin{array}{c c c c}
    I      & 0      & 0                  & 0       \\
    0      & I      & 0                  & 0       \\
    0      & 0      & I                  & 0       \\
    0      & 0      & -t_{44}T_{43}      & I
\end{array}
\right)
\left(
\begin{array}{c c c c}
       &                                 &                & 0       \\
       & \left( T_{S}^{(3)} \right)^{-1} &                & 0       \\
       &                                 &                & 0       \\
    0  & 0                               & 0              & t_{44}
\end{array}
\right)
\left(
\begin{array}{c c c c}
    I      & 0      & 0      & 0              \\
    0      & I      & 0      & 0              \\
    0      & 0      & I      & -T_{34}t_{44}  \\
    0      & 0      & 0      & I
\end{array}
\right),
\end{equation}
where
\[
T_{S}^{(3)} = \left(
\begin{array}{c c c}
    T_{11} & T_{12} & 0           \\
    T_{21} & T_{22} & T_{23}      \\
    0      & T_{32} & t_{S}^{(3)}
\end{array}
\right).
\]
Had we used $\ell=5$ instead of $\ell=4$, the reduction from $T_{S}^{(5)}$
down to $T_{S}^{(4)}$ would have been equivalent, going in that case also from a block tridiagonal matrix for the original system down to a block tridiagonal Schur complement. This shows the generality of our approach which is completely independent from the number of principle layers and the specific $4{\times}4$ example used here. In other words, our conclusions from rederiving RGF in this section hold for any number of principal layers $\ell$. Let us make this more precise first for the block sparse structure of the Schur complement.
\begin{proposition}\label{propo:RGF_Schur_complement_pattern}
Let $T$, block tridiagonal, correspond to a physical system with $\ell$ principal layers, labeled $a=1,2,...,\ell$. Let us split the layers in two domains, $\mathcal{D}_{1}=\{1,2,...,\ell-1\}$ and $\mathcal{D}_{2} = \{\ell\}$. Then, if we take the factorizations in eq.s \ref{eq:2x2_UDL_factorization} and \ref{eq:2x2_explicit_inverse}, the corresponding Schur complement $T_{S}^{(\ell-1)}$ is also block tridiagonal.
\end{proposition}
\begin{proof}
Take the same parametrization as in eq.\ \ref{eq:parametrization_M_tridiagonal}:
$$ T_{ab} = T_{a}^{[-1]}\delta_{a-1,b} + T_{a}^{[0]}\delta_{a,b} + T_{a}^{[+1]}\delta_{a+1,b}, \ a,b = 1, 2, ..., \ell. $$
The Schur complement is then built as
$$ (T_{S}^{(\ell-1)})_{ab} = T_{ab} - T_{a\ell}t_{\ell\ell}T_{\ell b}, \ a,b = 1,2,...,\ell-1, $$
where $t_{\ell \ell} = T_{\ell \ell}^{-1}$. It is clear that the first term contributing to the Schur complement amounts to a block tridiagonal contribution. The second one takes the following parametrized form:
$$ T_{a\ell}t_{\ell\ell}T_{\ell b} = (T_{a}^{[-1]}\delta_{a-1,\ell} + T_{a}^{[0]}\delta_{a,\ell} + T_{a}^{[+1]}\delta_{a+1,\ell})t_{\ell \ell}(T_{\ell}^{[-1]}\delta_{\ell-1,b} + T_{\ell}^{[0]}\delta_{\ell,b} + T_{\ell}^{[+1]}\delta_{\ell+1,b}), \ a,b = 1,2,...,\ell-1 . $$
This term is nonzero only if $a = \ell-1$ and $b=\ell-1$, leading to a block tridiagonal Schur complement.
\end{proof}
\noindent Going back to the $4{\times}4$ example, what remains to be determined is which part of the dense matrix $(T_{S}^{(3)})^{-1}$ is actually needed to obtain the desired $\textrm{b3diag}(T^{-1})$. The following proposition gives us the answer to this.
\begin{proposition}\label{propo:RGF_which_part_of_Schur_inverse}
Take $T$, $\mathcal{D}_{1}$ and $\mathcal{D}_{2}$ to be as in prop.\ \ref{propo:RGF_Schur_complement_pattern}. 
Then, if we take the factorizations in eq.s \ref{eq:2x2_UDL_factorization} and \ref{eq:2x2_explicit_inverse}, only $\textrm{b3diag}\left((T_{S}^{(\ell-1)})^{-1}\right)$ is needed to compute $\textrm{b3diag}(T^{-1})$.
\end{proposition}
\begin{proof}
First, reinterpret the factorization in eq.\ \ref{eq:one_step_LDU_for_T_inverse} as an LKU decomposition, with $K$ the central factor there. The $L$ and $U$ factors are the identity except for one entry, therefore:
\begin{align*}
    (T^{-1})_{ab} &= \sum_{c,d=1}^{\ell} L_{ac} K_{cd} U_{db} = \sum_{c,d=1}^{\ell} ( L_{a}^{[-1]}\delta_{a\ell}\delta_{\ell-1,c} + I\delta_{ac} ) K_{cd} ( I\delta_{db} + U_{a}^{[+1]}\delta_{d,\ell-1}\delta_{\ell,b} ) \\
    &= \sum_{c,d=1}^{\ell} ( L_{a}^{[-1]}\delta_{a,\ell}\delta_{\ell-1,c} ) K_{cd} ( I\delta_{db} ) + 
    \sum_{c,d=1}^{\ell} ( L_{a}^{[-1]}\delta_{a,\ell}\delta_{\ell-1,c} ) K_{cd} ( U_{a}^{[+1]}\delta_{d,\ell-1}\delta_{\ell,b} ) + 
    \sum_{c,d=1}^{\ell} ( I\delta_{ac} ) K_{cd} ( I\delta_{db} ) + 
    \sum_{c,d=1}^{\ell} ( I\delta_{ac} ) K_{cd} ( U_{a}^{[+1]}\delta_{d,\ell-1}\delta_{\ell,b} ) \\
    &= L_{a}^{[-1]}K_{\ell-1,b}\delta_{a,\ell} + L_{a}^{[-1]}U_{a}^{[+1]}K_{\ell-1,\ell-1}\delta_{a\ell}\delta_{\ell b} + K_{ab} + U_{a}^{[+1]}K_{a,\ell-1}\delta_{\ell b}.
\end{align*}
Due to needing only $\textrm{b3diag}(T^{-1})$, it is clear that the block tridiagonal part of $K$ is sufficient for this.
\end{proof}
\noindent Therefore, at this point the procedure is clear: apply the same factorizations on $T_{S}^{(3)}$ to compute the block tridiagonal of its inverse. This is the \textit{upward} pass of RGF. Continuing with this pass we can write the factorization in eq.\ \ref{eq:one_step_LDU_for_T_inverse} as $$ T^{-1} = L_{3} K_{3} U_{3} , $$
noting that we have used $K_{3}$ and not $D_{3}$ to clearly indicate that this is not a block diagonal matrix yet.
If we continue moving up recursively, the factorization of $T^{-1}$ will eventually become
\begin{equation} \label{eq:full_LDU_for_T_inverse}
    T^{-1} = L_{3} L_{2} L_{1} D_{1} U_{1} U_{2} U_{3} ,
\end{equation}
where $D_{1}$ is block diagonal. Note that $D_{1}$ and the $U's$ and $L's$ are taken here to be of the same size as the initial $T$. Now we have reached the end of the upward pass of RGF.\\
On the other hand, the \textit{downward} pass consists of recursively absorbing the $L$'s and $U$'s into $D_{1}$ in eq.\ \ref{eq:full_LDU_for_T_inverse}. In doing this, we can choose how many offdiagonals we want to retain in the output; as mentioned before, we want $\mathrm{b3diag}(T^{-1})$. We moreover make use of the flow illustrated in fig.\ \ref{fig:RGF_downward_pass} to get a more comprehensive understanding of this downward pass\footnote{Besides giving us a better visual understanding of the downward pass of RGF, fig.\ \ref{fig:RGF_downward_pass} will be crucial in future work when exploring recursive solutions for the Keldysh problem.}. The dotted blocks and arrows in fig.\ \ref{fig:RGF_downward_pass} indicate that we are not storing the blocks nor computing anything corresponding to those dotted arrows to get $\mathrm{b}3\mathrm{diag}(T^{-1})$.
\begin{figure}[h]
\includegraphics[width=7cm]{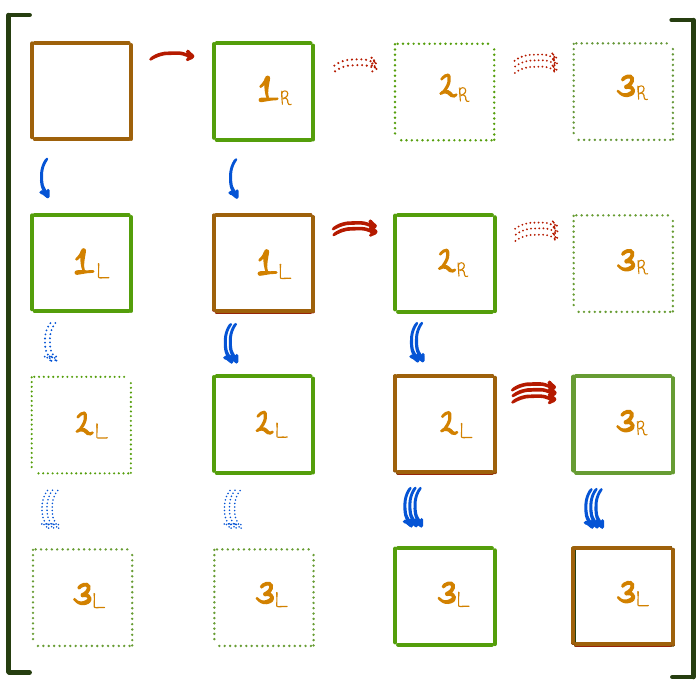}
\centering
\caption{Downward pass of RGF.} \label{fig:RGF_downward_pass}
\end{figure}

\noindent More explicitly, for this downward pass, we start with the $2{\times}2$ system:
\[
T_{S}^{(2)} = 
\left(
\begin{array}{c c}
    T_{11} & T_{12}   \\
    T_{21} & t_{S}^{(2)}
\end{array}
\right),
\]
and from this take the inverse:
\begin{align} \label{eq:2x2_exact_RGF}
\left( T_{S}^{(2)} \right)^{-1} 
&= \begin{pmatrix}
I             & 0                   \\
-t_{22}T_{21} & I
\end{pmatrix}
\begin{pmatrix}
t_{11}        & 0                   \\
0             & t_{22}
\end{pmatrix}
\begin{pmatrix}
I             & -T_{12}t_{22}  \\
0             & I
\end{pmatrix}
= \begin{pmatrix}
t_{11}               & -t_{11}T_{12}t_{22}   \\
-t_{22}T_{21}t_{11}  &  t_{22} + t_{22}T_{21}t_{11}T_{12}t_{22}
\end{pmatrix}.
\end{align}
The emergence of the elements in the rightmost matrix in eq.\ \ref{eq:2x2_exact_RGF} can be understood from the top-left $2{\times}2$ sub-matrix in fig.\ \ref{fig:RGF_downward_pass}. In that figure, we assume that we first absorb $U_{1}$ into $D_{1}$ (see eq.\ \ref{eq:full_LDU_for_T_inverse} and the only single red rightward arrow), followed by the absorption of $L_{1}$ (corresponding to the two blue single downward arrows). These right and left absorptions of $U_{1}$ and $L_{1}$ are further indicated in fig.\ \ref{fig:RGF_downward_pass} via the labels $1_{R}$ and $1_{L}$, respectively. To simplify the notation, we define:
$$ \widetilde{t}_{11} := t_{11}, \ \widetilde{t}_{22} := t_{22} + t_{22}T_{21}\widetilde{t}_{11}T_{12}t_{22}. $$
We can continue this recursion further, and after absorbing the other four ($L$ and $U$) factors in eq.\ \ref{eq:full_LDU_for_T_inverse}:
\begin{equation} \label{eq:4x4_exact_RGF}
   T^{-1} = \left(
\begin{array}{c c c c}
    \widetilde{t}_{11}              & -t_{11}T_{12}t_{22}  & t_{11}T_{12}t_{22}T_{23}t_{33}      & -t_{11}T_{12}t_{22}T_{23}t_{33}T_{34}t_{44}              \\
    -t_{22}T_{21}t_{11}             & \widetilde{t}_{22}   & -\widetilde{t}_{22}T_{23}t_{33}      & \widetilde{t}_{22}T_{23}t_{33}T_{34}t_{44}              \\
    t_{33}T_{32}t_{22}T_{21}t_{11}  & -t_{33}T_{32}\widetilde{t}_{22} & \widetilde{t}_{33}      & -\widetilde{t}_{33}T_{34}t_{44}  \\
    -t_{44}T_{43}t_{33}T_{32}t_{22}T_{21}t_{11}  & t_{44}T_{43}t_{33}T_{32}\widetilde{t}_{22} & -t_{44}T_{43}\widetilde{t}_{33}                  & \widetilde{t}_{44}
\end{array}
\right). 
\end{equation}
Again, we can connect the new terms that appear at each step of this recursion to the blocks and arrows in fig.\ \ref{fig:RGF_downward_pass}. For example, the element $\left( T^{-1}\right)_{23}$ corresponds to the only red double rightward arrow in fig.\ \ref{fig:RGF_downward_pass}, and the subsequent updates of $\left( T^{-1}\right)_{32}$ and $\left( T^{-1}\right)_{33}$ correspond to the two blue double downward arrows, and so on.\\
The concatenation of these two passes, upward and downward, gives us the RGF method. Alg.\ \ref{alg:rgf_block_tridiagonal} contains the pseudocode for RGF, where $B_{1}$ and $B_{2}$ are buffers; in practice, in the upward pass, we use the block diagonal of $M_{out}$ as $B_{2}$.
\begin{algorithm}
\caption{\enskip RGF for block tridiagonal systems.}\label{alg:rgf_block_tridiagonal}
\begin{algorithmic}
  \State \textbf{input}: $M_{inp}$, \textbf{output}: $M_{out}$
  \State // upward pass
  \State $B_{1}[\ell,\ell] = \mathrm{LU}(M_{inp}[\ell,\ell])$
  \For{$i=\ell-1:-1:1$}
    \State $B_{1}[i,i+1] = M_{inp}[i,i+1] \left( B_{1}[i+1,i+1] \right)^{-1}$
    \State $B_{1}[i+1,i] = \left( B_{1}[i+1,i+1] \right)^{-1} M_{inp}[i+1,i]$
    \State $B_{2}[i,i] = M_{inp}[i,i] - M_{inp}[i,i+1]B_{1}[i+1,i]$
    \State $B_{1}[i,i] = \mathrm{LU}(B_{2}[i,i])$
  \EndFor
  \State // downward pass
  \State $M_{out}[1,1] = \left( B_{1}[1,1] \right)^{-1}$
  \For{$i=2:\ell$}
    \State $M_{out}[i-1,i] = -M_{out}[i-1,i-1]B_{1}[i-1,i]$
    \State $M_{out}[i,i-1] = - B_{1}[i,i-1]M_{out}[i-1,i-1]$
    \State $M_{out}[i,i] = \left( B_{1}[i,i] \right)^{-1} - B_{1}[i,i-1]M_{out}[i-1,i]$
  \EndFor
\end{algorithmic}
\end{algorithm}
\vskip 0.08in
\noindent \textbf{Cost model:} to systematically understand how the execution time of our recursive methods scales with respect to the number of principal layers $\ell$ and the block size $b_{s}$, we introduce detailed analytical cost models for RGF and its variations in later sections. Beyond providing a theoretical baseline, this formalization serves a direct practical purpose: it drives an automated tuning mechanism within our \texttt{LibNEGF.jl} implementation. By evaluating these cost models using pre-measured average execution times of fundamental linear algebra kernels, the auto-tuner dynamically dictates the optimal concurrency parameters for the parallel version of RGF (see sect.\ \ref{sec:DDRGF}) prior to execution.\\
As our formulation of RGF relies on three fundamental BLAS operations: GEMM (multiplication of two dense matrices $AB$), LU (factorization of a dense matrix as $A = LU$) and GETRS (division from the left in the form $(LU)^{-1}B$). The modeling of LU and GETRS by simple FLOPs counting can be less accurate than in the GEMM case due to caching effects, which can significantly vary depending on the value of the block size $b_{s}$. To account for this, we make our cost model uniform by counting GEMMs and introducing the ratios
$$ r_{LU}(b_{s}) = \frac{t_{LU}(b_{s})}{t_{GEMM}(b_{s})}, \ r_{GETRS}(b_{s}) = \frac{t_{GETRS}(b_{s})}{t_{GEMM}(b_{s})} . $$
Here, the value of $b_{s}$ is taken as the average of the dimension of the diagonal blocks $T_{aa}$ in $T$; we assume $b_{s}$ to have relatively small variance. These ratios, for LU and GETRS, are determined on-the-fly in our implementation, i.e., whenever we need to evaluate this cost model, we use pre-computed ratios estimated from running GEMM, LU and GETRS several times and from averages of their execution times. This gives us a cost model that is easily adjustable across different architectures.
From simple counting, the reformulation of RGF presented in this section executes $\ell$ LU factorizations, $3\ell-2$ GETRSs and $4(\ell-1)$ GEMMs, totalling:
\begin{equation}\label{eq:cost_model_RGF}
    r_{RGF}(\ell,b_{s}) := \frac{t_{RGF}(\ell,b_{s})}{t_{GEMM}(b_{s})} = 4(\ell-1) + \ell r_{LU}(b_{s}) + (3\ell-2)r_{GETRS}(b_{s}) = \ell(4+r_{LU}(b_{s})+3r_{GETRS}(b_{s})) - (4+2r_{GETRS}(b_{s})) .
\end{equation}
This cost model clearly exposes the linear scaling with respect to the number of principal layers $\ell$. However, to accurately capture the dependence on the block size $b_s$, we must strictly distinguish between theoretical algorithmic complexity and empirical execution time. While the fundamental arithmetic complexity of the underlying GEMM, LU, and GETRS operations is strictly cubic ($O(b_s^3)$), the measured execution time $t_{\text{GEMM}}(b_s)$ does not follow a simple theoretical scaling law. Instead, the execution time is heavily influenced by hardware realities, such as cache hierarchy transitions, memory bandwidth saturation, and vectorization efficiency. Because these hardware constraints also cause the relative performance ratios $r_{\text{LU}}(b_s)$ and $r_{\text{GETRS}}(b_s)$ to fluctuate intricately as $b_s$ varies, the total runtime $t_{\text{RGF}}$ becomes a highly non-trivial function of the block size in practice. To rigorously characterize this practical behavior, we have included extensive microbenchmarks and a roofline model assessment in app.\ \ref{app:microbenchmarking}, providing a clear picture of the true functional dependency of $t_{\text{GEMM}}$ and these operational ratios on $b_s$.

\subsection{Recursive Green's function method for block $n$-diagonal systems}\label{sec:RGF_block_ndiagonal}
When inverting a block $n$-diagonal matrix with $n>3$, e.g.\ when computing $\mathrm{b5diag}(P^{-1})$ with $P$ a ${5\times}5$ block pentadiagonal matrix:
\[
P = 
\left(
\begin{array}{c c c c c}
    P_{11} & P_{12} & P_{13} & 0      & 0      \\
    P_{21} & P_{22} & P_{23} & P_{24} & 0      \\
    P_{31} & P_{32} & P_{33} & P_{34} & P_{35} \\
    0      & P_{42} & P_{43} & P_{44} & P_{45} \\
    0      & 0      & P_{53} & P_{54} & P_{55}
\end{array}
\right),
\]
we can use the same strategy employed in sect.\ \ref{sec:RGF_block_tridiagonal}. To this end, we can split the physical system in two domains, $\mathcal{D}_{1} = \{1,2,3,4\}$ and $\mathcal{D}_{2} = \{5\}$. Then, we can do an UDL factorization of $P$:
\begin{equation*}
P_{S}^{(5)} := P = 
\left(
\begin{array}{c c c c c}
    I      & 0      & 0      & 0   & 0                         \\
    0      & I      & 0      & 0   & 0                         \\
    0      & 0      & I      & 0   & P_{35}p_{55}       \\
    0      & 0      & 0      & I   & P_{45}p_{55}      \\
    0      & 0      & 0      & 0   & I
\end{array}
\right)
\left(
\begin{array}{c c c c c}
    P_{11} & P_{12} & P_{13}             & 0                  & 0           \\
    P_{21} & P_{22} & P_{23}             & P_{24}             & 0           \\
    P_{31} & P_{32} & (P_{S}^{(4)})_{33} & (P_{S}^{(4)})_{34} & 0           \\
    0      & P_{42} & (P_{S}^{(4)})_{43} & (P_{S}^{(4)})_{44} & 0           \\
    0      & 0      & 0                  & 0                  & p_{S}^{(5)}
\end{array}
\right)
\left(
\begin{array}{c c c c c}
    I      & 0      & 0                   & 0                          & 0       \\
    0      & I      & 0                   & 0                          & 0       \\
    0      & 0      & I                   & 0                          & 0       \\
    0      & 0      & 0                   & I                          & 0       \\
    0      & 0      & p_{55}P_{53}        & p_{55}P_{54}               & I
\end{array}
\right),
\end{equation*}
where:
$$ p_{S}^{(5)} = P_{55} , \ p_{55} = (p_{S}^{(5)})^{-1}, \ (P_{S}^{(4)})_{ab} = P_{ab} - P_{a5}p_{55}P_{5b}, \ a,b = 3,4 . $$
\noindent With the UDL factorization at hand, we then take the inverse of $P$:
\begin{equation*}
P^{-1} = \left( P_{S}^{(5)} \right)^{-1} = 
\left(
\begin{array}{c c c c c}
    I      & 0      & 0                   & 0                          & 0       \\
    0      & I      & 0                   & 0                          & 0       \\
    0      & 0      & I                   & 0                          & 0       \\
    0      & 0      & 0                   & I                          & 0       \\
    0      & 0      & -p_{55}P_{53}       & -p_{55}P_{54}              & I
\end{array}
\right)
\left(
\begin{array}{c c c c c}
           &        &                 &                & 0           \\
           & \left(P_{S}^{(4)}\right)^{-1} &                 &            & 0           \\
      &   &   &      & 0           \\
           &   &      &  & 0           \\
    0      & 0      & 0               & 0               & p_{55}
\end{array}
\right)
\left(
\begin{array}{c c c c c}
    I      & 0      & 0      & 0   & 0                         \\
    0      & I      & 0      & 0   & 0                         \\
    0      & 0      & I      & 0   & -P_{35}p_{55}       \\
    0      & 0      & 0      & I   & -P_{45}p_{55}      \\
    0      & 0      & 0      & 0   & I
\end{array}
\right),
\end{equation*}
with:
\begin{equation} \label{eq:block_pentadiagonal_4x4_system}
P_{S}^{(4)} = \left(
\begin{array}{c c c c}
    P_{11} & P_{12} & P_{13}             & 0               \\
    P_{21} & P_{22} & P_{23}             & P_{24}          \\
    P_{31} & P_{32} & (P_{S}^{(4)})_{33} & (P_{S}^{(4)})_{34}   \\
    0      & P_{42} & (P_{S}^{(4)})_{43} & (P_{S}^{(4)})_{44}
\end{array}
\right).
\end{equation}
\noindent This form for $P_{S}^{(4)}$ motivates a general statement on the structure of the Schur complement when dealing with block $n$-diagonal problems.
\begin{proposition}\label{propo:ndiag_RGF_Schur_complement_pattern}
Let $M$, block $n$-diagonal with $n \in \{3,5,7,...\}$, correspond to a physical system with $\ell$ principal layers, labeled $a=1,2,...,\ell$. Let us split the layers in two domains, $\mathcal{D}_{1}=\{1,2,...,\ell-1\}$ and $\mathcal{D}_{2} = \{\ell\}$. Then, if we take the factorizations in eq.s \ref{eq:2x2_UDL_factorization} and \ref{eq:2x2_explicit_inverse}, the corresponding Schur complement $M_{S}^{(\ell-1)}$ is also block $n$-diagonal.
\end{proposition}
\begin{proof}
First, we write the parametrized form of $M$:
$$ M_{ab} =  \sum_{i=-(n-1)/2}^{(n-1)/2} M_{a}^{[i]}\delta_{a+i,b}. $$
The Schur complement is built as
$$ (M_{S}^{(\ell-1)})_{ab} = M_{ab} - M_{a\ell}m_{\ell\ell}M_{\ell b}, \ a,b = 1,2,...,\ell-1, $$
where $m_{\ell \ell} = M_{\ell \ell}^{-1}$. It is clear that the first term contributing to the Schur complement amounts to a block $n$-diagonal contribution. The second one takes the following parametrized form:
$$ M_{a\ell}m_{\ell\ell}M_{\ell b} = 
\left(\sum_{i=-(n-1)/2}^{(n-1)/2} M_{a}^{[i]}\delta_{a+i,\ell}\right)m_{\ell \ell}\left(\sum_{j=-(n-1)/2}^{(n-1)/2} M_{\ell}^{[j]}\delta_{\ell+j,b}\right) = \sum_{i=-(n-1)/2}^{(n-1)/2}\sum_{j=-(n-1)/2}^{(n-1)/2} M_{a}^{[i]} m_{\ell \ell} M_{\ell}^{[j]} \delta_{a+i,\ell} \delta_{\ell+j,b} , \ a,b = 1,2,...,\ell-1 . $$
This term is nonzero only if
$$ a, b = \ell-1-\left(\frac{n-1}{2}-1\right), \ell-1-\left(\frac{n-1}{2}-2\right),...,\ell-1, $$
leading to a block $n$-diagonal Schur complement.\\
\end{proof}
\noindent In a similar way to the block tridiagonal case, we can also show that we need only $\textrm{b}n\textrm{diag}((M_{S}^{(\ell-1)})^{-1})$ to compute $\mathrm{b}n\textrm{diag}(M^{-1})$; we skip this proof here. The recursion at this point becomes general and we can take the inverse of $P_{S}^{(4)}$ by recursively employing the same corresponding steps used for $P_{S}^{(5)}$.\\
To explicitly detail the downward pass for $n>3$, we must move beyond the simple nearest-neighbor absorption depicted for the block tridiagonal case. When recursively absorbing the L and U factors into the central block-diagonal to extract $\textrm{bndiag}(P^{-1})$, the algorithmic dependencies expand according to the off-diagonal bandwidth $w = (n-1)/2$. Starting from the $w{\times}w$ top-left system, the downward recursion step for a target principal layer $i$ (where $i$ goes from 2 to $\ell$) requires updating $w$ upper off-diagonal blocks, $w$ lower off-diagonal blocks, and the central diagonal block. Specifically, the computation of $(P^{-1})_{i-k,i}$ and $(P^{-1})_{i,i-k}$ (for $k = 1,...,\textrm{min}(w,i-1)$) requires multiplying the previously resolved $(P^{-1})_{i-k,i-k}$ with the generalized L and U multipliers. Consequently, the update for the diagonal element $(P^{-1})_{i,i}$ is no longer a single matrix addition as in eq.\ \ref{eq:2x2_exact_RGF}, but an accumulation of w distinct matrix products involving the newly computed off-diagonal blocks. This expanded dependency web guarantees that exactly - and only - the required block $n$-diagonal elements are computed.
\vskip 0.08in
\noindent \textbf{Cost model:} to explicitly quantify the efficiency of our generalized block $n$-diagonal formulation, we derive an exact cost model that generalizes eq.\ \ref{eq:cost_model_RGF}. Let $w = (n-1)/2$ represent the off-diagonal bandwidth. In the upward pass, at step $i$ (going from $\ell-1$ down to 1), the boundary constraints dictate that the algorithm calculates left and right multipliers requiring $2 \min(w, i)$ GETRS operations, and applies the Schur complement update requiring $\min(w, i)^2$ GEMMs. In the downward pass, for layer $i$ (going from 2 to $\ell$), extracting the block $n$-diagonal elements requires $2 \min(w, i-1)^2$ GEMMs for the upper and lower off-diagonals, and $\min(w, i-1)$ GEMMs plus 1 GETRS for the central diagonal update. Both passes collectively execute exactly $\ell$ LU factorizations.\\
By precisely evaluating the summations over these loop bounds (assuming $\ell > w$), the exact number of fundamental operations for the $n$-diagonal RGF is:
\begin{itemize}
    \item \textbf{LU:} $\ell$
    \item \textbf{GETRS:} $\ell(2w + 1) - w^2 - w$
    \item \textbf{GEMMs:} $(3w^2 + w)\ell - 2w^3 - 2w^2$
\end{itemize}
\noindent The total execution cost, normalized to $t_{\text{GEMM}}(b_s)$, is therefore:
\begin{equation} \label{eq:cost_model_general_RGF}
    r_{\text{nRGF}}(\ell, b_s, w) = \left[ (3w^2 + w)\ell - 2w^3 - 2w^2 \right] + \ell \cdot r_{\text{LU}}(b_s) + \left[ \ell(2w + 1) - w^2 - w \right] r_{\text{GETRS}}(b_s).
\end{equation}
Note that for $w=1$, this perfectly reduces to $4(l-1)$ GEMMs and $3\ell-2$ GETRS, recovering eq.\ \ref{eq:cost_model_RGF}. An alternative and commonly used approach to handle $n>3$ systems is ``blocks fusing". In this strategy, $w$ adjacent principal layers of size $b_s$ are artificially grouped into a single ``super-block'' of size $W_s = w \cdot b_s$. This physically reduces the block $n$-diagonal system back to a standard block tridiagonal ($n=3$) system, allowing the direct use of the traditional RGF algorithm on a matrix with $\ell' = \lceil \ell/w \rceil$ principal layers. Using eq.\ \ref{eq:cost_model_RGF} directly on the fused system, the exact cost is:
\begin{equation} \label{eq:cost_model_general_RGF_fused}
    r_{\text{fused}}(\ell', W_s) = 4(\ell'-1) + \ell' \cdot r_{\text{LU}}(W_s) + (3\ell'-2) r_{\text{GETRS}}(W_s).
\end{equation}
While mathematically equivalent, comparing these exact models reveals a stark computational trade-off. Because the arithmetic complexity of GEMM, LU, and GETRS scales approximately cubically with the block size (e.g., $t_{\text{GEMM}}(W_s) \approx w^3 \cdot t_{\text{GEMM}}(b_s)$), the fusing strategy severely inflates the base computational cost. The fused approach effectively requires a runtime proportional to $4(\ell/w) \cdot w^3 = 4w^2 \ell$ base GEMMs, which is asymptotically heavier than the $(3w^2 + w)\ell$ GEMM scaling of the generalized $n$-diagonal RGF.\\
More critically, block fusing forces dense $O(W_s^3)$ GEMMs, LU factorizations and GETRS operations on the artificially bloated super-blocks. Conversely, our native $n$-diagonal formulation preserves the smaller $b_s \times b_s$ block granularity. This native approach strictly restricts computations to the non-zero sparsity pattern, completely avoiding the zero-padding overhead inherent to super-blocks and keeping the computationally expensive factorizations bound to approximately $O(b_s^3)$.

\section{Domain Decomposition Based Recursive Green's Function Method}\label{sec:DDRGF}
We now formulate a variation of RGF, exposing concurrency, for computing $\mathrm{b}3\mathrm{diag}(T^{-1})$. We call this method domain decomposition based RGF (DDRGF). To this end, the physical system is split into contiguous but disjoint sub-domains, in the way illustrated in fig.\ \ref{fig:parallel_RGF_domains}, where each rectangular colored block is a physical principal layer. In the representation of the original system matrix $T$, the principal layers in fig.\ \ref{fig:parallel_RGF_domains} are indexed as $1,2,3,...,\ell$. We turn now to a new indexing, induced by a permutation operator $R_{3}$ which indexes first the principal layers from domain $\mathcal{D}_{1}$ and then those of domain $\mathcal{D}_{2}$, leading to:
\begin{equation} \label{eq:T_hat_overall_structure}
    \widehat{T} = R_{3} T R_{3}^{T} = \left(
\begin{array}{c c}
    \widehat{T}_{11}  & \widehat{T}_{12}   \\
    \widehat{T}_{21}  & \widehat{T}_{22}
\end{array}
\right).
\end{equation}
\begin{figure}[h]
    \centering
        \begin{tikzpicture}

            \definecolor{dombluefill}{RGB}{164,194,230}
            \definecolor{domblueborder}{RGB}{46,116,181}
            \definecolor{domorangefill}{RGB}{244,177,131}
            \definecolor{domredborder}{RGB}{192,0,0}
            \definecolor{textorange}{RGB}{227,108,9}
            \definecolor{textred}{RGB}{152,38,73}
        
            \tikzset{
                bblock/.style={fill=dombluefill, draw=black, line width=0.5pt},
                oblock/.style={fill=domorangefill, draw=black, line width=0.5pt},
                bdomain/.style={draw=domblueborder, dashed, line width=1pt},
                odomain/.style={draw=domredborder, dashed, line width=1pt}
            }
        
            
            \newcommand{\drawblue}[2]{
                \begin{scope}[shift={(#1,0)}]
                    \draw[bdomain] (0, -1.6) rectangle (2.0, 1.6);
                    
                    \draw[bblock] (0, -0.7) rectangle (0.4, 0.7);
                    \draw[bblock] (0.4, -0.4) rectangle (0.8, 0.4);
                    \draw[bblock] (0.8, -1.2) rectangle (1.2, 1.2);
                    \draw[bblock] (1.2, -1.2) rectangle (1.6, 1.2);
                    \draw[bblock] (1.6, -0.7) rectangle (2.0, 0.7);
                    
                    \node[text=textorange] at (1.0, -2.5) {\LARGE $\mathcal{D}_2^{#2}$};
                \end{scope}
            }
        
            \newcommand{\draworange}[2]{
                \begin{scope}[shift={(#1,0)}]
                    \draw[odomain] (0, -1.6) rectangle (2.0, 1.6);
                    
                    \draw[oblock] (0, -0.9) rectangle (0.4, 0.9);
                    \draw[oblock] (0.4, -1.0) rectangle (0.8, 1.0);
                    \draw[oblock] (0.8, -0.5) rectangle (1.2, 0.5);
                    \draw[oblock] (1.2, -1.3) rectangle (1.6, 1.3);
                    \draw[oblock] (1.6, -0.8) rectangle (2.0, 0.8);
                    
                    \node[text=textred] at (1.0, -2.5) {\LARGE $\mathcal{D}_1^{#2}$};
                \end{scope}
            }
        
            
            \drawblue{0}{1}
            \draworange{2}{1}
            \drawblue{4}{2}
            \draworange{6}{2}
        
            \fill[domredborder] (8.5, 0) circle (3pt);
            \fill[domredborder] (9.1, 0) circle (3pt);
            \fill[domredborder] (9.7, 0) circle (3pt);
        
            \draworange{10.2}{\ell_1-1}
            \drawblue{12.2}{\ell_2}
            \draworange{14.2}{\ell_1}
        
    \end{tikzpicture}
    \caption{Domain decomposition of the original physical system in two domains, $\mathcal{D}_{1}$ and $\mathcal{D}_{2}$.}\label{fig:parallel_RGF_domains}
\end{figure}
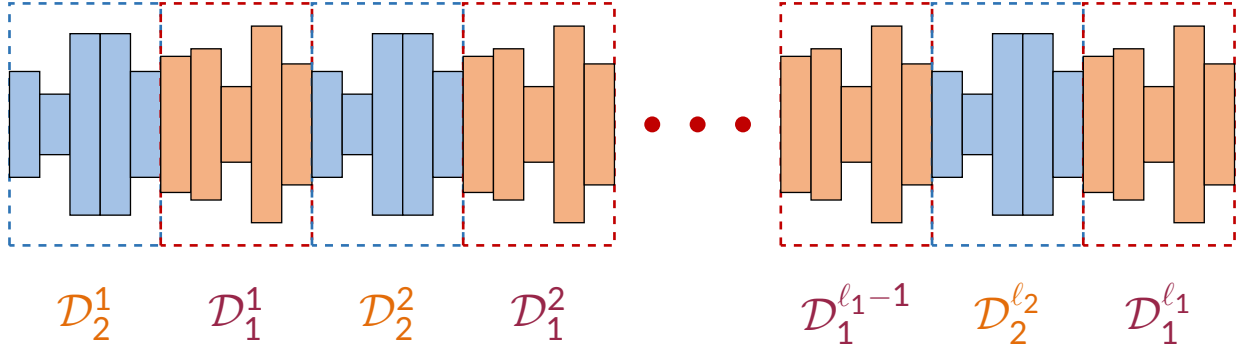
\noindent We always enforce $\ell_{1}=\ell_{2}$ in this work; this is driven by the need to have good load balance in a parallel implementation of this algorithm. The sub-index in $R_{3}$ indicates the reordering of a block $n$-diagonal matrix with $n=3$; $R_{n}$ is the reordering of block $n$-diagonal matrices. Note that $R_{3}R_{3}^{T} = I$ and therefore
\begin{equation} \label{eq:reordered_Tinverse_for_parallel_RGF}
    T^{-1} = R_{3}^{T} \widehat{T}^{-1} R_{3}.
\end{equation}
The task at hand is to compute
\begin{equation}\label{eq:mapping_b3diag_T_to_That}
    \left( \widehat{T}^{-1} \right)_{\mathrm{b3diag}} := R_{3} \left(\mathrm{b3diag}(T^{-1})\right) R_{3}^{T},
\end{equation}
implying the extraction of blocks from within $\widehat{T}^{-1}$ that match the nonzero pattern in $\widehat{T}$. Hence, our focus is on $\widehat{T}^{-1}$ from hereon.
The matrices $\widehat{T}_{11}$ and $\widehat{T}_{22}$ are block-diagonal at the level of the sub-domains within $\mathcal{D}_{1}$ and $\mathcal{D}_{2}$, respectively, and each of their submatrices are block tridiagonal within each sub-domain. Furthermore, we can use eq.s \ref{eq:2x2_UDL_factorization} and \ref{eq:2x2_explicit_inverse}:
\begin{equation} \label{eq:2x2_explicit_inverse_parallel_DDRGF}
\widehat{T}^{-1} = 
\begin{pmatrix}
I                  & 0\\
-\widehat{T}_{22}^{-1}\widehat{T}_{21} & I
\end{pmatrix}
\begin{pmatrix}
\widehat{T}_{S}^{-1} & 0\\
0          & \widehat{T}_{22}^{-1}
\end{pmatrix}
\begin{pmatrix}
I & -\widehat{T}_{12}\widehat{T}_{22}^{-1}\\
0 & I
\end{pmatrix} = 
\begin{pmatrix}
\widehat{T}_{S}^{-1}               & -\widehat{T}_{S}^{-1}\widehat{T}_{12}\widehat{T}_{22}^{-1}                      \\
-\widehat{T}_{22}^{-1}\widehat{T}_{21}\widehat{T}_{S}^{-1}  & \widehat{T}_{22}^{-1} + \widehat{T}_{22}^{-1}\widehat{T}_{21}\widehat{T}_{S}^{-1}\widehat{T}_{12}\widehat{T}_{22}^{-1}
\end{pmatrix},
\end{equation}
\begin{equation} \label{eq:Schur_complement_in_parallel_DDRGF}
    \widehat{T}_{S} = \widehat{T}_{11} - \widehat{T}_{12}\widehat{T}_{22}^{-1}\widehat{T}_{21} .
\end{equation}
We show now that, just as in RGF, the Schur complement here is also always block tridiagonal.
\begin{proposition}\label{propo:DDRGF_Schur_complement_pattern}
Let $T$, block tridiagonal, correspond to a physical system with $\ell$ principal layers, labeled $a=1,2,...,\ell$. Let us split the layers in two domains, $\mathcal{D}_{1}$ and $\mathcal{D}_{2}$, in the same manner as in fig.\ \ref{fig:parallel_RGF_domains}. Then, if we take the reordering in eq.\ \ref{eq:T_hat_overall_structure} and the factorizations in eq.s \ref{eq:2x2_explicit_inverse_parallel_DDRGF} and \ref{eq:Schur_complement_in_parallel_DDRGF}, the Schur complement $\widehat{T}_{S}$ is also block tridiagonal.
\end{proposition}
\begin{proof}
We call a matrix \textit{almost} block tridiagonal when its non-zero blocks are within the matrix band defined by the tridiagonal blocks.
The addition of two almost block tridiagonal matrices is clearly almost block tridiagonal as well. We start by noting that $\widehat{T}_{11}$ is almost block tridiagonal, with the blocks connecting the different sub-domains in $\mathcal{D}_{1}$ being zero due to the disjointness of the constructed domain decomposition. Therefore, we need to show only that $\widehat{T}_{12}\widehat{T}_{22}^{-1}\widehat{T}_{21}$ is almost block tridiagonal to prove this proposition. For this, we need to introduce indices $i$ and $j$ labeling the sub-domains $\mathcal{D}_{1}^{i}$ and $\mathcal{D}_{2}^{j}$, respectively. Now, let us first parametrize the $\widehat{T}_{12}$ factor as:
\begin{align*}
    \left((\widehat{T}_{12})_{im}\right)_{ac} = (\widehat{T}_{12})_{i}^{[-1]}\delta_{i-1,m}\delta_{a,1}\delta_{c,\ell_{2}} + (\widehat{T}_{12})_{i}^{[0]}\delta_{i,m}\delta_{a,\ell_{1}}\delta_{c,1} , \\ i = 1,2,...,\ell_{1}, \ m = 1,2,...,\ell_{2}, \ a = 1,2,...,|\mathcal{D}_{1}^{i}|, \ c = 1,2,...,|\mathcal{D}_{2}^{m}| ,
\end{align*}
where, just as in eq.\ \ref{eq:parametrization_M_tridiagonal}, $(\widehat{T}_{12})_{i}^{[\pm 1]}$ and $(\widehat{T}_{12})_{i}^{[0]}$ are dense blocks populating the block matrix, and the indices $a$ and $c$ run over the layers in the sub-domains in $\mathcal{D}_{1}$ and $\mathcal{D}_{2}$, respectively. Correspondingly:
\begin{align*}
    \left((\widehat{T}_{21})_{nj}\right)_{db} = (\widehat{T}_{12})_{n}^{[+1]}\delta_{n+1,j}\delta_{d,\ell_{1}}\delta_{b,1} + (\widehat{T}_{12})_{n}^{[0]}\delta_{n,j}\delta_{d,1}\delta_{b,\ell_{2}} , \\ n = 1,2,...,\ell_{1}, \ j = 1,2,...,\ell_{2}, \ d = 1,2,...,|\mathcal{D}_{1}^{n}|, \ b = 1,2,...,|\mathcal{D}_{2}^{j}| .
\end{align*}
Then:
\begin{align*}
\left((\widehat{T}_{12}\widehat{T}_{22}^{-1}\widehat{T}_{21})_{ij}\right)_{ab} &= \sum_{m,n}\sum_{c,d} \left((\widehat{T}_{12})_{im}\right)_{ac} \left( (\widehat{T}_{22}^{-1})_{mn}\delta_{mn} \right)_{cd} \left((\widehat{T}_{21})_{nj}\right)_{db} \\
&= \sum_{m} \sum_{c,d} \left( (\widehat{T}_{12})_{i}^{[-1]}\delta_{i-1,m}\delta_{a,1}\delta_{c,\ell_{2}} + (\widehat{T}_{12})_{i}^{[0]}\delta_{i,m}\delta_{a,\ell_{1}}\delta_{c,1} \right) \left( (\widehat{T}_{22}^{-1})_{mm} \right)_{cd} \left( (\widehat{T}_{12})_{m}^{[+1]}\delta_{m+1,j}\delta_{d,\ell_{1}}\delta_{b,1} + (\widehat{T}_{12})_{m}^{[0]}\delta_{m,j}\delta_{d,1}\delta_{b,\ell_{2}} \right) \\
&= \sum_{m} \sum_{c} \left( (\widehat{T}_{12})_{i}^{[-1]}\delta_{i-1,m}\delta_{a,1}\delta_{c,\ell_{2}} + (\widehat{T}_{12})_{i}^{[0]}\delta_{i,m}\delta_{a,\ell_{1}}\delta_{c,1} \right) \left( \left( (\widehat{T}_{22}^{-1})_{mm} \right)_{c,\ell_{1}}(\widehat{T}_{12})_{m}^{[+1]}\delta_{m+1,j}\delta_{b,1} + \left( (\widehat{T}_{22}^{-1})_{mm} \right)_{c,1}(\widehat{T}_{12})_{m}^{[0]}\delta_{m,j}\delta_{b,\ell_{2}} \right) \\
&= (\widehat{T}_{12})_{i}^{[-1]}\delta_{i-1,j-1}\delta_{a,1} \left( (\widehat{T}_{22}^{-1})_{j-1,j-1} \right)_{\ell_{2},\ell_{1}}(\widehat{T}_{12})_{j-1}^{[+1]}\delta_{b,1} + (\widehat{T}_{12})_{i}^{[-1]}\delta_{i-1,j}\delta_{a,1} \left( (\widehat{T}_{22}^{-1})_{jj} \right)_{\ell_{2},1}(\widehat{T}_{12})_{j}^{[0]}\delta_{b,\ell_{2}} \\
&+ (\widehat{T}_{12})_{i}^{[0]}\delta_{i,j-1}\delta_{a,\ell_{1}} \left( (\widehat{T}_{22}^{-1})_{j-1,j-1} \right)_{1,\ell_{1}}(\widehat{T}_{12})_{j-1}^{[+1]}\delta_{b,1} + (\widehat{T}_{12})_{i}^{[0]}\delta_{i,j}\delta_{a,\ell_{1}} \left( (\widehat{T}_{22}^{-1})_{jj} \right)_{1,1}(\widehat{T}_{12})_{j}^{[0]}\delta_{b,\ell_{2}} .
\end{align*}
By examining the Kronecker deltas in this final expression, we can rigorously determine the resulting block sparsity pattern. The presence of terms proportional to $\delta_{i,j}$ and $\delta_{i-1,j-1}$ dictates self-interactions within the individual sub-domains. More importantly, the cross-terms containing $\delta_{i-1,j}$ and $\delta_{i,j-1}$ explicitly couple a given sub-domain $i$ exclusively to its immediate nearest neighbors, $j=i+1$ and $j=i-1$. Because all other domain combinations evaluate to zero, the non-zero elements of the product $\widehat{T}_{12}\widehat{T}_{22}^{-1}\widehat{T}_{21}$ are strictly confined to the main diagonal and the first upper and lower sub-diagonals at the sub-domain level. Consequently, this product exhibits an almost block tridiagonal structure. Since the initial sub-matrix $\widehat{T}_{11}$ shares this exact same property, their difference---the Schur complement $\widehat{T}_S$---inherently preserves a block tridiagonal structure in general.
\end{proof}
\begin{figure}[h]
    \centering
        \begin{tikzpicture}[scale=0.8,
            bx/.style={draw=olive!80!black, thick, minimum size=0.55cm, inner sep=0pt, fill=white},
            bxB/.style={draw=blue!80!black, thick, minimum size=0.55cm, inner sep=0pt, fill=white},
            mathnode/.style={font=\Huge}
        ]
        
        \newcommand{\dotTL}[2]{\fill[#2] ([xshift=2.5pt, yshift=-4.5pt]#1.north west) +(-1pt,-1pt) rectangle +(3pt,3pt);}
        \newcommand{\dotTR}[2]{\fill[#2] ([xshift=-4.5pt, yshift=-4.5pt]#1.north east) +(-1pt,-1pt) rectangle +(3pt,3pt);}
        \newcommand{\dotBL}[2]{\fill[#2] ([xshift=2.5pt, yshift=2.8pt]#1.south west) +(-1pt,-1pt) rectangle +(3pt,3pt);}
        \newcommand{\dotBR}[2]{\fill[#2] ([xshift=-4.5pt, yshift=2.8pt]#1.south east) +(-1pt,-1pt) rectangle +(3pt,3pt);}
        
        \newcommand{\drawbrackets}[4]{
            \draw[thick, rounded corners=0.5pt] ([xshift=-4pt,yshift=4pt]#1.north west) -- +(-4pt,0) |- ([xshift=-4pt,yshift=-4pt]#2.south west);
            \draw[thick, rounded corners=0.5pt] ([xshift=4pt,yshift=4pt]#3.north east) -- +(4pt,0) |- ([xshift=4pt,yshift=-4pt]#4.south east);
        }
        
        
        \begin{scope}[shift={(-1.7,0)}]
        \foreach \i in {1,...,4} {
            \foreach \j in {1,...,4} {
                \node[bx] (m1-\i-\j) at (\j*0.8, -\i*0.8) {};
                \ifnum\i=\j \dotBL{m1-\i-\j}{red} \fi
                \pgfmathtruncatemacro{\im}{\i-1}
                \ifnum\j=\im \dotTR{m1-\i-\j}{red} \fi
            }
        }
        \drawbrackets{m1-1-1}{m1-4-1}{m1-1-4}{m1-4-4}
        \end{scope}
        
        \begin{scope}[shift={(2.5,0)}] 
        \foreach \i in {1,...,4} {
            \foreach \j in {1,...,4} {
                \node[bxB] (m2-\i-\j) at (\j*0.8, -\i*0.8) {};
                \ifnum\i=\j 
                    \dotTL{m2-\i-\j}{red} 
                    \dotBR{m2-\i-\j}{red} 
                    \dotTR{m2-\i-\j}{red} 
                    \dotBL{m2-\i-\j}{red} 
                \fi
            }
        }
        \drawbrackets{m2-1-1}{m2-4-1}{m2-1-4}{m2-4-4}
        \end{scope}
        
        
        \begin{scope}[shift={(6.7,0)}]
        \foreach \i in {1,...,4} {
            \foreach \j in {1,...,4} {
                \node[bx] (m3-\i-\j) at (\j*0.8, -\i*0.8) {};
                \ifnum\i=\j
                    \ifnum\i=4 \dotTR{m3-\i-\j}{red} \else \dotTR{m3-\i-\j}{red} \fi
                \fi
                \pgfmathtruncatemacro{\ip}{\i+1}
                \ifnum\j=\ip \dotBL{m3-\i-\j}{red} \fi
            }
        }
        \drawbrackets{m3-1-1}{m3-4-1}{m3-1-4}{m3-4-4}
        \end{scope}
        
        \draw[decorate, decoration={brace, amplitude=10pt, mirror}, very thick, orange!90!red] 
            ([yshift=-10pt]m1-4-1.south west) -- ([yshift=-10pt]m1-4-1.south west -| m2-4-4.south east);
        
        
        \begin{scope}[shift={(0,-4.5)}]
        
        \begin{scope}[shift={(0,0)}]
        \foreach \i in {1,...,4} {
            \foreach \j in {1,...,4} {
                \node[bx] (m4-\i-\j) at (\j*0.8, -\i*0.8) {};
                \ifnum\i=\j
                    \dotBL{m4-\i-\j}{red}
                    \dotBR{m4-\i-\j}{red}
                \fi
                \pgfmathtruncatemacro{\im}{\i-1}
                \ifnum\j=\im
                    \dotTL{m4-\i-\j}{red}
                    \dotTR{m4-\i-\j}{red}
                \fi
            }
        }
        \drawbrackets{m4-1-1}{m4-4-1}{m4-1-4}{m4-4-4}
        \end{scope}
        
        \begin{scope}[shift={(4.0,0)}]
        \foreach \i in {1,...,4} {
            \foreach \j in {1,...,4} {
                \node[bx] (m5-\i-\j) at (\j*0.8, -\i*0.8) {};
                \ifnum\i=\j
                    \ifnum\i=5 \dotBL{m5-\i-\j}{red} \else \dotTR{m5-\i-\j}{red} \fi
                \fi
                \pgfmathtruncatemacro{\ip}{\i+1}
                \ifnum\j=\ip \dotBL{m5-\i-\j}{red} \fi
            }
        }
        \drawbrackets{m5-1-1}{m5-4-1}{m5-1-4}{m5-4-4}
        \end{scope}
        
        \node[mathnode] at (8.5, -2.1) {$=$};

        \node[mathnode] at (-1.2, -2.1) {$=$};
        
        \begin{scope}[shift={(9.0,0)}]
        \foreach \i in {1,...,4} {
            \foreach \j in {1,...,4} {
                \node[bx] (m6-\i-\j) at (\j*0.8, -\i*0.8) {};
                \ifnum\i=\j
                    \ifnum\i=1
                        \dotBR{m6-\i-\j}{red}
                    \else
                        \dotTL{m6-\i-\j}{red}
                        \dotBR{m6-\i-\j}{red}
                    \fi
                \fi
                \pgfmathtruncatemacro{\im}{\i-1}
                \ifnum\j=\im
                    \dotTR{m6-\i-\j}{orange!90!yellow}
                \fi
                \pgfmathtruncatemacro{\ip}{\i+1}
                \ifnum\j=\ip
                    \dotBL{m6-\i-\j}{orange!90!yellow}
                \fi
            }
        }
        \drawbrackets{m6-1-1}{m6-4-1}{m6-1-4}{m6-4-4}
        \end{scope}
        
        \end{scope} 
        
        \end{tikzpicture}    
    
    \caption{General data dependencies in the term $\widehat{T}_{12}\widehat{T}_{22}^{-1}\widehat{T}_{21}$ in the Schur complement of DDRGF, with $\ell_{1}=4$ and $\ell_{2}=4$.}
    \label{fig:parallel_RGF_Schur_assembly}
\end{figure}
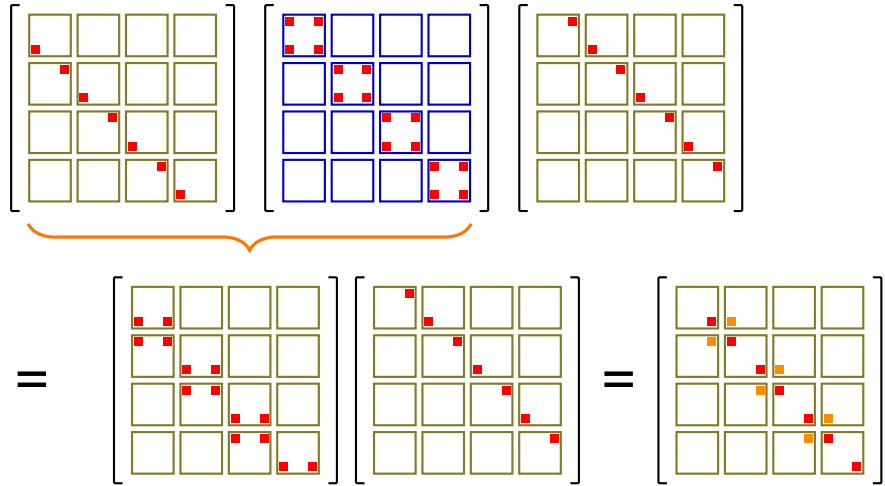
\noindent We can also show that the matrix $\widehat{T}_{S}$ is block tridiagonal in a more visual way, see fig.\ \ref{fig:parallel_RGF_Schur_assembly}. There, the top three matrices represent the product $\widehat{T}_{12} \widehat{T}_{22}^{-1}\widehat{T}_{21}$. Each matrix is split into sub-matrices, with each sub-matrix in turn being composed of various dense sub-blocks; those sub-matrices correspond to the interaction of the various sub-domains illustrated in fig.\ \ref{fig:parallel_RGF_domains}. The dense (red or orange) blocks drawn in fig.\ \ref{fig:parallel_RGF_domains} are the only data involved in the product $\widehat{T}_{12} \widehat{T}_{22}^{-1}\widehat{T}_{21}$; for example, the sub-matrices along the diagonal of $\widehat{T}_{22}^{-1}$ are in principle dense, but the block sparsity in $\widehat{T}_{12}$ and $\widehat{T}_{21}$ leads to only needing the corner blocks from each sub-matrix in $\widehat{T}_{22}^{-1}$. The orange off-diagonal blocks in the bottom-right matrix in fig.\ \ref{fig:parallel_RGF_Schur_assembly} are the effective couplings between neighboring principal layers in $\mathcal{D}_{1}$, which interact via full sub-domains in $\mathcal{D}_{2}$. Due to this distant interaction, one might intuitively expect that those orange blocks could be omitted without much loss in the accuracy of the final result, while enhancing parallelism and reducing the computational cost. Our numerical experiments showed, however, that the omission of the orange blocks is not possible in our case, although this does not exclude the possibility of trimming them for other physical systems.
\noindent Now that we have shown that $\widehat{T}_{S}$ is also block tridiagonal, we can turn to analyzing the data dependencies that emerge in the computation of the various building blocks of $(\widehat{T}^{-1})_{\textrm{b}3\textrm{diag}}$, i.e., $((\widehat{T}^{-1})_{\textrm{b}3\textrm{diag}})_{ij}, \ i,j \in \{1,2\}$; see eq.\ \ref{eq:2x2_explicit_inverse_parallel_DDRGF}. We do this in a visual way, similarly to what was done for $\widehat{T}_{S}$ via fig.\ \ref{fig:parallel_RGF_domains}, from hereon; corresponding formalizations, such as that in prop.\ \ref{propo:DDRGF_Schur_complement_pattern}, can be also constructed, but we omit them here. Let us explore the four building blocks in eq.\ \ref{eq:2x2_explicit_inverse_parallel_DDRGF}, keeping eq.\ \ref{eq:mapping_b3diag_T_to_That} in mind.
\vskip 0.08in
\noindent \textbf{Data dependencies:}
\begin{itemize}
\item \textbf{$((\widehat{T}^{-1})_{\textrm{b}3\textrm{diag}})_{11}$}: from eq.\ 
\ref{eq:2x2_explicit_inverse_parallel_DDRGF}, this amounts to computing nonzero block elements from $\widehat{T}_{S}^{-1}$ that correspond to the nonzero block sparse pattern of $\widehat{T}_{11}$ in eq.\ \ref{eq:T_hat_overall_structure}. It is then sufficient to take $\textrm{b}3\textrm{diag}(\widehat{T}_{S}^{-1})$ for computing this part of $(\widehat{T}^{-1})_{\textrm{b}3\textrm{diag}}$.
\begin{figure}[h]
    \centering
        \begin{tikzpicture}[scale=0.975,
            bx/.style={draw=olive!80!black, thick, minimum size=0.55cm, inner sep=0pt, fill=white},
            bxB/.style={draw=blue!80!black, thick, minimum size=0.55cm, inner sep=0pt, fill=white},
            mathnode/.style={font=\Huge}
        ]
        
        \newcommand{\dotTL}[2]{\fill[#2] ([xshift=2.5pt, yshift=-4.5pt]#1.north west) +(-1pt,-1pt) rectangle +(3pt,3pt);}
        \newcommand{\dotTR}[2]{\fill[#2] ([xshift=-4.5pt, yshift=-4.5pt]#1.north east) +(-1pt,-1pt) rectangle +(3pt,3pt);}
        \newcommand{\dotBL}[2]{\fill[#2] ([xshift=2.5pt, yshift=2.8pt]#1.south west) +(-1pt,-1pt) rectangle +(3pt,3pt);}
        \newcommand{\dotBR}[2]{\fill[#2] ([xshift=-4.5pt, yshift=2.8pt]#1.south east) +(-1pt,-1pt) rectangle +(3pt,3pt);}

        \newcommand{\barL}[2]{\fill[#2] ([xshift=2.5pt, yshift=-13.75pt]#1.north west) +(-1pt,-1pt) rectangle +(3pt,12pt);}
        \newcommand{\barR}[2]{\fill[#2] ([xshift=12.0pt, yshift=-13.75pt]#1.north west) +(-1pt,-1pt) rectangle +(3pt,12pt);}
        
        \newcommand{\barB}[2]{\fill[#2] ([xshift=2.75pt, yshift=-13.75pt]#1.north west) +(-1pt,-1pt) rectangle +(12pt,3pt);}
        \newcommand{\barT}[2]{\fill[#2] ([xshift=2.75pt, yshift=-4.5pt]#1.north west) +(-1pt,-1pt) rectangle +(12pt,3pt);}

        \newcommand{\barDiag}[2]{%
            \draw[line width=2.5pt, #2] 
            ([xshift=2pt, yshift=-2pt]#1.north west) -- ([xshift=-2pt, yshift=2pt]#1.south east);%
        }
        \newcommand{\barDiagL}[2]{%
            \draw[line width=2.5pt, #2] 
            ([xshift=2pt, yshift=-6.5pt]#1.north west) -- ([xshift=-6.5pt, yshift=2pt]#1.south east);%
        }
        \newcommand{\barDiagR}[2]{%
            \draw[line width=2.5pt, #2] 
            ([xshift=6.5pt, yshift=-2pt]#1.north west) -- ([xshift=-1.7pt, yshift=6.5pt]#1.south east);%
        }

        \newcommand{\drawbrackets}[4]{
            \draw[thick, rounded corners=0.5pt] ([xshift=-4pt,yshift=4pt]#1.north west) -- +(-4pt,0) |- ([xshift=-4pt,yshift=-4pt]#2.south west);
            \draw[thick, rounded corners=0.5pt] ([xshift=4pt,yshift=4pt]#3.north east) -- +(4pt,0) |- ([xshift=4pt,yshift=-4pt]#4.south east);
        }


        \begin{scope}[shift={(-5.3,0)}] 
        \foreach \i in {1,...,4} {
            \foreach \j in {1,...,4} {
                \node[bxB] (m1-\i-\j) at (\j*0.8, -\i*0.8) {};
                \ifnum\i=\j 
                    \ifnum\i=4
                        \barL{m1-\i-\j}{red}
                    \else
                        \barL{m1-\i-\j}{red}
                        \barR{m1-\i-\j}{red}
                    \fi
                \fi
            }
        }
        \drawbrackets{m1-1-1}{m1-4-1}{m1-1-4}{m1-4-4}
        \end{scope}
                
        \begin{scope}[shift={(-1.7,0)}]
        \foreach \i in {1,...,4} {
            \foreach \j in {1,...,4} {
                \node[bx] (m2-\i-\j) at (\j*0.8, -\i*0.8) {};
                \ifnum\i=\j
                    \ifnum\i=4 \dotBR{m2-\i-\j}{red} \else \dotTR{m2-\i-\j}{red} \fi
                \fi
                \pgfmathtruncatemacro{\ip}{\i+1}
                \ifnum\j=\ip \dotBL{m2-\i-\j}{red} \fi
            }
        }
        \drawbrackets{m2-1-1}{m2-4-1}{m2-1-4}{m2-4-4}
        \end{scope}

        \begin{scope}[shift={(1.9,0)}] 
        \foreach \i in {1,...,4} {
            \foreach \j in {1,...,4} {
                \node[bxB] (m3-\i-\j) at (\j*0.8, -\i*0.8) {};
                \ifnum\i=\j
                    \ifnum\i=1
                        \dotBR{m3-\i-\j}{red} 
                    \else
                        \dotTL{m3-\i-\j}{red} 
                        \dotBR{m3-\i-\j}{red} 
                    \fi
                \fi
                \pgfmathtruncatemacro{\ip}{\i+1}
                \ifnum\j=\ip \dotBL{m3-\i-\j}{red} \fi
                \pgfmathtruncatemacro{\im}{\i-1}
                \ifnum\j=\im \dotTR{m3-\i-\j}{red} \fi
            }
        }
        \drawbrackets{m3-1-1}{m3-4-1}{m3-1-4}{m3-4-4}
        \end{scope}

        
        \begin{scope}[shift={(5.5,0)}]
        \foreach \i in {1,...,4} {
            \foreach \j in {1,...,4} {
                \node[bx] (m4-\i-\j) at (\j*0.8, -\i*0.8) {};
                \ifnum\i=\j \dotBL{m4-\i-\j}{red} \fi
                \pgfmathtruncatemacro{\im}{\i-1}
                \ifnum\j=\im \dotTR{m4-\i-\j}{red} \fi
            }
        }
        \drawbrackets{m4-1-1}{m4-4-1}{m4-1-4}{m4-4-4}
        \end{scope}
        
        \begin{scope}[shift={(9.1,0)}] 
        \foreach \i in {1,...,4} {
            \foreach \j in {1,...,4} {
                \node[bxB] (m5-\i-\j) at (\j*0.8, -\i*0.8) {};
                \ifnum\i=\j 
                    \ifnum\i=4
                        \barT{m5-\i-\j}{red}
                    \else
                        \barT{m5-\i-\j}{red}
                        \barB{m5-\i-\j}{red}
                    \fi
                \fi
            }
        }
        \drawbrackets{m5-1-1}{m5-4-1}{m5-1-4}{m5-4-4}
        \end{scope}

        \draw[decorate, decoration={brace, amplitude=10pt, mirror}, very thick, orange!90!red] 
            ([yshift=-10pt]m1-4-1.south west) -- ([yshift=-10pt]m1-4-1.south west -| m2-4-4.south east);
        \draw[decorate, decoration={brace, amplitude=10pt, mirror}, very thick, orange!90!red] 
            ([yshift=-10pt]m4-4-1.south west) -- ([yshift=-10pt]m4-4-1.south west -| m5-4-4.south east);

        
        \begin{scope}[shift={(-2.5,-4.5)}]
        \foreach \i in {1,...,4} {
            \foreach \j in {1,...,4} {
                \node[bx] (m1-\i-\j) at (\j*0.8, -\i*0.8) {};
                \ifnum\i=\j \barR{m1-\i-\j}{red} \fi
                \pgfmathtruncatemacro{\ip}{\i+1}
                \ifnum\j=\ip \barL{m1-\i-\j}{red} \fi
            }
        }
        \drawbrackets{m1-1-1}{m1-4-1}{m1-1-4}{m1-4-4}
        \end{scope}
        
        \begin{scope}[shift={(1.9,-4.5)}] 
        \foreach \i in {1,...,4} {
            \foreach \j in {1,...,4} {
                \node[bxB] (m2-\i-\j) at (\j*0.8, -\i*0.8) {};
                \ifnum\i=\j
                    \ifnum\i=1
                        \dotBR{m2-\i-\j}{red} 
                    \else
                        \dotTL{m2-\i-\j}{red} 
                        \dotBR{m2-\i-\j}{red} 
                    \fi
                \fi
                \pgfmathtruncatemacro{\ip}{\i+1}
                \ifnum\j=\ip \dotBL{m2-\i-\j}{red} \fi
                \pgfmathtruncatemacro{\im}{\i-1}
                \ifnum\j=\im \dotTR{m2-\i-\j}{red} \fi
            }
        }
        \drawbrackets{m2-1-1}{m2-4-1}{m2-1-4}{m2-4-4}
        \end{scope}
        
        \node[mathnode] at (-3.0, -6.6) {$=$};

        
        \begin{scope}[shift={(6.3,-4.5)}]
        \foreach \i in {1,...,4} {
            \foreach \j in {1,...,4} {
                \node[bx] (m3-\i-\j) at (\j*0.8, -\i*0.8) {};
                \ifnum\i=\j
                    \barB{m3-\i-\j}{red}
                \fi
                \pgfmathtruncatemacro{\im}{\i-1}
                \ifnum\j=\im \barT{m3-\i-\j}{red} \fi
            }
        }
        \drawbrackets{m3-1-1}{m3-4-1}{m3-1-4}{m3-4-4}
        \end{scope}
        
        \draw[decorate, decoration={brace, amplitude=10pt, mirror}, very thick, orange!90!red] 
            ([yshift=-10pt]m1-4-1.south west) -- ([yshift=-10pt]m1-4-1.south west -| m2-4-4.south east);
        
        
        \begin{scope}[shift={(0,-9.0)}]
        
        \begin{scope}[shift={(0,0)}]
        \foreach \i in {1,...,4} {
            \foreach \j in {1,...,4} {
                \node[bx] (m1-\i-\j) at (\j*0.8, -\i*0.8) {};
                \ifnum\i=\j \barR{m1-\i-\j}{red} \fi
                \pgfmathtruncatemacro{\ip}{\i+1}
                \ifnum\j=\ip \barL{m1-\i-\j}{red} \fi
            }
        }
        \drawbrackets{m1-1-1}{m1-4-1}{m1-1-4}{m1-4-4}
        \end{scope}
        
        \begin{scope}[shift={(4.0,0)}]
        \foreach \i in {1,...,4} {
            \foreach \j in {1,...,4} {
                \node[bx] (m2-\i-\j) at (\j*0.8, -\i*0.8) {};
                \ifnum\i=\j
                    \barB{m2-\i-\j}{red}
                \fi
                \pgfmathtruncatemacro{\im}{\i-1}
                \ifnum\j=\im \barT{m2-\i-\j}{red} \fi
            }
        }
        \drawbrackets{m2-1-1}{m2-4-1}{m2-1-4}{m2-4-4}
        \end{scope}
        
        \node[mathnode] at (-0.6, -2.1) {$=$};

        \node[mathnode] at (8.5, -2.1) {$\rightarrow$};
        
        \begin{scope}[shift={(9.0,0)}]
        \foreach \i in {1,...,4} {
            \foreach \j in {1,...,4} {
                \node[bx] (m3-\i-\j) at (\j*0.8, -\i*0.8) {};
                \ifnum\i=\j
                    \barDiag{m3-\i-\j}{red}
                    \barDiagL{m3-\i-\j}{red}
                    \barDiagR{m3-\i-\j}{red}
                \fi
            }
        }
        \drawbrackets{m3-1-1}{m3-4-1}{m3-1-4}{m3-4-4}
        \end{scope}

        \end{scope} 
        \end{tikzpicture}
    \caption{Dependencies in the term $\widehat{T}_{22}^{-1}\widehat{T}_{21}\widehat{T}_{S}^{-1}\widehat{T}_{12}\widehat{T}_{22}^{-1}$ in DDRGF, with $\ell_{1}=4$ and $\ell_{2}=4$.}\label{fig:parallel_RGF_term_22}
\end{figure}
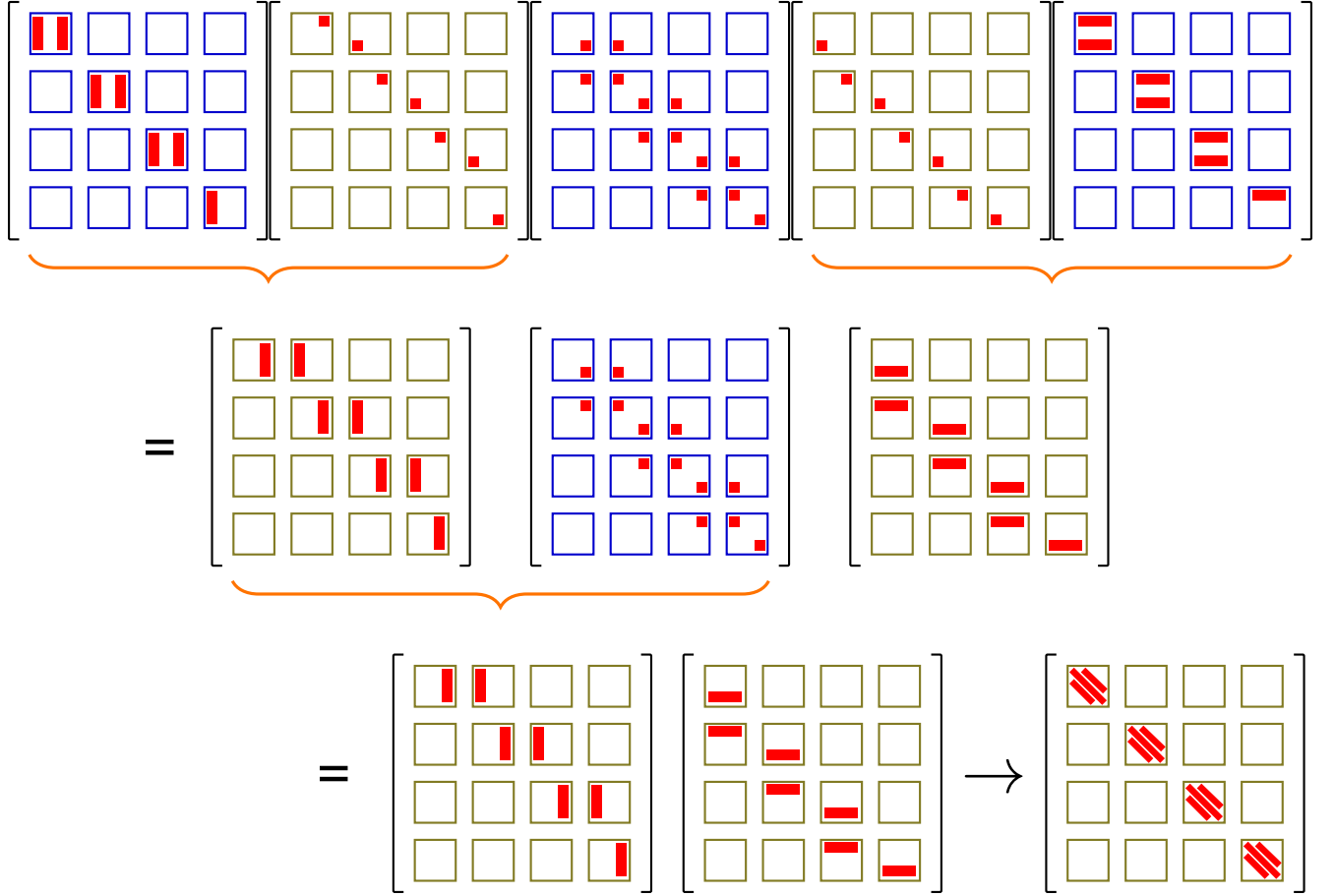
\item \textbf{$((\widehat{T}^{-1})_{\textrm{b}3\textrm{diag}})_{22}$}: this is a more complicated term than the previous one. We need the part of $\widehat{T}_{22}^{-1} + \widehat{T}_{22}^{-1}\widehat{T}_{21}\widehat{T}_{S}^{-1}\widehat{T}_{12}\widehat{T}_{22}^{-1}$ matching the nonzero block sparse pattern of $\widehat{T}_{22}$. Let us skip the $\widehat{T}_{22}^{-1}$ part, which should be easy to infer at this point. For the second term, $\widehat{T}_{22}^{-1}\widehat{T}_{21}\widehat{T}_{S}^{-1}\widehat{T}_{12}\widehat{T}_{22}^{-1}$, we can understand the data dependencies via fig.\ \ref{fig:parallel_RGF_term_22}, where the red bars represent full block columns or full block rows. The arrow pointing towards the bottom right matrix indicates that we want to extract a pattern matching that of $\widehat{T}_{22}$. This term makes it clear that computing $\textrm{b}3\textrm{diag}(\widehat{T}_{S}^{-1})$ is not only sufficient but necessary when dealing with both $((\widehat{T}^{-1})_{\textrm{b}3\textrm{diag}})_{11}$ and $((\widehat{T}^{-1})_{\textrm{b}3\textrm{diag}})_{22}$. Fig.\ \ref{fig:parallel_RGF_term_22} also tells us what is needed from $\widehat{T}_{22}^{-1}$ to compute this part of $(\widehat{T}^{-1})_{\textrm{b}3\textrm{diag}}$, being clear from the data dependencies there that we need to compute the following elements ($i=1,2,...,\ell_{2}$):
$$ \textrm{b}3\textrm{diag}\left((\widehat{T}_{22}^{-1})_{ii}\right), $$
$$ \left((\widehat{T}_{22}^{-1})_{ii}\right)_{ab}, \ b = 1, a = 2,3,...,|\mathcal{D}_{2}^{i}| , \ \ \left((\widehat{T}_{22}^{-1})_{ii}\right)_{ab}, \ b = |\mathcal{D}_{2}^{i}|, a = 1,2,3,...,|\mathcal{D}_{2}^{i}|-1 , $$
$$ \left((\widehat{T}_{22}^{-1})_{ii}\right)_{ab}, \ a = 1, b = 2,3,...,|\mathcal{D}_{2}^{i}|-1 , \ \ \left((\widehat{T}_{22}^{-1})_{ii}\right)_{ab}, \ a = |\mathcal{D}_{2}^{i}|, b = 2,3,...,|\mathcal{D}_{2}^{i}|-1 , $$
That is, within each sub-domain of $\mathcal{D}_{2}$, we take the block tridiagonal part of the inverse plus any missing blocks to form the internal halos. Finally, due to the structural symmetry of this term, the second row of matrices in fig.\ \ref{fig:parallel_RGF_term_22} can be grouped either from the left or from the right, i.e., that operation is associative.
\item \textbf{$((\widehat{T}^{-1})_{\textrm{b}3\textrm{diag}})_{12}$ and $((\widehat{T}^{-1})_{\textrm{b}3\textrm{diag}})_{21}$}: the two off-diagonal parts of $(\widehat{T}^{-1})_{\textrm{b}3\textrm{diag}}$ are structurally symmetric, therefore we illustrate one of them only. Fig.\ \ref{fig:parallel_RGF_term_12} displays the data dependencies in this case. The data dependencies are simpler in these two cases, compared to $((\widehat{T}^{-1})_{\textrm{b}3\textrm{diag}})_{11}$ and $((\widehat{T}^{-1})_{\textrm{b}3\textrm{diag}})_{22}$.
\end{itemize}
\begin{figure}
    \centering
        \begin{tikzpicture}[scale=0.9,
            bx/.style={draw=olive!80!black, thick, minimum size=0.55cm, inner sep=0pt, fill=white},
            bxB/.style={draw=blue!80!black, thick, minimum size=0.55cm, inner sep=0pt, fill=white},
            mathnode/.style={font=\Huge}
        ]
        
        \newcommand{\dotTL}[2]{\fill[#2] ([xshift=2.5pt, yshift=-4.5pt]#1.north west) +(-1pt,-1pt) rectangle +(3pt,3pt);}
        \newcommand{\dotTR}[2]{\fill[#2] ([xshift=-4.5pt, yshift=-4.5pt]#1.north east) +(-1pt,-1pt) rectangle +(3pt,3pt);}
        \newcommand{\dotBL}[2]{\fill[#2] ([xshift=2.5pt, yshift=2.8pt]#1.south west) +(-1pt,-1pt) rectangle +(3pt,3pt);}
        \newcommand{\dotBR}[2]{\fill[#2] ([xshift=-4.5pt, yshift=2.8pt]#1.south east) +(-1pt,-1pt) rectangle +(3pt,3pt);}

        \newcommand{\barL}[2]{\fill[#2] ([xshift=2.5pt, yshift=-13.75pt]#1.north west) +(-1pt,-1pt) rectangle +(3pt,12pt);}
        \newcommand{\barR}[2]{\fill[#2] ([xshift=12.0pt, yshift=-13.75pt]#1.north west) +(-1pt,-1pt) rectangle +(3pt,12pt);}
        
        \newcommand{\barB}[2]{\fill[#2] ([xshift=2.75pt, yshift=-13.75pt]#1.north west) +(-1pt,-1pt) rectangle +(12pt,3pt);}
        \newcommand{\barT}[2]{\fill[#2] ([xshift=2.75pt, yshift=-4.5pt]#1.north west) +(-1pt,-1pt) rectangle +(12pt,3pt);}

        \newcommand{\barDiag}[2]{%
            \draw[line width=2.5pt, #2] 
            ([xshift=2pt, yshift=-2pt]#1.north west) -- ([xshift=-2pt, yshift=2pt]#1.south east);%
        }
        \newcommand{\barDiagL}[2]{%
            \draw[line width=2.5pt, #2] 
            ([xshift=2pt, yshift=-6.5pt]#1.north west) -- ([xshift=-6.5pt, yshift=2pt]#1.south east);%
        }
        \newcommand{\barDiagR}[2]{%
            \draw[line width=2.5pt, #2] 
            ([xshift=6.5pt, yshift=-2pt]#1.north west) -- ([xshift=-1.7pt, yshift=6.5pt]#1.south east);%
        }

        \newcommand{\drawbrackets}[4]{
            \draw[thick, rounded corners=0.5pt] ([xshift=-4pt,yshift=4pt]#1.north west) -- +(-4pt,0) |- ([xshift=-4pt,yshift=-4pt]#2.south west);
            \draw[thick, rounded corners=0.5pt] ([xshift=4pt,yshift=4pt]#3.north east) -- +(4pt,0) |- ([xshift=4pt,yshift=-4pt]#4.south east);
        }


        \begin{scope}[shift={(-2.3,0)}] 
        \foreach \i in {1,...,4} {
            \foreach \j in {1,...,4} {
                \node[bxB] (m1-\i-\j) at (\j*0.8, -\i*0.8) {};
                \ifnum\i=\j 
                    \ifnum\i=4
                        \dotTL{m1-\i-\j}{red}
                    \else
                        \dotTL{m1-\i-\j}{red}
                        \dotBR{m1-\i-\j}{red}
                    \fi
                \fi
            }
        }
        \drawbrackets{m1-1-1}{m1-4-1}{m1-1-4}{m1-4-4}
        \end{scope}
                
        \begin{scope}[shift={(1.8,0)}]
        \foreach \i in {1,...,4} {
            \foreach \j in {1,...,4} {
                \node[bx] (m2-\i-\j) at (\j*0.8, -\i*0.8) {};
                \ifnum\i=\j
                    \ifnum\i=4 \dotBR{m2-\i-\j}{red} \else \dotTR{m2-\i-\j}{red} \fi
                \fi
                \pgfmathtruncatemacro{\ip}{\i+1}
                \ifnum\j=\ip \dotBL{m2-\i-\j}{red} \fi
            }
        }
        \drawbrackets{m2-1-1}{m2-4-1}{m2-1-4}{m2-4-4}
        \end{scope}

        \begin{scope}[shift={(5.9,0)}] 
        \foreach \i in {1,...,4} {
            \foreach \j in {1,...,4} {
                \node[bxB] (m3-\i-\j) at (\j*0.8, -\i*0.8) {};
                \ifnum\i=\j
                    \ifnum\i=1
                        \dotBR{m3-\i-\j}{red} 
                    \else
                        \dotTL{m3-\i-\j}{red} 
                        \dotBR{m3-\i-\j}{red} 
                    \fi
                \fi
            }
        }
        \drawbrackets{m3-1-1}{m3-4-1}{m3-1-4}{m3-4-4}
        \end{scope}

        \draw[decorate, decoration={brace, amplitude=10pt, mirror}, very thick, orange!90!red] 
            ([yshift=-10pt]m1-4-1.south west) -- ([yshift=-10pt]m1-4-1.south west -| m2-4-4.south east);

        
        \begin{scope}[shift={(-0.5,-4.5)}]
        \foreach \i in {1,...,4} {
            \foreach \j in {1,...,4} {
                \node[bx] (m1-\i-\j) at (\j*0.8, -\i*0.8) {};
                \ifnum\i=\j \dotTR{m1-\i-\j}{red} \fi
                \pgfmathtruncatemacro{\ip}{\i+1}
                \ifnum\j=\ip \dotBL{m1-\i-\j}{red} \fi
            }
        }
        \drawbrackets{m1-1-1}{m1-4-1}{m1-1-4}{m1-4-4}
        \end{scope}
        
        \begin{scope}[shift={(3.6,-4.5)}] 
        \foreach \i in {1,...,4} {
            \foreach \j in {1,...,4} {
                \node[bxB] (m2-\i-\j) at (\j*0.8, -\i*0.8) {};
                \ifnum\i=\j
                    \ifnum\i=1
                        \dotBR{m2-\i-\j}{red} 
                    \else
                        \dotTL{m2-\i-\j}{red} 
                        \dotBR{m2-\i-\j}{red} 
                    \fi
                \fi
            }
        }
        \drawbrackets{m2-1-1}{m2-4-1}{m2-1-4}{m2-4-4}
        \end{scope}
        
        \node[mathnode] at (-1.0, -6.6) {$=$};

        \node[mathnode] at (8.5, -6.6) {$\rightarrow$};
        
        \begin{scope}[shift={(9.0,-4.5)}]
        \foreach \i in {1,...,4} {
            \foreach \j in {1,...,4} {
                \node[bx] (m3-\i-\j) at (\j*0.8, -\i*0.8) {};
                \ifnum\i=\j \dotTR{m3-\i-\j}{red} \fi
                \pgfmathtruncatemacro{\ip}{\i+1}
                \ifnum\j=\ip \dotBL{m3-\i-\j}{red} \fi
            }
        }
        \drawbrackets{m3-1-1}{m3-4-1}{m3-1-4}{m3-4-4}
        \end{scope}

        \end{tikzpicture}

        \caption{Data dependencies in the computation of $((\widehat{T}^{-1})_{\textrm{b}3\textrm{diag}})_{21}$ from the term  $\widehat{T}_{22}^{-1}\widehat{T}_{21}\widehat{T}_{S}^{-1}$ in DDRGF with $\ell_{1}=4$ and $\ell_{2}=4$.}
    \label{fig:parallel_RGF_term_12}
\end{figure}
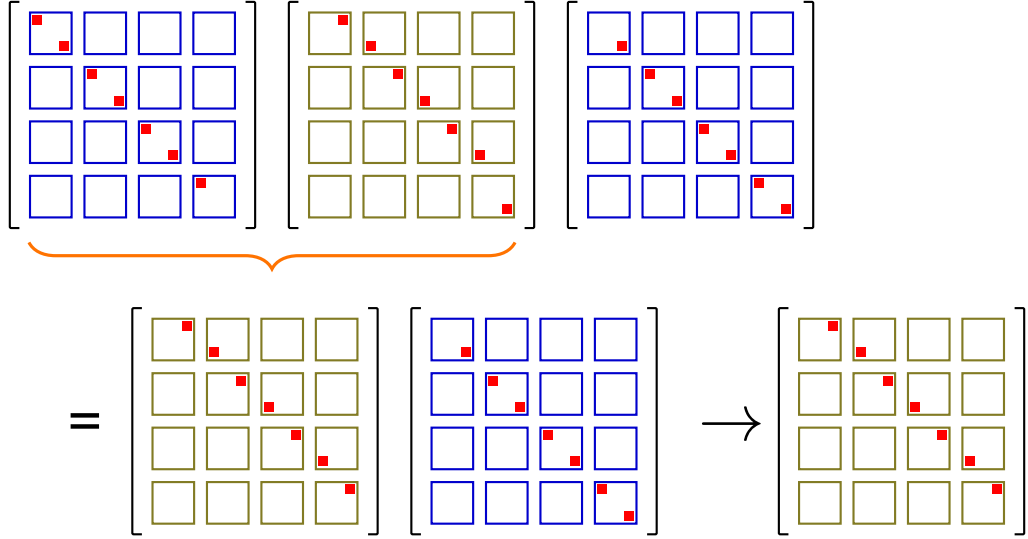
\vskip 0.08in
\noindent As we can see, the data dependencies in DDRGF are non-trivial. To clarify this further, fig.\ \ref{fig:nonzero_T22inv_pattern_needed} shows, for a particular example where $i$ is a sub-domain within $\mathcal{D}_{2}$ and such that $|\mathcal{D}_{2}^{i}| = 8$ is the number of principal layers in that sub-domain, the blocks needed from the inverse of $(\widehat{T}_{22})_{ii}$ to compute the wanted solution $\mathrm{b3diag}(T^{-1})$. Pseudo-code for DDRGF is presented in alg.\ \ref{alg:ddrgf}, where we have clearly outlined the parallel regions.
\begin{algorithm}
\caption{DDRGF for block tridiagonal systems}
\label{alg:ddrgf}
\begin{algorithmic}[1]
\Require $T$ (block tridiagonal matrix), $\mathcal{D}_1, \mathcal{D}_2$ (domain decompositions)
\Ensure $\text{b3diag}(T^{-1})$
\State \textbf{// Domain Reordering}
\State $\hat{T} \gets R_3 T R_3^T = \begin{pmatrix} \hat{T}_{11} & \hat{T}_{12} \\ \hat{T}_{21} & \hat{T}_{22} \end{pmatrix}$

\State \textbf{// Step 1: Invert $\mathcal{D}_2$ sub-domains (Parallel over $n_{\text{tasks}}$)}
\ForAll{sub-domain $i \in \mathcal{D}_2$} \Comment{Execute concurrently}
    \State Compute $(\hat{T}_{22}^{-1})_{ii}$ \Comment{Requires extended RGF for internal halos}
\EndFor

\State \textbf{// Step 2: Build the Schur Complement (Parallel)}
\State $\hat{T}_S \gets \hat{T}_{11} - \hat{T}_{12} \hat{T}_{22}^{-1} \hat{T}_{21}$

\State \textbf{// Step 3: Invert the Schur Complement (Recursive or Sequential)}
\State $\hat{T}_S^{-1} \gets \text{b3diag}(\hat{T}_S^{-1})$ \Comment{Via RGF (Alg 1) or recursive DDRGF}

\State \textbf{// Step 4: Compute final blocks of $\hat{T}^{-1}_{\text{b3diag}}$ (Parallel)}
\State $((\hat{T}^{-1})_{\text{b3diag}})_{11} \gets \hat{T}_S^{-1}$
\ForAll{sub-domain $i \in \mathcal{D}_2$}
    \State $((\hat{T}^{-1})_{\text{b3diag}})_{12} \gets -(\hat{T}_S^{-1} \hat{T}_{12} \hat{T}_{22}^{-1})_{\text{b3diag}}$
    \State $((\hat{T}^{-1})_{\text{b3diag}})_{21} \gets -(\hat{T}_{22}^{-1} \hat{T}_{21} \hat{T}_S^{-1})_{\text{b3diag}}$
    \State $((\hat{T}^{-1})_{\text{b3diag}})_{22} \gets (\hat{T}_{22}^{-1} + \hat{T}_{22}^{-1} \hat{T}_{21} \hat{T}_S^{-1} \hat{T}_{12} \hat{T}_{22}^{-1})_{\text{b3diag}}$
\EndFor

\State \textbf{// Step 5: Inverse Reordering}
\State $\text{b3diag}(T^{-1}) \gets R_3^T ((\hat{T}^{-1})_{\text{b3diag}}) R_3$
\end{algorithmic}
\end{algorithm}
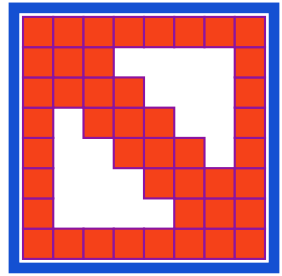
\begin{figure}
    \centering
        \begin{tikzpicture}[scale=0.4]
            \definecolor{boxred}{RGB}{245, 65, 25}
            \definecolor{linepurple}{RGB}{140, 20, 160}
            \definecolor{frameblue}{RGB}{20, 80, 210}
        
            \draw[frameblue, line width=4pt] (-0.3, -0.3) rectangle (8.3, 8.3);
        
            \foreach \x in {0,...,7} {
                \foreach \y in {0,...,7} {
                    
                    \pgfmathsetmacro{\iswhite}{
                        (\x==1 && \y==2) || 
                        (\x==2 && \y==2) || 
                        (\x==3 && \y==2) || 
                        (\x==3 && \y==1) ||
                        (\x==2 && \y==1) ||
                        (\x==1 && \y==1) ||
                        (\x==4 && \y==1) ||
                        (\x==1 && \y==4) ||
                        (\x==6 && \y==3) ||
                        (\x==6 && \y==6) ||
                        (\x==5 && \y==6) ||
                        (\x==4 && \y==6) ||
                        (\x==3 && \y==6) ||
                        (\x==1 && \y==3) || 
                        (\x==2 && \y==3) || 
                        (\x==5 && \y==4) || 
                        (\x==6 && \y==4) || 
                        (\x==4 && \y==5) || 
                        (\x==5 && \y==5) || 
                        (\x==6 && \y==5) ? 1 : 0
                    }
                    
                    \ifnum\iswhite=0
                        \draw[linepurple, thick, fill=boxred] (\x, \y) rectangle ++(1, 1);
                    \fi
                }
            }
        \end{tikzpicture}
    \caption{For a particular sub-domain $i$ in $\mathcal{D}_{2}$, with $|\mathcal{D}_{2}^{i}| = 8$, these are the nonzero blocks needed from $(\widehat{T}_{22})_{ii}^{-1}$ to compute the wanted solution $\mathrm{b3diag}(T^{-1})$.}
    \label{fig:nonzero_T22inv_pattern_needed}
\end{figure}
\vskip 0.08in
\noindent\textbf{Cost model:} to get a better understanding of the expected behavior of DDRGF, we write a cost model for it. Just as with RGF, we express the execution cost relative to a single GEMM operation. To this end, first note that a particular difference of DDRGF compared to RGF is that, for the sub-domains in $\mathcal{D}_{2}$ (see fig.\ \ref{fig:parallel_RGF_domains}), we need to compute fuller outputs when RGF is applied to their corresponding block tridiagonal matrices. Namely, we need to compute more than the block tridiagonal of $\widehat{T}_{22}^{-1}$, as discussed above (see fig.\ \ref{fig:nonzero_T22inv_pattern_needed}).
The number of extra GEMMs needed to compute the full inverse with RGF when $\ell\leq4$ is:
$$ r_{RGF}^{extra}(\ell) =
\begin{cases} 
 0 & \text{if } \ell \leq 2 \\
 1{\cdot}2       & \text{if } \ell = 3 \\
 3{\cdot}2       & \text{if } \ell = 4
\end{cases} . $$
To rigorously define the cost of DDRGF, we must clarify its recursive nature. DDRGF is applied layer by layer (or ``level by level''). At any given recursion level $k$, the active physical system of size $\ell^{(k)}$ is partitioned into non-overlapping, repeating domain pairs $(\mathcal{D}_1^i, \mathcal{D}_2^i)$. To guarantee perfect load balancing among parallel workers, we strictly enforce that all $\mathcal{D}_{1}$ sub-domains at level $k$ contain exactly $s_{1}^{(k)}$ principal layers, and all $\mathcal{D}_{2}$ sub-domains contain exactly $s_{2}^{(k)}$ principal layers. For the purpose of this analytical cost model and the subsequent numerical experiments, we assume a geometrically homogeneous system where the dimension of the principal layers ($b_s$) remains constant across the device. Because the structure is uniform, the linear algebra operations executed within each domain pair are identical, allowing us to encapsulate the workload of processing one such domain pair as a single standard ``task''. While realistic physical systems can often be structurally heterogeneous with varying cross-sections, we consider uniform systems here for theoretical simplicity. Nevertheless, the cost model derived below is fundamentally adaptable: for inhomogeneous systems, the algorithm can be straightforwardly modified to compute layer-specific block weights to find the most adequate, load-balanced partitioning. Therefore, for our uniform analysis, the total number of independent tasks generated at level $k$ is simply the total number of layers divided by the footprint of one domain pair: $n_{tasks}^{(k)} = \lceil \ell^{(k)} / (s_{1}^{(k)} + s_{2}^{(k)}) \rceil$.\\
To model the parallel execution time for a single recursion level, we drop the superscript $(k)$ for the rest of this section for notational simplicity. We assume the $n_{tasks}$ independent tasks are distributed evenly across $n_{threads}$ hardware threads. Under an ideal load-balancing assumption, the wall-time cost scales by a factor of approximately $1/n_{threads}$ relative to the total accumulated work.\\
Consequently, the contribution to the total parallel cost of DDRGF due to the computation of $\widehat{T}_{22}^{-1}$ is quantified as:
$$ r_{22}^{inv}(s_{2},b_{s}) := \frac{n_{tasks}}{n_{threads}} (r_{RGF}(s_{2},b_{s}) + r_{RGF}^{extra}(s_{2})) . $$
Next, we need to compute both $\widehat{T}_{22}^{-1}\widehat{T}_{21}$ and $\widehat{T}_{12}\widehat{T}_{22}^{-1}$, one of which will be used to build DDRGF's Schur complement $\widehat{T}_{S}$, while the other one will be used in a later computation. These two costs amount to:
$$ r_{11}^{12}(s_{2}) = r_{11}^{21}(s_{2}) := 2s_{2}\frac{n_{tasks}}{n_{threads}} , $$
and then to finish building the Schur complement we need the following extra cost:
$$ r_{S} = 4\frac{n_{tasks}}{n_{threads}} . $$
The cost for the inverse of the Schur complement is less simple to write, as we might choose to use sequential RGF directly to get the block tridiagonal of its inverse, or we might call DDRGF recursively (moving to level $k+1$); in either case, we call this subsequent contribution $r_{S}^{inv}$. On the other hand, for fully building $((\widehat{T}^{-1})_{\textrm{b}3\textrm{diag}})_{12}$ and $((\widehat{T}^{-1})_{\textrm{b}3\textrm{diag}})_{21}$ we need some further computations:
$$ r_{12}(s_{2}) := 4s_{2}\frac{n_{tasks}}{n_{threads}}, \ r_{21} := 4\frac{n_{tasks}}{n_{threads}} . $$
Finally, $((\widehat{T}^{-1})_{\textrm{b}3\textrm{diag}})_{22}$ requires the following extra operations:
$$ r_{22}(s_{2}) := \frac{n_{tasks}}{n_{threads}}(4(s_{2}-1) + 2s_{2}) . $$
The total parallel cost for one level of DDRGF is the addition of all of the above:
\begin{equation}\label{eq:cost_model_DDRGF}
    r_{DDRGF}(s_{2},b_{s}) = r_{22}^{inv}(s_{2},b_{s}) + (r_{11}^{12}(s_{2})+r_{11}^{21}(s_{2})) + r_{S} + (r_{12}(s_{2}) + r_{21}) + r_{22}(s_{2}) + r_{S}^{inv}(\ell_{S},b_{s}) ,
\end{equation}
where the number of principal layers composing the next-level Schur complement $\widehat{T}_{S}$ is
$$ \ell_{S} = n_{tasks} \cdot s_{1} = \left\lceil\frac{\ell}{s_{1} + s_{2}} \right\rceil s_{1}. $$
This cost model allows us to explicitly analyze the concurrency-to-complexity trade-off in DDRGF. The larger the value of $s_{2}$ is, the more aggressive the domain coarsening becomes. This rapidly shrinks $\ell_{S}$, leading to a much cheaper Schur complement inversion ($r_{S}^{inv}$), and effectively reducing the necessary depth of the recursion. However, this mathematical compression comes with two penalties: first, a smaller $\ell_S$ inherently generates fewer concurrent tasks, directly reducing the available parallelism. Second, enlarging $s_{2}$ thickens the local dense matrix operations within the sub-domains, drastically increasing the local sequential cost terms (particularly $r_{22}^{inv}$). To achieve optimal performance, one must carefully tune the sequence of $s_2$ values across the recursion levels to balance these opposing forces.

\section{Numerical experiments}\label{sec:NumericalExperiments}
We now present numerical experiments studying and comparing the scalability, and change under several parameters, of the various methods presented and developed in the previous sections. The analysis is done with respect to four parameters: number of principal layers $\ell$, block size $b_{s}$, number of concurrent tasks $n_{tasks}$ and number of parallelizing threads $n_{threads}$ (the latter two in DDRGF only, see sect.\ \ref{sec:DDRGF}). The functional behavior of all of these methods with respect to $b_{s}$ is encapsulated through the analysis in app.\ \ref{app:microbenchmarking}, therefore we focus in this section on the dependence on the other three parameters. The focus of most of the developments presented in the previous sections are on block tridiagonal matrices, except sect.\ \ref{sec:RGF_block_ndiagonal}. Consequently, most of the numerical experiments follow the same logic focusing on tridiagonal matrices with the exception of sect.\ \ref{sec:results_rgf_ngt3}, where we extend RGF to block $n$-diagonal with $n>3$, i.e., beyond block tridiagonal. Therefore, with the formulations and developments presented here, we have covered all possible cases (i.e., sequential, parallel, general $n$), leading to a clear and unified framework for writing general implementations for Dyson for block $n$-diagonal systems for any $n=3, 5, 7, \dots$; the parallel implementation for $n>3$ will be part of a future release of our \texttt{LibNEGF.jl} library.\\
Our implementation is generic in the underlying data type, namely, one can choose either real or complex data, using double, single or even half precision\footnote{Note that, in order for the half precision execution to perform well, there needs to be support not only from the hardware but also from the underlying BLAS.}. Our numerical experiments consistently showed a factor of approximately 1.5x-1.7x reduction in the execution time when moving from complex-double down to complex-single. To keep a cleaner outlining of the results, all of the runs presented here are only for complex double data. All of our numerical experiments are run on a single node of the JUWELS Cluster Module, consisting of two Intel Xeon Platinum 8168 (Skylake-SP) CPUs (each having 24 cores) operating at 2.7 GHz, with 96 (12${\times}$8) GB DDR4. Because this architecture features two AVX-512 FMA units per core, it executes 32 FLOPs per cycle (in double precision), establishing a clear baseline for the Theoretical Peak Performance (TPP) of our benchmarks.

\subsection{RGF ($n=3$)}\label{sec:results_rgf_neq3}
The execution time of RGF, which was formulated in sect.\ \ref{sec:RGF_block_tridiagonal} by means of domain decomposition concepts, is fundamentally driven by two parameters: the number of principal layers ($\ell$) and the block size ($b_{s}$). We built diagonally-dominant synthetic matrices with underlying random data to verify that our implementation of RGF behaves as expected when both of these parameters are varied. Fig.\ \ref{fig:execTimeSequentialRGF} illustrates the resulting behavior.\\
In the left panel, the horizontal axis represents the number of principal layers $\ell$, with each curve corresponding to a different fixed block size $b_{s}$. As strictly predicted by the cost model in eq.\ \ref{eq:cost_model_RGF}, our implementation of RGF scales perfectly linearly with $\ell$.\\
Conversely, the right panel of fig.\ \ref{fig:execTimeSequentialRGF} isolates the scaling behavior with respect to the block size $b_s$ for fixed values of $\ell$. Because the RGF execution time is overwhelmingly dominated by dense matrix multiplications (GEMMs) alongside LU factorizations and triangular solves, the theoretical algorithmic complexity is strictly $O(b_s^3)$. This is reflected in the empirical data: as the block size increases towards $b_s = 512$, the measured execution time (most notably for the largest system, $\ell=160$) gets closer to the ideal cubic scaling reference line. This behavior is in agreement with the microbenchmarking observations detailed in app.\ \ref{app:microbenchmarking}, which demonstrate that these fundamental Level-3 BLAS kernels overcome memory-latency bottlenecks and successfully transition into a highly efficient, compute-bound regime at larger block dimensions.
\begin{figure}[h]
\centering
\begin{tikzpicture}
\begin{axis}[
    title={},
    xlabel={number of principal layers},
    ylabel={time (milliseconds)},
    xmin=13, xmax=180,
    ymin=7, ymax=90000,
    ymode = log, xmode = log,
    xtick={10, 20, 40, 80, 160, 200},
    xticklabels={10, 20, 40, 80, 160, 200},
    ytick={1.0, 10.0, 100.0, 391.0, 1000.0, 3250.0, 20000.0},
    yticklabels={1.0, 10.0, 100.0, 391.0, 1000.0, 3250.0, 20000.0},
    legend pos=north west,
    ymajorgrids=true,
    grid style=dashed,
    width=7.5cm,
]
{\addplot[
    color=blue,
    mark=square,
    ] coordinates {
    (20,2350.0)(40,4810.0)(80,9790.0)(160,19900.0)
    };\addlegendentry{$b_{s}$=512}}
{\addplot[
    color=green,
    mark=*,
    ] coordinates {
    (20,391.0)(40,804.0)(80,1620.0)(160,3250.0)
    };\addlegendentry{$b_{s}$=256}}
{\addplot[
    color=red,
    mark=x,
    ] coordinates {
    (20,68.1)(40,142.0)(80,292.0)(160,586.0)
    };\addlegendentry{$b_{s}$=128}}
{\addplot[
    color=orange,
    mark=+,
    ] coordinates {
    (20,11.9)(40,24.9)(80,51.8)(160,107.0)
    };\addlegendentry{$b_{s}$=64}}
\end{axis}
\end{tikzpicture}
\begin{tikzpicture}
\begin{axis}[
    title={},
    xlabel={block size $b_s$},
    ylabel={time (milliseconds)},
    xmin=45, xmax=700,
    ymin=7, ymax=90000,
    ymode = log, xmode = log,
    xtick={64, 128, 256, 512},
    xticklabels={64, 128, 256, 512},
    ytick={10.0, 100.0, 1000.0, 10000.0},
    yticklabels={10.0, 100.0, 1000.0, 10000.0},
    legend pos=north west,
    ymajorgrids=true,
    grid style=dashed,
    width=7.5cm,
]
{\addplot[
    color=blue,
    mark=square,
    ] coordinates {
    (64,107.0)(128,586.0)(256,3250.0)(512,19900.0)
    };\addlegendentry{$\ell$=160}}
{\addplot[
    color=green,
    mark=*,
    ] coordinates {
    (64,51.8)(128,292.0)(256,1620.0)(512,9790.0)
    };\addlegendentry{$\ell$=80}}
{\addplot[
    color=red,
    mark=x,
    ] coordinates {
    (64,24.9)(128,142.0)(256,804.0)(512,4810.0)
    };\addlegendentry{$\ell$=40}}
{\addplot[
    color=orange,
    mark=+,
    ] coordinates {
    (64,11.9)(128,68.1)(256,391.0)(512,2350.0)
    };\addlegendentry{$\ell$=20}}
\addplot[
    color=black,
    mark=none,
    dashed,
    thick
] coordinates {
    (64,38.8)
    (512,19900.0)
}; \addlegendentry{$O(b_s^3)$}
\end{axis}
\end{tikzpicture}
\caption{Execution time of sequential RGF with 1 BLAS thread on a single core in a node of the JUWELS Cluster Module. \textbf{Left panel:} Scaling with respect to the number of principal layers $\ell$ for fixed block sizes $b_s$. On this log-log scale, the data follows a linear trend, confirming the dependency $\log(t) = \log(f_{\ell}) + \log(\ell)$. \textbf{Right panel:} Scaling with respect to the block size $b_s$ for fixed $\ell$. The dashed black line represents the ideal theoretical algorithmic complexity of $O(b_s^3)$.
We have added axes labels with the original non-log values. Each point corresponds to an average over 30 runs of RGF.}\label{fig:execTimeSequentialRGF}
\end{figure}
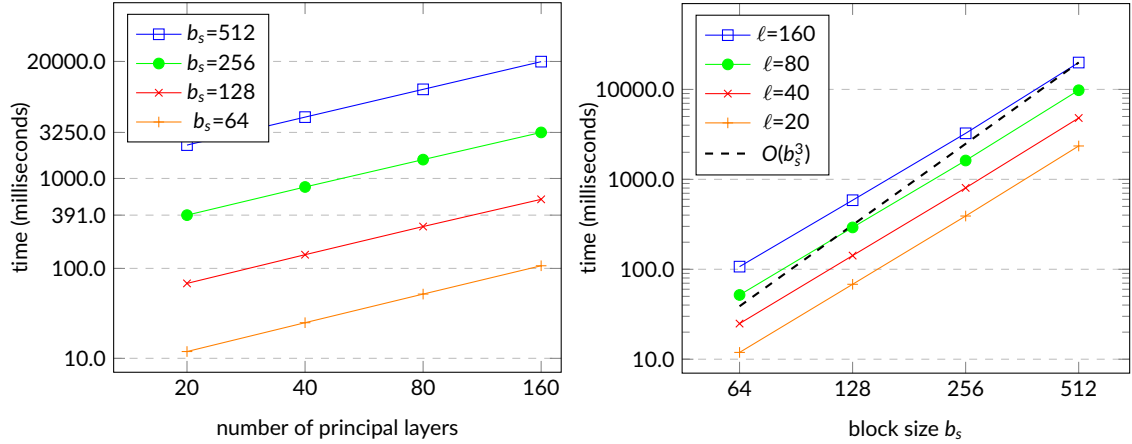

\subsection{RGF ($n>3$)}\label{sec:results_rgf_ngt3}
For $n>3$, the exact theoretical cost models for the native block $n$-diagonal RGF and the fused-based method are given by eq.s \ref{eq:cost_model_general_RGF} and \ref{eq:cost_model_general_RGF_fused}, respectively. We have implemented both methods in our \texttt{LibNEGF.jl} code and run them on a single core of a JUWELS Cluster node, with the corresponding execution times and theoretical scaling curves presented in fig.\ \ref{fig:execTimesGeneralRGF}. As the block bandwidth $w$ increases, the native $n$-diagonal RGF significantly outperforms the fused approach, rapidly widening its performance lead. This empirical result validates our theoretical analysis: by preserving the native $b_s \times b_s$ block granularity, our algorithm restricts expensive arithmetic strictly to the non-zero sparsity pattern, actively avoiding the severe $O(W_s^3)$ complexity penalty incurred by the bloated super-blocks of the fused method.\\
While the theoretical curves track the measured execution times well, the physical realities of modern hardware architecture introduce hardware-induced deviations at larger bandwidths (e.g., at $w=3$, the fused super-blocks reach a massive dimension of $W_s = w \cdot b_s = 3 \times 512 = 1536$). For the native approach, the measured execution time slightly exceeds the theoretical prediction. This minor discrepancy is driven by the unmodeled overhead due to data movement in memory---such as zeroing data buffers and copying matrices---which begins to saturate the memory bandwidth and add latency as the overall memory footprint grows. Conversely, for the fused approach at $w=3$, the strict $O(N^3)$ theoretical scaling actually over-predicts the measured execution time. This occurs because the cache-aware blocking strategy utilized by OpenBLAS (the GotoBLAS architecture) successfully reclaims some computational efficiency when operating on larger, denser matrices. While this creates a hardware-dependent sweet spot where the fused model leverages dense GEMM efficiency, the fused approach ultimately computes strictly more operations; our native implementation remains fundamentally superior by avoiding structural zeros entirely.
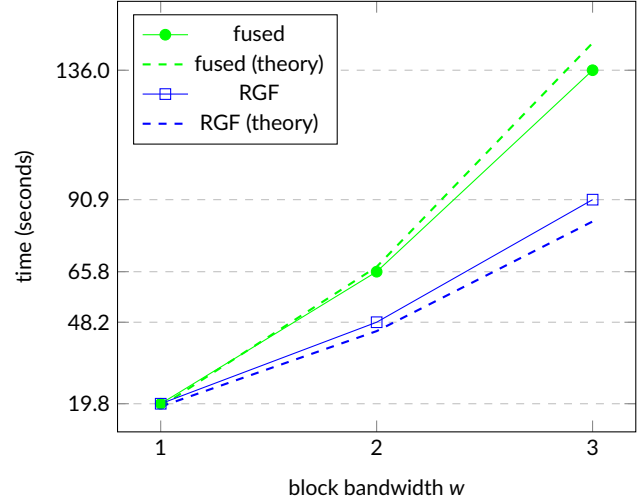
\begin{figure}[h]
\centering
\begin{tikzpicture}
\begin{axis}[
    title={},
    xlabel={block bandwidth $w$},
    ylabel={time (seconds)},
    xmin=0.8, xmax=3.2,
    ymin=10, ymax=160,
    ymode = linear, xmode = linear,
    xtick={1, 2, 3},
    xticklabels={1, 2, 3},
    ytick={19.8, 48.2, 65.8, 90.9, 136.0},
    yticklabels={19.8, 48.2, 65.8, 90.9, 136.0},
    legend pos=north west,
    ymajorgrids=true,
    grid style=dashed,
]
{\addplot[
    color=green,
    mark=*,
    ] coordinates {
    (1,19.8)(2,65.8)(3,136.0)
    };\addlegendentry{fused}}
{\addplot[
    color=green,
    mark=none,
    dashed,
    thick,
    ] coordinates {
    (1,18.8)(2,67.5)(3,145.5)
    };\addlegendentry{fused (theory)}}
{\addplot[
    color=blue,
    mark=square,
    ] coordinates {
    (1,19.8)(2,48.2)(3,90.9)
    };\addlegendentry{RGF}}
{\addplot[
    color=blue,
    mark=none,
    dashed,
    thick,
    ] coordinates {
    (1,18.8)(2,45.1)(3,83.3)
    };\addlegendentry{RGF (theory)}}
\end{axis}
\end{tikzpicture}
\caption{Execution time for solving a block $n$-diagonal system, with either the native $n$-diagonal RGF presented in sect.\ \ref{sec:RGF_block_ndiagonal} or the fused approach with the block tridiagonal RGF solver, using 1 BLAS thread on a single core in a node of the JUWELS Cluster Module. Each point corresponds to an average over 30 runs. The system has parameters $\ell=160$ and $b_{s} = 512$. The theoretical scaling of the exact cost models for both approaches are plotted as dashed lines for direct comparison, illustrating the severe penalties of super-block padding.}
\label{fig:execTimesGeneralRGF}
\end{figure}

\subsection{DDRGF ($n=3$)}\label{sec:results_ddrgf}
In this section, we benchmark the domain decomposition based RGF method (DDRGF) introduced in sect.~\ref{sec:DDRGF}. While DDRGF shares fundamental system parameters with RGF (number of principal layers $\ell$ and block size $b_s$), its recursive, parallel nature introduces several operational parameters. Because DDRGF calls itself recursively to solve the reduced Schur complement, we define these parameters on a level-by-level basis. Let $k$ denote the current recursion level, spanning from 1 (the finest level) up to $n_{levels}$ (the deepest level, where sequential RGF is finally called). The parameters are:
\begin{itemize}
    \item $s_{1}^{(k)}$: The number of principal layers in each sub-domain of $\mathcal{D}_{1}$ at level $k$. Throughout our experiments, we uniformly fix $s_{1}^{(k)} = 1$ for all levels to maximize boundary sparsity.
    \item $s_{2}^{(k)}$: The number of principal layers in each sub-domain of $\mathcal{D}_{2}$ at level $k$. Because this dictates the coarsening rate, we can choose a different value at each depth, effectively defining a sequence $\mathbf{s}_2 = (s_2^{(1)}, s_2^{(2)}, \dots, s_2^{(n_{levels})})$.
    \item $n_{tasks}^{(k)}$: The number of independent, concurrent matrix operations generated at level $k$, mathematically given by $\lceil \ell^{(k)} / (s_{1}^{(k)} + s_{2}^{(k)}) \rceil$.
    \item $n_{threads}$: The total number of hardware threads available to embarrassingly parallelize the $n_{tasks}^{(k)}$ tasks at any given level.
\end{itemize}
\noindent We implemented DDRGF in our \texttt{LibNEGF.jl} library, using Julia's native threads for shared-memory parallelism. To prevent resource contention between Julia's task scheduler and OpenBLAS's internal thread pool, we turn off BLAS threading for all DDRGF recursion levels except the deepest sequential one. In contrast, when profiling the baseline standard RGF, we fully enable BLAS threading to ensure a fair comparison.\\
Recall that our \texttt{LibNEGF.jl} implementation features a dynamic auto-tuner (introduced in sect.~\ref{sec:RGF_block_tridiagonal}) that automatically selects the optimal recursion parameters based on the hardware. However, to practically illustrate the concurrency-to-complexity trade-off predicted by our cost model, our analysis begins with fig.\ \ref{fig:execTimeAndMemDDRGF}, where we explicitly bypass this auto-tuner and manually force specific $\mathbf{s}_2$ sequences. We hold $\ell = 1440$ and $b_{s} = 256$ constant. It is crucial to note, from this experiment, how the number of available hardware threads dictates the optimal recursion geometry.\\
In the left panel, where a high number of threads ($n_{threads}=24$) is available, the blue curve (using a thick initial coarsening $s_2^{(1)}=4$, i.e., $\mathbf{s}_2=(4,1,1,1,1)$) yields the fastest execution. By absorbing more layers early, we drastically shrink the size of the subsequent Schur complement. The heavy local arithmetic is easily digested by the 24 waiting threads, making deeper recursion (with the minimum on the blue curve at $n_{levels}=4$) highly effective. Conversely, looking at the right panel where threads are severely restricted ($n_{threads}=8$), the narrative flips. The heavy local blocks of $s_2^{(1)}=4$ become a bottleneck because there are not enough threads to process them concurrently. Instead, the orange curve ($\mathbf{s}_2=(1,1,1,1)$) becomes the superior strategy. Because thread scarcity prevents us from hiding the arithmetic overhead of domain stitching, the optimal execution point for the orange curve occurs at $n_{levels}=1$. By maintaining minimal coarsening and strictly halting at a shallow recursion depth, we avoid the mathematical penalty that parallelism cannot compensate for.\\
Because manually probing these recursion depths and domain sizes is tedious and highly hardware-dependent, we rely on the DDRGF auto-tuner for the remainder of our experiments (fig.s \ref{fig:scalingDDRGF} and \ref{fig:scalingDDRGF_2880_512}). Before the main execution, the auto-tuner probes the host architecture to measure the exact $r_{LU}(b_s)$ and $r_{GETRS}(b_s)$ hardware ratios. It feeds these directly into the analytical cost model (eq.\ \ref{eq:cost_model_DDRGF}) to dynamically predict and select the optimal $\mathbf{s}_2$ sequence and $n_{levels}$ depth for the given thread count.
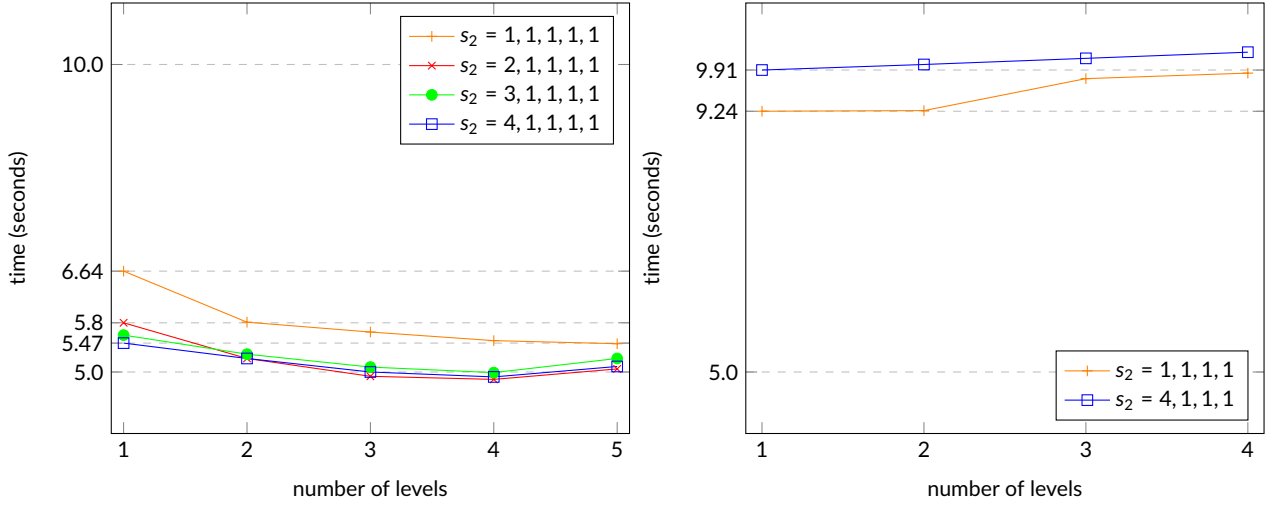
\begin{figure}[h]
\centering
\begin{tikzpicture}
\begin{axis}[
    title={},
    xlabel={number of levels},
    ylabel={time (seconds)},
    xmin=0.9, xmax=5.1,
    ymin=4, ymax=11,
    ymode = linear, xmode = linear,
    xtick={1, 2, 3, 4, 5},
    xticklabels={1, 2, 3, 4, 5},
    ytick={5.0, 5.47, 5.8, 6.64, 10.0, 15.0},
    yticklabels={5.0, 5.47, 5.8, 6.64, 10.0, 15.0},
    legend pos=north east,
    ymajorgrids=true,
    grid style=dashed,
]
{\addplot[
    color=orange,
    mark=+,
    ] coordinates {
    (1,6.64)(2,5.81)(3,5.65)(4,5.51)(5,5.46)
    };\addlegendentry{$s_{2}=1,1,1,1,1$}}
{\addplot[
    color=red,
    mark=x,
    ] coordinates {
    (1,5.8)(2,5.22)(3,4.93)(4,4.88)(5,5.05)
    };\addlegendentry{$s_{2}=2,1,1,1,1$}}
{\addplot[
    color=green,
    mark=*,
    ] coordinates {
    (1,5.6)(2,5.29)(3,5.08)(4,4.99)(5,5.22)
    };\addlegendentry{$s_{2}=3,1,1,1,1$}}
{\addplot[
    color=blue,
    mark=square,
    ] coordinates {
    (1,5.47)(2,5.22)(3,5.0)(4,4.92)(5,5.09)
    };\addlegendentry{$s_{2}=4,1,1,1,1$}}
\end{axis}
\end{tikzpicture}
\begin{tikzpicture}
\begin{axis}[
    title={},
    xlabel={number of levels},
    ylabel={time (seconds)},
    xmin=0.9, xmax=4.1,
    ymin=4, ymax=11,
    ymode = linear, xmode = linear,
    xtick={1, 2, 3, 4},
    xticklabels={1, 2, 3, 4},
    ytick={5.0, 9.24, 9.91, 15.0},
    yticklabels={5.0, 9.24, 9.91, 15.0},
    legend pos=south east,
    ymajorgrids=true,
    grid style=dashed,
]
{\addplot[
    color=orange,
    mark=+,
    ] coordinates {
    (1,9.24)(2,9.25)(3,9.77)(4,9.86)
    };\addlegendentry{$s_{2}=1,1,1,1$}}
{\addplot[
    color=blue,
    mark=square,
    ] coordinates {
    (1,9.91)(2,10.0)(3,10.1)(4,10.2)
    };\addlegendentry{$s_{2}=4,1,1,1$}}
\end{axis}
\end{tikzpicture}
\caption{Execution time of DDRGF on one node of the JUWELS Cluster Module. Each curve corresponds to fixed $\mathbf{s}_2$ sequences, see the legends, where we see that $s_{2}=1$ for all levels except the finest. We have also set $s_{1} = 1$ and $\ell = 1440$. We vary the number of levels with a fixed number of Julia threads, $n_{threads} = 24$ on the left panel and $n_{threads} = 8$ on the right, with block size $b_{s} = 256$ in all cases. Each point in either panel corresponds to an average over 30 runs of DDRGF.}\label{fig:execTimeAndMemDDRGF}
\end{figure}
The exact parameter sequences chosen by the auto-tuner for our scaling runs are detailed in tab.\ \ref{tab:values_s2_DDRGF_forced_dims}. Note how the auto-tuner intelligently modulates the vector $\mathbf{s}_2$ as the available thread count shifts.\\
In fig.\ \ref{fig:scalingDDRGF}, we study the strong scaling of auto-tuned DDRGF on the JUWELS Cluster Module ($\ell=1440$, $b_{s}=256$). This plot perfectly encapsulates the central premise of our algorithm: at a low thread count (1 to 3 threads), the baseline RGF outperforms DDRGF simply because DDRGF inherently executes more total arithmetic operations to stitch the domains together. 
\begin{table}[h]
    \centering
    \begin{tabular}{ |c|c|c|c|c|c| }
     \hline
     \multirow{2}{*}{$\ell$, $b_{s}$} & \multicolumn{5}{c|}{$s_{2}$, $n_{threads}$} \\
     \cline{2-6}
      & level = 1 & level = 2 & level = 3 & level = 4 & level = 5\\
     \hline
     \multirow{6}{*}{1440, 256} & 4, 1 & -, - & -, - & -, - & -, -\\
      & 4, 3 & 1, 3 & 1, 3 & 1, 3 & 2, 3\\
      & 4, 6 & 1, 6 & 1, 6 & 2, 6 & 3, 6\\
      & 4, 12 & 1, 12 & 2, 12 & 3, 12 & 2, 4\\
      & 4, 24 & 2, 24 & 3, 24 & 2, 8 & 2, 3\\
      & 4, 48 & 1, 48 & 3, 36 & 2, 12 & 2, 4\\
     \hline
     \multirow{6}{*}{2880, 512} & 4, 1 & - & - & - & -\\
      & 4, 3 & 1, 3 & 1, 3 & 2, 3 & 3, 3\\
      & 4, 6 & 1, 6 & 2, 6 & 3, 6 & 3, 6\\
      & 4, 12 & 2, 12 & 3, 12 & 3, 12 & 2, 4\\
      & 4, 24 & 1, 24 & 2, 24 & 3, 24 & 1, 12\\
      & 4, 48 & 2, 48 & 3, 48 & 1, 24 & 1, 12\\
     \hline
    \end{tabular}
    \caption{Values, per level, of $s_{2}$ and the number of threads, dynamically set by DDRGF's auto-tuner in the runs in fig.s \ref{fig:scalingDDRGF} and \ref{fig:scalingDDRGF_2880_512}.}
    \label{tab:values_s2_DDRGF_forced_dims}
\end{table}
However, as the thread count increases, the strict sequential dependencies of standard RGF cause its performance to stagnate. Meanwhile, DDRGF effectively amortizes its extra arithmetic complexity by exploiting the exposed concurrency, overtaking RGF at 12 threads and continuing to scale efficiently up to 48 threads.\\
Finally, in fig.\ \ref{fig:scalingDDRGF_2880_512}, we double the problem size to $\ell=2880$ and $b_{s}=512$. Due to the memory limitations of a single node in the JUWELS Cluster Module for a system of this magnitude, we executed these runs on a node of the JUWELS Booster Module. While we still execute purely on CPU cores for these benchmarks, a Booster node possesses a different architecture consisting of eight NUMA domains.
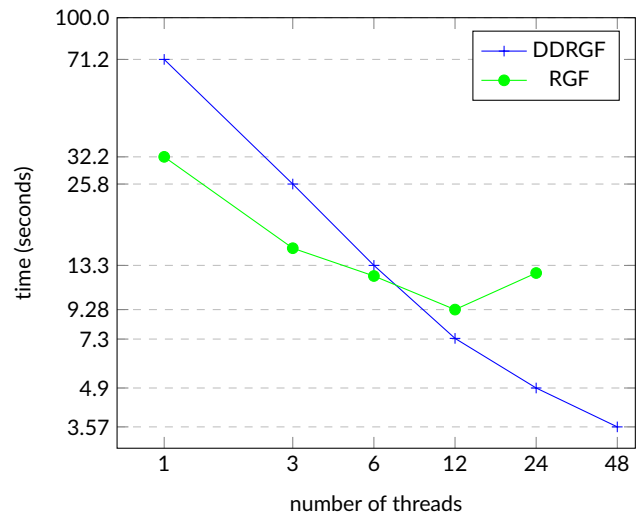
\begin{figure}[h]
\centering
\begin{tikzpicture}
\begin{axis}[
    title={},
    xlabel={number of threads},
    ylabel={time (seconds)},
    xmin=-3, xmax=56,
    ymin=3.0, ymax=100,
    ymode = log, xmode = log,
    xtick={1, 3, 6, 12, 24, 48},
    xticklabels={1, 3, 6, 12, 24, 48},
    ytick={1.0, 3.57, 4.9, 7.3, 9.28, 13.3, 25.8, 32.2, 71.2, 100.0},
    yticklabels={1.0, 3.57, 4.9, 7.3, 9.28, 13.3, 25.8, 32.2, 71.2, 100.0},
    legend pos=north east,
    ymajorgrids=true,
    grid style=dashed,
]
{\addplot[
    color=blue,
    mark=+,
    ] coordinates {
    (1,71.2)(3,25.8)(6,13.3)(12,7.34)(24,4.9)(48,3.57)
    };\addlegendentry{DDRGF}
    }
{\addplot[
    color=green,
    mark=*,
    ] coordinates {
    (1,32.2)(3,15.3)(6,12.2)(12,9.28)(24,12.5)
    };\addlegendentry{RGF}
    }
\end{axis}
\end{tikzpicture}
\caption{Execution time of DDRGF on one node of the JUWELS Cluster Module, with varying number of threads. We have set $b_{s} = 256$ and $\ell=1440$. The execution time of RGF is added for comparison, where the horizontal axis indicates the number of BLAS threads in that case. Each point in either panel corresponds to an average over 30 runs of DDRGF.}\label{fig:scalingDDRGF}
\end{figure}
We observe two key phenomena in this scaling plot. First, the baseline RGF execution exhibits a severe performance degradation (a jump in execution time) starting at 12 BLAS threads. This occurs because spanning OpenBLAS threads across the dense NUMA boundaries of the Booster node incurs massive memory access latencies. DDRGF completely circumvents this bottleneck by enforcing strict memory locality within its independent, isolated domain tasks. Second, DDRGF starts ``winning'' (overtaking RGF) slightly later in the thread count compared to fig.\ \ref{fig:scalingDDRGF}. This delayed crossover is expected: the larger block size ($b_{s}=512$) cubically increases the induced arithmetic complexity in DDRGF's boundary updates, requiring a correspondingly larger pool of concurrent threads to offset the extra work.\\
While the strong scaling results in fig.s \ref{fig:scalingDDRGF} and \ref{fig:scalingDDRGF_2880_512} prove that DDRGF can successfully overcome its mathematical overhead given enough threads, one must carefully contextualize these single-node performance gains within a typical NEGF workflow. In many practical simulations, researchers evaluate transport observables over a vast grid of independent momentum ($k$) and energy ($E$) points.
\begin{figure}[h]
\centering
\begin{tikzpicture}
\begin{axis}[
    axis y line*=right,
    axis x line=none,
    title={},
    xlabel={number of threads},
    ylabel={time (seconds)},
    xmin=-3, xmax=60,
    ymin=30.0, ymax=2000,
    ymode = log, xmode = log,
    xtick={1, 3, 6, 12, 24, 48},
    xticklabels={1, 3, 6, 12, 24, 48},
    ytick={1.0, 48.9, 74.7, 143.0, 280.0, 545.0, 1623.0},
    yticklabels={1.0, 48.9, 74.7, 143.0, 280.0, 545.0, 1623.0},
    legend pos=north east,
    ymajorgrids=true,
    grid style=dashed,
]
{\addplot[
    color=orange,
    mark=*,
    ] coordinates {
    (1,1623.0)(3,545.0)(6,280.0)(12,143.0)(24,74.7)(48,48.9)
    };\addlegendentry{DDRGF}
    }
{\addplot[
    color=blue,
    mark=x,
    ] coordinates {
    (24,127.0)
    };\addlegendentry{RGF}
    }
\end{axis}
\begin{axis}[
    title={},
    xlabel={number of threads},
    ylabel={time (seconds)},
    xmin=-3, xmax=60,
    ymin=30.0, ymax=2000,
    ymode = log, xmode = log,
    xtick={1, 3, 6, 12, 24, 48},
    xticklabels={1, 3, 6, 12, 24, 48},
    ytick={1.0, 106.0, 127.0, 163.0, 241.0, 579.0},
    yticklabels={1.0, 106.0, 127.0, 163.0, 241.0, 579.0},
    legend pos=north east,
    ymajorgrids=true,
    grid style=dashed,
]
{\addplot[
    color=blue,
    mark=x,
    ] coordinates {
    (1,579.0)(3,241.0)(6,163.0)(12,106.0)(24,127.0)
    };
    }
\end{axis}
\end{tikzpicture}
\caption{Execution time of DDRGF on one node of the JUWELS Booster Module, with varying number of threads. We solve in this case a system with $\ell=2880$ and $b_{s}=512$. The execution time of RGF is added for comparison, where the horizontal axis indicates the number of BLAS threads in that case. Each point in either panel corresponds to an average over 30 runs of DDRGF.}\label{fig:scalingDDRGF_2880_512}
\end{figure}
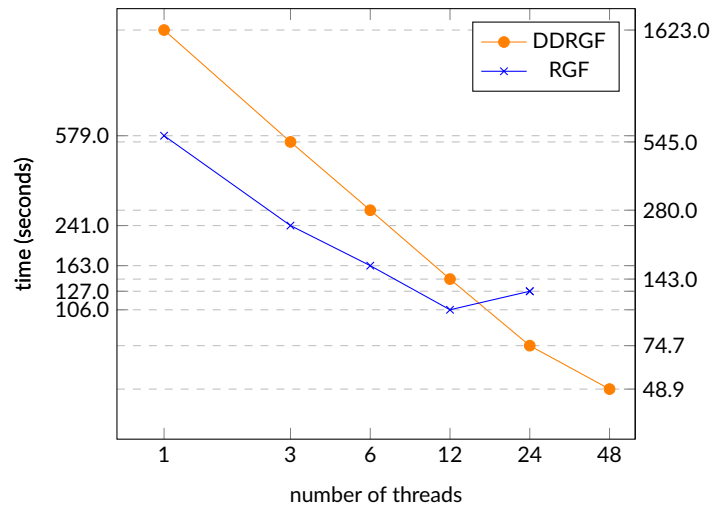
If the device cross-section is small enough that a single RGF calculation easily fits into the memory of a single node, the most efficient resource utilization is often an embarrassingly parallel distribution of standard RGF executions across the available cores. In such high-throughput scenarios, blindly forcing all threads to cooperate on a single DDRGF calculation would unnecessarily degrade overall throughput due to the domain-stitching overhead.\\
However, the direct empirical comparisons presented here are crucial for establishing the strict performance boundaries of both methods. By precisely quantifying where DDRGF overtakes RGF, we lay the groundwork for an intelligent ``orchestrator'' mechanism within the \texttt{LibNEGF.jl} library. This orchestrator will dynamically evaluate the available hardware and system dimensions to automatically dispatch either RGF (for highly-parallel $k-E$ sweeps of smaller devices) or DDRGF (for massively parallel, constrained single-point calculations), ensuring optimal execution without blind reliance on a single algorithm.\\
Most importantly, the true necessity of DDRGF becomes absolute when scaling toward realistic, full-device simulations. As device cross-sections grow, a single block-tridiagonal inversion will eventually exceed the RAM capacity of an individual node, rendering standard RGF entirely unusable. Because DDRGF isolates domains into independent tasks, it provides the exact mathematical framework required to gracefully scale out to distributed-memory clusters (e.g., via MPI). Finally, this task-based isolation is the key to unlocking multi-GPU acceleration. While sequential RGF is fundamentally bottlenecked by the execution pipeline of a single GPU, the exposed concurrency of DDRGF allows independent domain calculations to be mapped across multiple accelerators simultaneously, unlocking extreme performance leaps that are mathematically impossible under the standard RGF formulation.

\section{Conclusions and outlook}\label{sec5}

The simulation of quantum transport in realistic nanodevices demands highly efficient solutions to the Dyson and Keldysh equations. In this work, we set out to address the scalability and flexibility limitations of the standard Recursive Green's Function (RGF) algorithm by carefully reformulating it through the mathematical lens of Domain Decomposition and Schur Complement theory. This structural re-evaluation allowed us to naturally generalize the recursive formalism to block $n$-diagonal systems ($n>3$). By processing these higher-order stencils in their native granularity, we mathematically demonstrated and empirically validated that our approach actively avoids the severe complexity penalties associated with the dense artificial super-blocks of traditional fusing strategies.\\
Furthermore, this domain decomposition framework directly enabled the development of Domain-Decomposition based RGF (DDRGF). By explicitly formalizing the block-sparse data dependencies across macroscopic device partitions, we exposed a transparent layer of concurrency obscured in the inherently sequential RGF. Alongside this algorithm, we derived hardware-aware cost models that rigorously map the fundamental trade-off between the increased arithmetic complexity of the boundary updates and the time saved via exposed parallelism. Integrated into our Julia-based \texttt{LibNEGF.jl} library, these cost models power a dynamic auto-tuner that consistently identifies the optimal concurrency strategy. Ultimately, our reformulation successfully bridges the gap between the physics-driven need for non-equilibrium observables and the rigorous demands of modern high-performance numerical linear algebra.\\
Looking forward, while our shared-memory implementation of DDRGF effectively accelerates simulations for moderate system sizes, tackling massive physical systems requires transcending the limits of a single computational node. For such large-scale problems on CPU architectures, a hybrid paradigm combining \texttt{MPI.jl} with Julia threads represents the natural next step. Consequently, we plan to extend \texttt{LibNEGF.jl} to fully support distributed memory environments, enabling seamless multi-node scaling.\\
Beyond distributed CPU environments, the DDRGF architecture is uniquely positioned to exploit GPU accelerators. Because traditional CPU-based BLAS threading is discarded in a GPU paradigm, the task-based concurrency exposed by DDRGF becomes essential for orchestrating asynchronous GPU kernel dispatches. This approach provides a robust pathway to attain extreme performance on next-generation supercomputers, such as the JUPITER exascale system. We intend to realize this by first porting the core computational kernels via \texttt{CUDA.jl}. Once this GPU foundation is established, we plan to integrate \texttt{NCCL.jl} to highly optimize multi-GPU communication across distributed nodes, followed eventually by a generalization to diverse GPU backends using \texttt{Juliana.jl}.\\
Alongside these high-performance computing optimizations, we will expand the algorithmic scope of the parallel DDRGF implementation to fully support block $n$-diagonal systems for $n>3$. Finally, while this paper focuses exclusively on the Dyson equation, we have already laid the theoretical groundwork for solving the Keldysh equation recursively via an algorithm we term Recursive Keldysh (RKD). The immediate next phase of our research will translate the domain decomposition parallelization schemes developed here and the upcoming GPU adaptations into the RKD solver, ultimately establishing a comprehensive, large-scale-ready framework for NEGF simulations.

\bmsection*{Acknowledgments}
We acknowledge financial support from the EoCoE-III project, which has received funding from the European High Performance Computing Joint Undertaking under grant agreement No. 101144014. The authors gratefully acknowledge the Gauss Centre for Supercomputing e.\ V.\ for funding this project by providing computing time through the John von Neumann Institute for Computing (NIC) on the GCS Supercomputer JUWELS at Julich Supercomputing Centre (JSC), under the project with id \texttt{mat4energy}.

\bmsection*{Conflict of interest}

The authors declare no potential conflict of interests.

\bibliography{wileyNJD-AMA}

@article{Vetsch2025,
  author = {Vetsch, N. and Maeder, A. and Maillou, V. and Winka, A. and Cao, J. and Kwasniewski, G. and Luisier, M.},
  title = {Ab-initio Quantum Transport with the {GW} Approximation, 42,240 Atoms, and Sustained Exascale Performance},
  journal = {SC '25: Proceedings of the International Conference for High Performance Computing, Networking, Storage and Analysis},
  year = {2025}
}

@article{Maillou2025,
  author = {Maillou, V. and Ziogas, A. N. and Gaedke-Merzhäuser, L. and Schenk, O. and Luisier, M.},
  title = {Parallel Selected Inversion of Block-Tridiagonal with Arrowhead Matrices},
  journal = {IEEE Cluster},
  year = {2025}
}

@article{Petersen2009,
  author = {Petersen, D. E. and Li, S. and Stokbro, K. and S{\o}rensen, H. H. B. and Hansen, P. C. and Skelboe, S. and Darve, E.},
  title = {A hybrid method for the parallel computation of Green's functions},
  journal = {Journal of Computational Physics},
  volume = {228},
  number = {14},
  pages = {5020--5039},
  year = {2009}
}

@article{Drouvelis2006,
  author = {Drouvelis, P. S. and Schmelcher, P. and Bastian, P.},
  title = {Parallel implementation of the recursive Green's function method},
  journal = {Journal of Computational Physics},
  volume = {215},
  pages = {741--756},
  year = {2006}
}

@article{Jacquelin2016,
  author = {Jacquelin, M. and Lin, L. and Yang, C.},
  title = {PSelInv: A Distributed Shared Memory Parallel Algorithm for Selected Inversion: The Symmetric Case},
  journal = {ACM Trans. Math. Softw.},
  year = {2016}
}

@article{Li2008,
  author = {Li, S. and Ahmed, S. and Klimeck, G. and Darve, E.},
  title = {Computing entries of the inverse of a sparse matrix using the FIND algorithm},
  journal = {Journal of Computational Physics},
  volume = {227},
  pages = {9408--9427},
  year = {2008}
}

@article{Kilic2013,
  author = {K{\i}l{\i}{\c{c}}, E. and Stanica, P.},
  title = {The inverse of banded matrices},
  journal = {Journal of Computational and Applied Mathematics},
  volume = {237},
  pages = {126--135},
  year = {2013}
}

@article{Jain2007,
  author = {Jain, J. and Li, H. and Cauley, S. and Koh, C. and Balakrishnan, V.},
  title = {Numerically Stable Algorithms for Inversion of Block Tridiagonal and Banded Matrices},
  journal = {Purdue University ECE Technical Reports},
  year = {2007}
}

@article{Chavez2018,
  author = {Chavez, G. and Turkiyyah, G. and Zampini, S. and Ltaief, H. and Keyes, D.},
  title = {Accelerated Cyclic Reduction: A distributed-memory fast solver for structured linear systems},
  journal = {Parallel Computing},
  volume = {74},
  pages = {65--83},
  year = {2018}
}

@article{Zhao2018,
  title={Nested dissection solver for transport in 3D nano-electronic devices},
  author={Zhao, Y and Hetmaniuk, U and Patil, SR and Qi, J and Anantram, MP},
  journal={Journal of Computational Electronics},
  volume={15},
  number={2},
  pages={708--720},
  year={2016},
  publisher={Springer}
}

@article{Hetmaniuk2013,
  author = {Hetmaniuk, U. and Zhao, Y. and Anantram, M. P.},
  title = {A Nested Dissection Approach to Modeling Transport in Nanodevices},
  journal = {International Journal for Numerical Methods in Engineering},
  year = {2013}
}

@book{Datta2005,
  author = {Datta, S.},
  title = {Quantum Transport: Atom to Transistor},
  publisher = {Cambridge University Press},
  year = {2005}
}

@article{Luisier2006,
  author = {Luisier, M. and Schenk, A. and Fichtner, W. and Klimeck, G.},
  title = {Atomistic simulation of nanowires in the $sp^3d^5s^*$ tight-binding formalism: From boundary conditions to strain calculations},
  journal = {Phys. Rev. B},
  volume = {74},
  pages = {205323},
  year = {2006}
}

@article{d1990conductance,
  title={Conductance of a disordered linear chain including inelastic scattering events},
  author={d’Amato, Jorge L and Pastawski, Horacio M},
  journal={Physical Review B},
  volume={41},
  number={11},
  pages={7411},
  year={1990},
  publisher={APS}
}

@article{mackinnon1985calculation,
  title={The calculation of transport properties and density of states of disordered solids},
  author={MacKinnon, A},
  journal={Zeitschrift f{\"u}r Physik B Condensed Matter},
  volume={59},
  number={4},
  pages={385--390},
  year={1985},
  publisher={Springer}
}

@article{svizhenko2002two,
  title={Two-dimensional quantum mechanical modeling of nanotransistors},
  author={Svizhenko, A and Anantram, MP and Govindan, TR and Biegel, B and Venugopal, R},
  journal={Journal of Applied Physics},
  volume={91},
  number={4},
  pages={2343--2354},
  year={2002},
  publisher={American Institute of Physics}
}

@article{svizhenko2002role,
  title={Role of scattering in nanotransistors},
  author={Svizhenko, Alexei and Anantram, MP},
  journal={arXiv preprint cond-mat/0211069},
  year={2002}
}

@article{goto2008anatomy,
  title={Anatomy of high-performance matrix multiplication},
  author={Goto, Kazushige and Geijn, Robert A van de},
  journal={ACM Transactions on Mathematical Software (TOMS)},
  volume={34},
  number={3},
  pages={1--25},
  year={2008},
  publisher={ACM New York, NY, USA}
}

@software{LibNEGF_jl_2026_20043546,
  author       = {Conrads, Christoph and
                  Di Napoli, Edoardo Angelo and
                  Pecchia, Alessandro and
                  Ramirez-Hidalgo, Gustavo and
                  Robeyns, Matthieu},
  title        = {SimLabQuantumMaterials/LibNEGF.jl: LibNEGF.jl
                   v0.2.0
                  },
  month        = may,
  year         = 2026,
  publisher    = {Zenodo},
  version      = {v0.2.0},
  doi          = {10.5281/zenodo.20043546},
  url          = {https://doi.org/10.5281/zenodo.20043546},
}

@article{mackinnon85,
  author    = {MacKinnon, A.},
  title     = {The calculation of transport properties and density of states of disordered solids},
  journal   = {Zeitschrift f{\"u}r Physik B Condensed Matter},
  volume    = {59},
  number    = {4},
  pages     = {385--390},
  year      = {1985}
}

@article{peters09,
  author    = {Peters, A. and others},
  title     = {Efficient and parallel strategy for the inversion of block tridiagonal matrices},
  journal   = {Journal of Computational Physics},
  volume    = {228},
  number    = {14},
  pages     = {5020--5035},
  year      = {2009}
}

@article{meurant12,
  author    = {Meurant, G.},
  title     = {Numerically stable algorithms for inversion of block tridiagonal and banded matrices},
  journal   = {Numerical Linear Algebra with Applications},
  volume    = {19},
  number    = {3},
  pages     = {541--556},
  year      = {2012}
}

@article{somasunderam12,
  author    = {Somasunderam, N. and Chandrasekaran, S.},
  title     = {On the infinitesimal limits of the Schur complements of tridiagonal matrices},
  journal   = {Linear Algebra and its Applications},
  volume    = {436},
  number    = {3},
  pages     = {659--681},
  year      = {2012}
}

@article{osorghin24,
  author    = {Osorghin, N. and others},
  title     = {Massively parallel quantum transport simulations},
  journal   = {arXiv preprint arXiv:2405.14288},
  year      = {2024}
}

@article{stuckermann25,
  author    = {Stuckermann, F. and others},
  title     = {High-throughput quantum transport},
  journal   = {Nature Computational Science},
  volume    = {5},
  pages     = {123--134},
  year      = {2025}
}

@article{afzalian17,
  author    = {Afzalian, A. and others},
  title     = {Distributed non-equilibrium Green's function algorithms},
  journal   = {arXiv preprint arXiv:1702.05167},
  year      = {2017}
}

@article{jiang11,
  author    = {Jiang, L. and others},
  title     = {Parallel algorithms for the Green's function in quantum transport},
  journal   = {Journal of Applied Physics},
  volume    = {110},
  number    = {4},
  pages     = {043713},
  year      = {2011}
}

@article{Hirt1974,
      author = "Hirt, C W and Amsden, A A and Cook, J L",
       title = "An arbitrary {L}agrangian-{E}ulerian computing method for all flow speeds",
        year = "1974",
     journal = "J {C}omput {P}hys",
      volume = "14",
      number = "3",
       pages = "227--253"
}

@article{Benson1992,
      author = "Benson, D J",
       title = "Computational methods in {L}agrangian and {E}ulerian hydrocodes",
        year = "1992",
     journal = "Comput {M}ethod {A}ppl {M}",
      volume = "99",
      number = "2--3",
       pages = "235--394"
}

@article{Dukowicz1984,
      author = "Dukowicz, J",
       title = "Conservative rezoning (remapping) for general quadrilateral meshes",
        year = "1984",
     journal = "J {C}omput {P}hys",
      volume = "54",
      number = "3",
       pages = "411--424"
}

@article{Margolin2003,
      author = "Margolin, L G and Shashkov, M",
       title = "Second-order sign-preserving conservative interpolation (remapping) on general grids",
        year = "2003",
     journal = "J {C}omput {P}hys",
      volume = "184",
      number = "1",
       pages = "266--298"
}

@inproceedings{Kenamond2013,
      author = "Kenamond, M A and Burton, D E",
       title = "Exact intersection remapping of multi-material domain-decomposed polygonal meshes",
      series = "Talk at {M}ultimat 2013, {I}nternational {C}onference on {N}umerical {M}ethods for {M}ulti-{M}aterial {F}luid {F}lows",
organization = "The Organization",
     address = "San {F}rancisco",
        year = "September 2--6, 2013",
        note = "LA-UR-13-26794"
}

@inproceedings{Burton2013,
      author = "Burton, D E and Kenamond, M A and Morgan, N R and Carney, T C and Shashkov, M J and Author, A B",
       title = "An intersection based {ALE} scheme {(xALE)} for cell centered hydrodynamics {(CCH)}",
      series = "Talk at {M}ultimat 2013, {I}nternational {C}onference on {N}umerical {M}ethods for {M}ulti-{M}aterial {F}luid {F}lows",
     address = "San {F}rancisco",
organization = "The Organization", 
        year = "September 2--6, 2013",
        note = "LA-UR-13-26756.2"
}

@article{Berndt2011,
      author = "Berndt, M and Breil, J and Galera, S and Kucharik, M and Maire, P H and Shashkov, M",
       title = "Two-step hybrid conservative remapping for multimaterial arbitrary {L}agrangian-{E}ulerian methods",
        year = "2011",
     journal = "J {C}omput {P}hys",
      volume = "230",
      number = "17",
       pages = "6664--6687"
}

@article{Kucharik2012,
      author = "Kucharik, M and Shashkov, M",
       title = "One-step hybrid remapping algorithm for multi-material arbitrary {L}agrangian-{E}ulerian methods",
        year = "2012",
     journal = "J {C}omput {P}hys",
      volume = "231",
      number = "7",
       pages = "2851--2864"
}

@article{Breil2015,
      author = "Breil, J and Alcin, H and Maire, P H",
       title = "A swept intersection-based remapping method for axisymmetric {ReALE} computation",
        year = "2015",
     journal = "Int {J} {N}umer {M}eth {F}l",
      volume = "77",
      number = "11",
       pages = "694--706",
        note = "Fld.3996"
}

@incollection{Barth1997,
      author = "Barth, T J",
       title = "Numerical methods for gasdynamic systems on unstructured meshes",
   booktitle = "An {I}ntroduction to {R}ecent {D}evelopments in {T}heory and {N}umerics for {C}onservation {L}aws, {P}roceedings of the {I}nternational {S}chool on {T}heory and {N}umerics for {C}onservation {L}aws",
      editor = "Kroner, D and Rohde, C and Ohlberger, M",
      series = "Lecture {N}otes in {C}omputational {S}cience and {E}ngineering",
	  edition = "2",
   publisher = "Springer",
        year = "1997",
}

@article{Liska2010,
      author = "Liska, R and Shashkov, M and Vachal, P and Wendroff, B and Author, A B and Author, B B and Author, C C",
       title = "Optimization-based synchronized flux-corrected conservative interpolation (remapping) of mass and momentum for arbitrary {L}agrangian-{E}ulerian methods",
        year = "2010",
     journal = "J {C}omput {P}hys",
      volume = "229",
      number = "5",
       pages = "1467--1497"
}

@article{Kucharik2003,
      author = "Kucharik, M and Shashkov, M and Wendroff, B",
       title = "An efficient linearity-and-bound-preserving remapping method",
        year = "2003",
     journal = "J {C}omput {P}hys",
      volume = "188",
      number = "2",
       pages = "462--471"
}

@misc{Blanchard2015,
      author = "Blanchard, G and Loubere, R",
       title = "High-Order {C}onservative {R}emapping with a posteriori {MOOD} stabilization on polygonal meshes",
        year = "2015",
howpublished = "Details on how published",
	note = "Accessed January 13, 2016. \url{https://hal.archives-ouvertes.fr/hal-01207156}, the {HAL} {O}pen {A}rchive, hal-01207156."
}

@article{Lauritzen2011,
      author = "Lauritzen, P and Erath, C and Mittal, R",
       title = "On simplifying `incremental remap'-based transport schemes",
        year = "2011",
     journal = "J {C}omput {P}hys",
      volume = "230",
      number = "22",
       pages = "7957--7963"
}

@article{Klima2017,
      author = "Klima, M and Kucharik, M and Shashkov, M",
       title = "Local error analysis and comparison of the swept- and intersection-based remapping methods",
        year = "2017",
     journal = "Commun {C}omput {P}hys",
      volume = "21",
      number = "2",
       pages = "526--558"
}

@article{Dukowicz2000,
      author = "Dukowicz, J K and Baumgardner, J R",
       title = "Incremental remapping as a transport/advection algorithm",
        year = "2000",
     journal = "J {C}omput {P}hys",
      volume = "160",
      number = "1",
       pages = "318--335"
}

@incollection{Kucharik2011,
      author = "Kucharik, M and Shashkov, M",
       title = "Flux-based approach for conservative remap of multi-material quantities in {2D} arbitrary {L}agrangian-{E}ulerian simulations",
   booktitle = "Finite {V}olumes for {C}omplex {A}pplications {VI} {P}roblems \& {P}erspectives",
      editor = "Fo\v{r}t, J and F{\"{u}}rst, J and Halama, J and Herbin, R and Hubert, F",
      series = "Springer {P}roceedings in {M}athematics",
	 edition = "1",
   publisher = "Springer",
        year = "2011",
       pages = "623--631"
}

@article{Kucharik2014,
      author = "Kucharik, M and Shashkov, M",
       title = "Conservative multi-material remap for staggered multi-material arbitrary {L}agrangian-{E}ulerian methods",
        year = "2014",
     journal = "J {C}omput {P}hys",
      volume = "258",
       pages = "268--304"
}

@article{Loubere2005,
      author = "Loubere, R and Shashkov, M",
       title = "A subcell remapping method on staggered polygonal grids for arbitrary-{L}agrangian-{E}ulerian methods",
        year = "2005",
     journal = "J {C}omput {P}hys",
      volume = "209",
      number = "1",
       pages = "105--138"
}

@techreport{Margolin2002,
      author = "Margolin, L G and Shashkov, M",
       title = "Second-order sign-preserving remapping on general grids",
     address = "The address",
 institution = "Los {A}lamos {N}ational {L}aboratory",
        year = "2002",
      number = "Technical Report LA-UR-02-525"
}

@inproceedings{Mavriplis2003,
      author = "Mavriplis, D J",
       title = "Revisiting the least-squares procedure for gradient reconstruction on unstructured meshes",
        year = "June 23--26, 2003",
      series = "AIAA 2003-3986. 16th {AIAA} {C}omputational {F}luid {D}ynamics {C}onference",
organization = "The organization",
     address = "Orlando, {F}lorida"
}

@article{Scovazzi2008,
      author = "Scovazzi, G and Love, E and Shashkov, M",
       title = "Multi-scale {L}agrangian shock hydrodynamics on {Q1/P0} finite elements: {T}heoretical framework and two-dimensional computations",
        year = "2008",
     journal = "Comput {M}ethod {A}ppl {M}",
      volume = "197",
      number = "9--12",
       pages = "1056--1079"
}

@article{Caramana1998,
      author = "Caramana, E J and Shashkov, M J",
       title = "Elimination of artificial grid distortion and hourglass-type motions by means of {L}agrangian subzonal masses and pressures",
        year = "1998",
     journal = "J {C}omput {P}hys",
      volume = "142",
      number = "2",
       pages = "521--561"
}

@techreport{Hoch2009,
      author = "Hoch, P",
       title = "An arbitrary {L}agrangian-{E}ulerian strategy to solve compressible fluid flows",
      number = "Technical {R}eport",
 institution = "CEA",
     address = "The address",
        year = "2009",
        note = "Accessed January 13, 2016. HAL: hal-00366858. https://hal.archives-ouvertes.fr/docs/00/36/68/58/PDF/ale2d.pdf."
}

@book{Shashkov1996,
      author = "Shashkov, M",
       title = "Conservative {F}inite-{D}ifference {M}ethods on {G}eneral {G}rids",
   publisher = "CRC {P}ress",
        year = "1996"
}

@article{Knupp1999,
      author = "Knupp, P M",
       title = "Winslow smoothing on two-dimensional unstructured meshes",
        year = "1999",
     journal = "Eng {C}omput",
      volume = "15",
       pages = "263--268"
}

@techreport{kamm2000,
      author = "Kamm, J",
       title = "Evaluation of the {S}edov-von {N}eumann-{T}aylor blast wave solution",
 institution = "Los {A}lamos {N}ational {L}aboratory",
     address = "The address",
        year = "2000",
      number = "Technical {R}eport LA-UR-00-6055"
}

@article{Taylor1937,
      author = "Taylor, G I and Green, A E",
       title = "Mechanism of the production of small eddies from large ones",
        year = "1937",
     journal = "P {R}oy {S}oc {L}ond {A} {M}at",
      volume = "158",
      number = "895",
       pages = "499--521",
        note = "\url{https://doi.org/10.1098/rspa.1937.0036}, \url{http://rspa.royalsocietypublishing.org/content/158/895/499}"
}

@article{Yin2025,
  author = {Yin, J. and Ibrahim, K. Z. and Del Ben, M. and Deslippe, J. and Chan, Y. and Yang, C.},
  title = {GPU Acceleration of Non-equilibrium Green's Function Calculation using OpenACC and CUDA FORTRAN},
  journal = {arXiv preprint arXiv:2505.19467},
  year = {2025}
}

@article{Godoy1994,
  author = {Godoy, A. and Svizhenko, A. and Anantram, M. P.},
  title = {Recursive Green's Function Method},
  journal = {Phys. Rev. B},
  year = {1994}
}

@article{Somasunderam2012,
  author = {Somasunderam, N. and Chandrasekaran, S.},
  title = {On the infinitesimal limits of the Schur complements of tridiagonal matrices},
  journal = {Linear Algebra and its Applications},
  volume = {436},
  pages = {659--681},
  year = {2012}
}

@article{Gichev1995,
  author = {Gichev, D.},
  title = {Parallel Inversion of Block Tridiagonal Matrices},
  journal = {Mathematica Balkanica},
  volume = {9},
  number = {2-3},
  year = {1995}
}

@article{Petersen2008,
  author = {Petersen, D. E. and S{\o}rensen, H. H. B. and Hansen, P. C. and Skelboe, S. and Stokbro, K.},
  title = {Block tridiagonal matrix inversion and fast transmission calculations},
  journal = {Journal of Computational Physics},
  volume = {227},
  pages = {3174--3190},
  year = {2008}
}

@article{Gaury2014,
  author = {Gaury, B. and Weston, J. and Santin, M. and Houzet, M. and Groth, C. and Waintal, X.},
  title = {Numerical simulations of time-resolved quantum electronics},
  journal = {Physics Reports},
  volume = {534},
  pages = {1--37},
  year = {2014}
}

@article{Mahmood1991,
  author = {Mahmood, A. and Lynch, D. J. and Philipp, L. D.},
  title = {A Fast Banded Matrix Inversion Using Connectivity of Schur's Complements},
  journal = {Washington State University Technical Report},
  year = {1991}
}

@article{Grady2019,
	author = "Grady, J. S. and Her, M. and Moreno, G. and Perez, C. and Yelinek, J.",
	title = "Emotions in storybooks: A comparison of storybooks that represent ethnic and racial groups in the United States",
	year = "2019", 
	journal = "Psychology of Popular Media Culture",
	volume = "8",
	number = "3", 
	pages = "207--217"
}

@article{Grady2020,
	author = "Grady, J. S. and Her, M. and Moreno, G. and Perez, C. and Yelinek, J.",
	title = "Emotions in storybooks: A comparison of storybooks that represent ethnic and racial groups in the United States",
	year = "2020",
	journal = "Psychology of Popular Media Culture",
	volume = "8",
	number = "3", 
	pages = "207--217",
	note = "\url{https://doi.org/10.1037/ppm0000185}"
}

@article{Light2008,
	author = "Light, M. A. and Light, I. H.",
	title = "The geographic expansion of Mexican immigration in the United States and its implications for local law enforcement",
	year = "2008",
	journal = "Law Enforcement Executive Forum Journal",
	volume = "8",
	number = "1",
	pages = "73--82"
}

@article{Ledebur2007,
	author = "Von Ledebur, S. C.",
	title = "Optimizing knowledge transfer by new employees in companies",
	year = "2007",
	journal = "Knowledge Management Research \& Practice",
	note = "Advance online publication. \url{https://doi.org/10.1057/palgrave.kmrp.8500141}"
}

@article{Sanchiz2017,
	author = "Sanchiz, M. and Chevalier, A. and Amadieu, F.",
	title = "How do older and young adults start searching for information? Impact of age, domain knowledge and problem complexity on the different steps of information searching", 
	year = "2017",
	journal = "Computers in Human Behavior", 
	volume = "72", 
	pages = "67--78",
	note = "\url{https://doi.org/10.1016/j.chb.2017.02.038}"
}

@magazine{Lyons2009,
      author = "Lyons, D.",
       title = "Don't ‘iTune’ us: It’s geeks versus writers. {G}uess who’s winning",
        year = "2009",
		month = "June 15",
     journal = "Newsweek",
      volume = "153",
      number = "24",
       pages = "27"
}

@book{Jackson2019, 
	author = "Jackson, L. M.",
	title = "The psychology of prejudice: From attitudes to social action",
	year = "2019",
	edition = "2",
	publisher = "American Psychological Association",
	note = "\url{https://doi.org/10.1037/0000168-000}"
}

@book{Sapolsky2017,
	author = "Sapolsky, R. M.",
	title = "Behave: The biology of humans at our best and worst",
	year = "2017",
	publisher = "Penguin Books"
}

@book{Hygum2010,
	editor = "Hygum, E. and Pedersen, P. M.",
	title = "Early childhood education: Values and practices in Denmark",
	year = "2010",
	publisher = "Hans Reitzels Forlag",
	note = "\url{https://earlychildhoodeducation.digi.hansreitzel.dk/}"
}

@book{Watson2013,
	author = "Watson, J. B. and Rayner, R.",
	title = "Conditioned emotional reactions: The case of Little Albert",
	year = "2013",
	editor = "D. Webb", 
	publisher = "CreateSpace Independent Publishing Platform",
	note = "\url{http://a.co/06Se6Na (Original work published 1920)}"
}

@inproceedings{Aron2019,
	author = "Aron, L. and Botella, M. and Lubart, T.",
	title = "Culinary arts: Talent and their development",
	year = "2019",
	editor = "R. F. Subotnik and P. Olszewski-Kubilius and F. C. Worrell", 
	booktitle = "The psychology of high performance: Developing human potential into domain-specific talent",
	publisher = "American Psychological Association",
	pages = "345--359",
	note = "\url{https://doi.org/10.1037/0000120-016}"
}

@inproceedings{Dillard2020,
	author = "Dillard, J. P.",
	title = "Currents in the study of persuasion",
	year = "2020",
	editor = "M. B. Oliver and A. A. Raney and J. Bryant",
	booktitle = "Media effects: Advances in theory and research",
	edition = "4",
	publisher = "Routledge",
	pages = "115--129" 
}

@techreport{Terry2010,
      author = "Terry, M.A. and Lopez, F. M.",
       title = "Racism and poverty in the Bay Area",
 institution = "Embarcadero Institute",
        year = "2010",
      number = "Research Report No. 10.4",
	  note = "\url{http://www.bayarearesearch.org}"
}

@inproceedings{Abrams1973,
	author = "Abrams, D. M.",
	booktitle = "Conductive Polymers", 
	editor = "Seymour, R. S. and Smith, A.",
	publisher = "Springer",
	address = "Berlin Heidelberg New York", 
	year = "1973", 
	pages = "307"
}

@book{Ibach1996,
	author = "Ibach, H. and Lüth, H.",
	title = "Solid-State Physics,", 
	edition = "2",
	publisher = "Springer",
	address = "New York",
	year = "1996"
}

@inproceedings{MacKay1973,
	author = "MacKay, D. M.",
	booktitle = "Handbook of Sensory Physiology",
	editor = "Jung, R. and MacKay,  D.M.",
	publisher = "Springer",
	address = "Heidelberg",
	year = "1973", 
	pages = "307"
}

@inbook{Smith1976,
	author = "Smith, S. E.",
	booktitle = "Neuromuscular Junction",
	editor = "Zaimis., E.",
	series = "Handbook of Experimental Pharmacology", 
	volume = "42",
	publisher = "Springer", 
	address = "Heidelberg",
	year = "1976", 
	pages = "593"
}

@inbook{Zowghi1996,
	author = "Zowghi, D. and et al.", 
	booktitle = "PRICAI '96: Topics in Artificial Intelligence",
	editor = "Foo, N. and Goebel., R.", 
	series = "4th Pacific Rim Conference on Artificial Intelligence, Cairns, August 1996. Lecture Notes in Computer Science. Lecture notes in artificial intelligence",
	volume = "1114",
	publisher = "Springer",
	address = "Heidelberg", 
	year = "1996", 
	pages = "157"
}

@proceedings{Chung1978,
	author = "Chung, S. T. and Morris, R. L.",
	booktitle = "Abstracts of the 3rd International Symposium on the Genetics of Industrial Microorganisms", 
	publisher = "University of Wisconsin", 
	address = "Madison", 
	year = "4–9 June 1978"
}

@article{LaSalle2017,
      author = "LaSalle, Peter and Jose V. Lopez",
       title = "Conundrum: A Story about Reading",
        year = "2017",
     journal = "New England Review",
      volume = "38",
      number = "1",
       pages = "95--109"
}

@article{Campbell2015,
      author = "Campbell, Alexandra M. and Jay Fleisher and Christopher Sinigalliano and James R. White and Jose V. Lopez",
       title = "Dynamics of Marine Bacterial Community Diversity of the Coastal Waters of the Reefs, Inlets, and Wastewater Outfalls of Southeast Florida",
        year = "2015",
     journal = "Microbiology Open",
      volume = "4",
      number = "2",
       pages = "1--19",
	   note = "\url{https://doi.org/10.1002/mbo3.245}"
}

@article{Spieler1971,
      author = "Spieler, Richard E. and James R. White",
       title = "A Carp-Goldfish Hybrid with No Caudal Fin",
        year = "1971",
     journal = "Transactions of the Kansas Academy of Science",
      volume = "74",
      number = "3/4",
       pages = "342--43",
	   note = "\url{http://nsuworks.nova.edu/occ_facarticles/215/}"
}

@article{Campbell2019,
		author = "Campbell, Jesse W. and Richard E.",
        title = "Obtrusive, Obstinate and Conspicuous: Red Tape from a Heideggerian Perspective",
        year = "2019",
     journal = "International Journal of Organizational Analysis",
      volume = "27",
      number = "5",
       pages = "1657--72"
}

@magazine{Stolberg2010,
	author = "Stolberg, Sheryl Gay and Robert Pear",
	title = "Wary Centrists Posing Challenge in Health Care Vote",
	year = "2010",
	journal = "New York Times",
	date = "February 28",
	note = "\url{http://www.nytimes.com/2010/02/28/us/politics/28health.html}"
}

@article{Kossinets2009,
      author = "Kossinets, Gueorgi and Duncan J. Watts",
       title = "Origins of Homophily in an Evolving Social Network",
        year = "2009",
     journal = "American Journal of Sociology",
      volume = "~115",
       pages = "405--50",
        note = "\url{https://doi.org/10.1086/599247}"
}

@paper{Ku2000,
	author = "Ku, Bhatt and Helen Bake",
	title = "Re-reading the ‘167 event’: The Politics of Numbers and the Making of Hong Kong ‘Others’",
	year = "2000",
	number = "54",
	howpublished = "~Twelfth Annual Meeting on Socioeconomics",
	address = "London",
	date = "July 7--10"
}

@inproceedings{Chiswick1977,
      author = "Chiswick, Bake R.",
       title = "A Longitudinal Analysis of the Occupational Mobility of Immigrants",
   booktitle = "Proceedings of the 30th Annual Winter Meetings, Industrial Relations Research Association",
      editor = "Barbara D. Dennis",
     address = "20–7 Madison",
   publisher = "WI: IRRA",
        year = "1977"
}

@misc{Royko1992,
	author = "Royko, Mike.",
	title = "Next Time, Dan, Take Aim at Arnold",
	year = "1992",
	howpublished = "Chicago Tribune",
	date = "September 23, 1992"
}

@misc{Pai2017,
	author = "Pai, Tanya.",
	title = "The Squishy, Sugary History of Peeps",
	year = "2017",
	howpublished = "Vox",
	date = "April 11, 2017",
	note = "\url{http://www.vox.com/culture/2017/4/11/15209084/peeps-easter}"
}

@phdthesis{Rutz2013,
	author = "Rutz, Cynthia Lillian.",
	title = "King Lear and Its Folktale Analogues",
	year = "2013",
	series = "PhD diss.",
	publisher = "University of Chicago"
}

@mastersthesis{Pruzinsky2018,
	author = "Pruzinsky, Nina.",
	title = "Identification and Spatiotemporal Dynamics of Tuna (Family: Scombridae; Tribe: Thunnini) Early Life Stages in the Oceanic Gulf of Mexico",
	year = "2018",
	series = "MS thesis",
	publisher = "Nova Southeastern University",
	note = "\url{https://nsuworks.nova.edu/occ_stuetd/472/}"
}

@book{Pollan2006,
      author = "Pollan, Michael.",
       title = "The Omnivore's Dilemma: A Natural History of Four Meals",
	    year = "2006",
   publisher = "Penguin",
     address = "New York"
}

@book{Heatherton2007,
	author = "Heatherton, Joyce and James Fitzgilroy and Jackson Hsu",
	title = "Meteors and Mudslides: A Trip through",
	year = "2007",
	address = "New York",
	publisher = "Knopf"
}

@book{Purkis2011,
	author = "Purkis, Samuel and Victor Klemas",
	title = "Remote Sensing and Global Environmental Change",
	year = "2011",
	address = "Oxford",
	publisher = "Wiley-Blackwell"
}

@booked{Woodward1987,
	author = "Woodward, David, ed.",
	title = "Art and Cartography: Six Historical Essays",
	year = "1987",
	publisher = "University of Chicago Press",
	address = "Chicago"
}

@book{Soloviev2006,
	author = "Soloviev, Alexander and Roger Lukas",
	title = "The Near-Surface Layer of the Ocean: Structure, Dynamics and Applications",
	editor = "Lawrence A. Mysak and Kevin Hamilton",
	year = "2006",
	publisher = "Springer",
	address = "Dordrecht"
}

@inproceedings{Messing2008,
	author = "Messing, Charles G. and John K. Reed and Sandra D. Brooke and Steve W. Ross",
	title = "Deep-Water Coral Reefs of the United States",
	editor = "Bernhard M. Riegl and Richard E. Dodge",
	year = "2008",
	booktitle = "Coral Reefs of the USA",
	publisher = "Springer",
	address = "Dordrecht",
	pages = "767--92"
}

@incollection{Hodgson1876,
	author = "Hodgson, Shadworth.",
	title = "The Pre-Suppositions of Miracles",
	editor = "Catherine Marshall, Bernard Lightman, Richard England",
	booktitle = "The Papers of the Metaphysical Society, 1869–1880: A Critical Edition",
	volume = "2",
	publisher = "OUP",
	address = "Oxford",
	year ="1876",
	pages = "373--92"
}

@article{Keng2017,
      author = "Keng, Shao-Hsun and Chun-Hung Lin and Peter F. Orazem",
       title = "Expanding College Access in Taiwan, 1978–2014: Effects on Graduate Quality and Income Inequality",
        year = "2017",
     journal = "Journal of Human Capital",
      volume = "11",
	  number = "1, Spring",
       pages = "1--34"
}

@book{Jacobs2011,
	author = "Alan Jacobs",
	title = "The Pleasures of Reading in an Age of Distraction",
	publisher = "Oxford UP",
	year = "2011"
	}

@book{Dorris1999,
	author = "Michael Dorris and Louise Erdrich ",
	title = "The Crown of Columbus",
	publisher = "HarperCollins Publishers",
	year = "1999"
	}

@book{Charon2017,
	author = "Rita Charon and Michael Dorris  and Alan Jacobs and Baron Naomi and Alan Amsden", 
	title = "The Principles and Practice of Narrative Medicine", 
	publisher = "Oxford UP", 
	year = "2017"
}

@bookeditor{Baron2007, 
      author = "Baron, Sabrina Alcorn and M. Kucharik and M. Shashkov and Jacobs Michael and Alan Jose",
       title = "Agent of Change: Print Culture Studies after Elizabeth L. Eisenstein",
   booktitle = "Finite Volumes for Complex Applications VI Problems \& Perspectives",
      editor = "Baron, Sabrina Alcorn and Kucharik, M.",
      series = "U of Massachusetts P / Library of Congress, Center for the Book",
        year = "2007"
}

@book{Milton1998,
	author = "John Milton",
	title = "The Riverside Milton",
	booktitle = "Book title for this book",
	editor = "Roy Flannagan",
	publisher = "Houghton Mifflin",
	year = "1998"
}

@book{Milkis1994,
	author = "Sidney M. Milkis and Michael Nelson",
	title = "The American Presidency: Origins and Development, 1776–1993",
	edition = "2nd",
	publisher = "CQ Press",
	year = "1994"
}

@incollection{Toorn2017,
	author = "Penny van Toorn and Daniel Justice",
	title = "Aboriginal Writing",
	booktitle = "The Cambridge Companion to Canadian Literature",
	editor = "Eva-Marie Kröller",
	publisher = "Cambridge UP",
	year = "2017",
	pages = "26–58"
}

@article{Baron2013,
      author = "Naomi S. Baron",
       title = "Redefining Reading: The Impact of Digital Communication Media",
        year = "Jan. 2013",
     journal = "PMLA",
      volume = "128",
      number = "1",
	  pages = "193–200"
}

@article{Kincaid2001,
	author = "Jamaica Kincaid",
	title = "In History",
	year = "spring 2001",
	journal = "Callaloo",
	volume = "24",
	number = "2",
	pages = "620–26",
}

@article{Kafka2007,
	author = "Ben Kafka and Barbara Adams",
	title = "The Demon of Writing: Paperwork, Public Safety, and the Reign of Terror",
	year = "2007",
	journal = "Representations",
	number = "98",
	pages = "1–24",
}

@article{Belton2008, 
	author = "John Belton and Ken Borsuk",
	title = "Painting by the Numbers: The Digital Intermediate",
	year = "spring 2008",
	journal = "Film Quarterly",
	volume = "61",
	number = "3",
	pages = "58-65"
}

@article{Helmling2006,
	author = "Steven Helmling and John Belton and Ken Borsuk and Barbara Adams",
	title = "A Martyr to Happiness: Why Adorno Matters",
	year = "2006",
	journal = "Kenyon Review",
	volume = "28",
	number = "4",
	pages = "156-72"
}

@article{Quirk2012,
	author = "Tom Quirk",
	title = "The Flawed Greatness of Huckleberry Finn",
	year = "fall 2012",
	journal = "American Literary Realism",
	volume = "45",
	number = "1",
	pages = "38--48",
	note  = "JSTOR, https://doi.org/10.5406/amerlitereal.45.1.0038"
}

@article{Goldman2010,
	author = "Anne Goldman",
	title = "Questions of Transport: Reading Primo Levi Reading Dante",
	year = "spring 2010",
	journal = "The Georgia Review",
	volume = "64",
	number = "1",
	pages = "69--88",
	note = "JSTOR, www.jstor.org/stable/41403188"
}

\bmsection*{Data and Code Availability}
The \texttt{LibNEGF.jl} library developed in this work, along with the benchmarking scripts used to generate the numerical experiments, are openly available in the GitHub repository at \url{https://github.com/SimLabQuantumMaterials/LibNEGF.jl}.

\appendix

\bmsection{Microbenchmarking}\label{app:microbenchmarking}

This section presents a precise quantitative and qualitative description of the execution time of GEMM, $t_{GEMM}$, and the ratios $r_{LU}$ and $r_{GETRS}$, in our \texttt{LibNEGF.jl} Julia implementation. Throughout this appendix, we take all the matrices to be square with dimension $N$; we have used $b_{s}$, instead of $N$, in the cost models in sect.s \ref{sec:LDU_and_RGF} and \ref{sec:DDRGF}. 

\bmsubsection{Measuring the performance of GEMM in Julia\label{app:roofline}}

The installation of Julia that we have used for all the numerical experiments in this work calls OpenBLAS, which in turn uses AVX-512 fused-multiply add instructions. We have used complex-double precision throughout the executions in this work, but our implementation is generic in the data precision. All our numerical experiments run on a single core in a node of the JUWELS Cluster Module, more specifically, on a single core of an Intel Xeon Platinum 8168 (Skylake-SP) node, which has two AVX-512 FMA units per core ($32$ flops/cycle). With these specifications in mind, we are bound by a Theoretical Peak Performance (TPP---assuming a frequency of $f = 2.7$ GHz) of $\pi_{\text{TPP}} = 86.4$ GFLOPs/s. With a DRAM of type DDR4 and 2666 MHz, we have a Peak Memory Bandwidth (PMB) of $\beta = 21.32$ GB/s.\\
We will assess now the performance of the GEMM called from within our implementation. In order to analyze the attainable performance, we need to first quantify the arithmetic intensity. The exact work for a single GEMM, assuming we are multiplying two $N \times N$ complex-double matrices, requires:
$$ w = 8N^2(N-2) \approx 8N^3 \ \text{FLOPs}, $$
as a single complex multiply-add requires 8 real FLOPs; we have used $b_s$, instead of $N$, in the cost models in sect.s \ref{sec:LDU_and_RGF} and \ref{sec:DDRGF}.\\
In the standard GEMM interface ($C = \alpha AB + \beta C$), the terms $q_C, q_B$, and $q_A$ represent the bytes transferred from main memory for matrices $C$, $B$, and $A$, respectively. The explanation of the hierarchical blocking strategy in OpenBLAS \cite{goto2008anatomy} is out of the scope of this work, but in short, it partitions the matrices into sub-blocks that fit perfectly into L1, L2, and L3 caches to maximize data reuse and minimize redundant fetches from RAM. Assuming a simplified model where data is streamed directly without reuse, we approximate the memory traffic as:
$$ q = q_C + q_B + q_A \approx (32 + 16 + 16) N^2 \ \text{Bytes} = 64 N^2 \ \text{Bytes}. $$
\noindent We note that a more accurate expression of $q$, for GEMM as done in OpenBLAS, would imply an analysis of the hierarchical approach where various loads are done from DRAM and throughout the cache hierarchy. With the expression of $q$ above, the Theoretical Arithmetic Intensity takes the form:
$$ a = \frac{w}{q} \approx \frac{8N^3 \ \text{FLOPs}}{64N^2 \ \text{Bytes}} = \frac{N}{8} \ \text{FLOPs/Byte}. $$
\noindent We write then the attainable performance as:
$$ \pi_{att.}(a) = \min\{\pi_{\text{TPP}}, a\beta\}, $$
where the $\min$ function mathematically enforces the roofline model: a kernel's performance is strictly capped by either the memory bandwidth limit ($a\beta$) or the hardware's compute ceiling ($\pi_{\text{TPP}}$). The measured performance is explicitly calculated using the exact operation count:
$$ \pi_{meas.} = \frac{8N^2(N-2) \ \text{FLOPs}}{t_{exec.}}, $$
with $t_{exec.}$ being the execution time.\\
We show, in fig.\ \ref{fig:roofline_assessment_GEMM}, how the measured and the attainable performance behave for GEMM in our implementation.
\begin{figure}[h]
    \centering
        \begin{tikzpicture}[scale=0.8]
        \begin{axis}[
            xlabel={Arithmetic Intensity (FLOPs/Byte)},
            ylabel={Performance (GFLOPs/s)},
            xmode=log,
            ymode=log,
            xmin=0.5, xmax=150,
            ymin=8, ymax=150,
            grid=both,
            major grid style={line width=.2pt,draw=gray!50},
            minor grid style={line width=.1pt,draw=gray!20},
            width=13cm,
            height=8cm,
            legend pos=south east,
            xtick={1, 2, 4, 8, 16, 32, 64, 128},
            xticklabels={
                1 ($N$=8), 
                2 ($N$=16), 
                4 ($N$=32), 
                8 ($N$=64), 
                16 ($N$=128), 
                32 ($N$=256), 
                64 ($N$=512),
                128 ($N$=1024)
            },
            xticklabel style={font=\small, rotate=45, anchor=north east},
            log ticks with fixed point,
        ]
        
        \addplot[
            color=black,
            very thick,
            domain=0.5:150,
            samples=200
        ] {min(86.4, 21.32 * x)};
        \addlegendentry{Attainable Performance}
        
        \addplot[
            color=blue,
            thick,
            dashed,
            domain=0.5:8
        ] {21.32 * x};
        \addlegendentry{Memory Bandwidth Limit}
        
        \addplot[
            color=red,
            thick,
            dashed,
            domain=2:100
        ] {86.4};
        \addlegendentry{Peak Compute ($86.4$ GFLOPs/s)}
        
        \node[
            label={270:{($4.05$, $86.4$)}}, 
            circle, 
            fill=black, 
            inner sep=1.5pt
        ] at (axis cs:4.0525,86.4) {};
        
        \addplot[
            only marks,
            mark=x,
            mark size=4pt,
            very thick,
            color=purple,
            error bars/.cd,
            y dir=both,      
            y explicit,      
            error mark=-     
        ] coordinates {(128, 83.4) +- (0, 0.71)};
        \addlegendentry{Measured Performance}
        
        \node[
            label={[purple]270:{\small ($83.4$)}}, 
            inner sep=1.5pt
        ] at (axis cs:128,83.4) {};

        \addplot[
            only marks,
            mark=x,
            mark size=4pt,
            very thick,
            color=purple,
            error bars/.cd,
            y dir=both,      
            y explicit,      
            error mark=-     
        ] coordinates {(64, 78.64) +- (0, 1.26)};
        \addlegendentry{Measured Performance}
        
        \node[
            label={[purple]270:{\small ($78.64$)}}, 
            inner sep=1.5pt
        ] at (axis cs:64,78.64) {};

        \addplot[
            only marks,
            mark=x,
            mark size=4pt,
            very thick,
            color=purple,
            error bars/.cd,
            y dir=both,      
            y explicit,      
            error mark=-     
        ] coordinates {(32, 67.24) +- (0, 1.0)};
        
        \node[
            label={[purple]270:{\small ($67.24$)}}, 
            inner sep=1.5pt
        ] at (axis cs:32,67.24) {};

        \addplot[
            only marks,
            mark=x,
            mark size=4pt,
            very thick,
            color=purple,
            error bars/.cd,
            y dir=both,      
            y explicit,      
            error mark=-     
        ] coordinates {(16, 60.66) +- (0, 1.39)};
        
        \node[
            label={[purple]270:{\small ($60.66$)}}, 
            inner sep=1.5pt
        ] at (axis cs:16,60.66) {};

        \addplot[
            only marks,
            mark=x,
            mark size=4pt,
            very thick,
            color=purple,
            error bars/.cd,
            y dir=both,      
            y explicit,      
            error mark=-     
        ] coordinates {(8, 42.1) +- (0, 2.22)};
        
        \node[
            label={[purple]270:{\small ($42.1$)}}, 
            inner sep=1.5pt
        ] at (axis cs:8,42.1) {};

        \addplot[
            only marks,
            mark=x,
            mark size=4pt,
            very thick,
            color=purple,
            error bars/.cd,
            y dir=both,      
            y explicit,      
            error mark=-     
        ] coordinates {(4, 28.0) +- (0, 3.4)};
        
        \node[
            label={[purple]270:{\small ($28.0$)}}, 
            inner sep=1.5pt
        ] at (axis cs:4,28.0) {};

        \addplot[
            only marks,
            mark=x,
            mark size=4pt,
            very thick,
            color=purple,
            error bars/.cd,
            y dir=both,      
            y explicit,      
            error mark=-     
        ] coordinates {(2, 14.67) +- (0, 5.34)};
        
        \node[
            label={[purple]270:{\small ($14.67$)}}, 
            inner sep=1.5pt
        ] at (axis cs:2,14.67) {};

        \end{axis}

        \end{tikzpicture}
    \caption{Roofline model assessment of GEMM in our \texttt{LibNEGF.jl} Julia implementation. The cross-marked points are measured performance (i.e., $\pi_{meas.}$) values; error bars have been included on top of each cross. The sample size used for these runs is 1000 for all cases except $N=1024$, with 100 for the latter.}
    \label{fig:roofline_assessment_GEMM}
\end{figure}
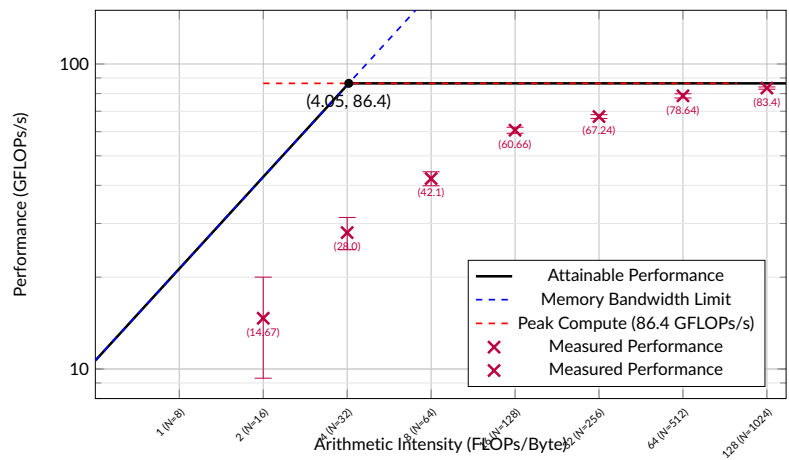
The memory-bound region could be made tighter towards the measured performance if we wrote a more precise expression for the arithmetic intensity $a(N)$. In particular, note that a more faithful expression for the memory traffic would give larger values in $q(N)$, leading to smaller values in $a(N)$ and then a tighter upper bound in the memory-bound region. With fig.\ \ref{fig:roofline_assessment_GEMM} we show that our Julia implementation allows us reaching the upper limits of performance from the underlying hardware.

\bmsubsection{Quantifying the $r_{LU}$ and $r_{GETRS}$ ratios\label{app:timings}}
Next to GEMM, all of the recursive methods described in this work rely on two other fundamental BLAS operations: LU and GETRS. We now run benchmarks on these three kernels to expose the dependence of their execution times ratios, $r_{LU}$ and $r_{GETRS}$, on the matrix dimension. These ratios are needed to have a complete description of the cost models of both RGF and DDRGF. The hardware on which we run these benchmarks is again a single core within a node of the JUWELS Cluster Module, with the code implemented via our \texttt{LibNEGF.jl} Julia library.\\
The results presented in fig.\ \ref{fig:r_ratios_vs_N} were produced with sample sizes of 1000 for $N \le 512$, 100 for $N = 1024$, and 10 for $N \ge 2048$; all of the used matrices contained random entries. The standard BLAS algorithms implemented in OpenBLAS perform strictly $O(N^3)$ operations. Therefore, the steep empirical scaling observed in the left panel of Figure A2 for smaller $N$ is not a change in algorithmic complexity, but rather a transition in hardware performance. As $N$ grows, the computation moves from a memory-bound regime into a compute-bound regime where AVX-512 vectorization and cache hierarchies are fully utilized.\\
To understand the asymptotic behavior of the execution time ratios $r_{LU}$ and $r_{GETRS}$, we must explicitly compare their theoretical FLOP complexities. For complex-double data:
\begin{itemize}
    \item \textbf{GEMM} ($C=AB$) requires $\approx 8N^3$ FLOPs.
    \item \textbf{LU Factorization} of an $N \times N$ matrix requires $\approx \frac{8}{3}N^3$ FLOPs.
    \item \textbf{GETRS} (solving $LU X = B$ where $B$ is an $N \times N$ matrix) requires forward and backward substitutions for $N$ right-hand sides, totaling $\approx 8N^3$ FLOPs.
\end{itemize}
\noindent In highly optimized BLAS/LAPACK libraries, block-based versions of LU and GETRS are implemented specifically to cast the bulk of their computations as Level-3 BLAS GEMM calls. Because GEMM acts as the underlying workhorse, the computational efficiency (GFLOPs/s) of LU and GETRS asymptotically approaches the maximal efficiency of GEMM for very large $N$. Consequently, their execution time ratios naturally tend toward the exact ratios of their theoretical FLOP complexities:
$$ r_{LU} \to \frac{\frac{8}{3}N^3}{8N^3} = \frac{1}{3} \approx 0.33, \quad r_{GETRS} \to \frac{8N^3}{8N^3} = 1.0. $$
\noindent In practice, due to the un-blockable overhead of pivoting, panel factorizations, and the difficulty of perfectly saturating the AVX-512 FMA units during triangular solves, the empirical limits are slightly higher. As shown in the right panel of fig.\ \ref{fig:r_ratios_vs_N}, $r_{LU}(N)$ tends to a value between $0.38$ and $0.42$, while $r_{GETRS}(N)$ tends to approximately $1.05$, perfectly aligning with these theoretical and architectural expectations.

\begin{figure}[h]
\centering
\begin{tikzpicture}[scale=0.95]
\begin{axis}[
    title={},
    xlabel={$N$},
    ylabel={time (microseconds)},
    xmin=10, xmax=700,
    ymin=0.3, ymax=20000, 
    ymode = log, xmode = log,
    xtick={16, 32, 64, 128, 256, 512},
    xticklabels={16, 32, 64, 128, 256, 512},
    ytick={2.1, 8.7, 46.4, 257.6, 1859.0, 12720.0},
    yticklabels={2.1, 8.7, 46.4, 257.6, 1859.0, 12720.0},
    legend pos=north west,
    ymajorgrids=true,
    grid style=dashed,
]
\addplot[
    color=blue,
    mark=square,
    error bars/.cd,
    y dir=both,
    y explicit
] coordinates {
    (16,2.08) +- (0,0.76)
    (32,8.72) +- (0,1.06)
    (64,46.417) +- (0,2.44)
    (128,257.5833) +- (0,5.93)
    (256,1859.0) +- (0,27.6)
    (512,12720.0) +- (0,0.2079)
};
\addplot[
    color=red,
    mark=none,
    dashed,
    thick
] coordinates {
    (16,0.388)
    (512,12720.0)
};
\end{axis}
\end{tikzpicture}
\begin{tikzpicture}[scale=0.95]
\begin{axis}[
    title={},
    xlabel={$N$},
    ylabel={ratio},
    xmin=10, xmax=4500,
    ymin=0.25, ymax=13,
    ymode = log, xmode = log,
    xtick={16, 32, 64, 128, 256, 512, 1024, 2048, 4096},
    xticklabels={16, 32, 64, 128, 256, 512, 1024, 2048, 4096},
    ytick={0.38, 1.0, 2.8, 3.13, 10.0},
    yticklabels={0.38, 1.0, 2.8, 3.13, 10.0},
    legend pos=north west,
    ymajorgrids=true,
    grid style=dashed,
]
{\addplot[
    color=green,
    mark=*,
    error bars/.cd,
    y dir=both,
    y explicit
    ] coordinates {
    (16,3.13)  +- (0,1.28)
    (32,3.10)  +- (0,0.42)
    (64,2.95)  +- (0,0.31)
    (128,2.69) +- (0,0.14)
    (256,1.97) +- (0,0.04)
    (512,1.56) +- (0,0.024)
    (1024,1.3) +- (0,0.015)
    (2048,1.09) +- (0,0.04)
    (4096,1.056) +- (0,0.01)
    };\addlegendentry{$r_{GETRS}$}}
{\addplot[
    color=blue,
    mark=square,
    error bars/.cd,
    y dir=both,
    y explicit
    ] coordinates {
    (16,2.8)  +- (0,0.87)
    (32,1.78)  +- (0,0.17)
    (64,1.27)  +- (0,0.084)
    (128,1.014) +- (0,0.027)
    (256,0.75) +- (0,0.014)
    (512,0.59) +- (0,0.014)
    (1024,0.474) +- (0,0.0074)
    (2048,0.42) +- (0,0.012)
    (4096,0.405) +- (0,0.02)
    };\addlegendentry{$r_{LU}$}}
\end{axis}
\end{tikzpicture}
\caption{Execution time of GEMM (left panel) and ratios $r_{LU}$ and $r_{GETRS}$ (right panel), as a function of the dimension $N$, on a single core in a node of the JUWELS Cluster Module. We have added axes labels with the original non-log values. Each point in either panel corresponds to a sample size of 1000 for $N \le 512$, 100 for $N = 1024$, and 10 for $N \ge 2048$; error bars have been added on top of each data point. We have added the line $t_{GEMM}(N) = f_{GEMM}N^{3}$ as a dashed straight line on the left panel, passing by (512,12720.0).}\label{fig:r_ratios_vs_N}
\end{figure}
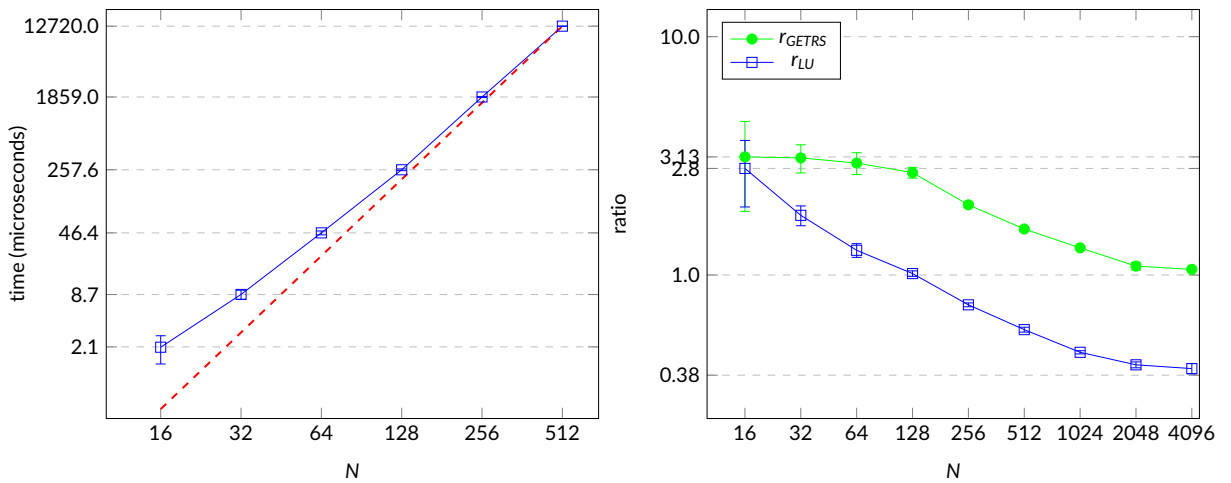

{\color{blue}

\nocite{*}

\end{document}